%% file: arxiv.tex
\documentclass[english,11pt,a4paper]{article}

\usepackage{a4wide}
\usepackage{amsmath,amsthm,epsfig,amssymb,amsbsy}
\usepackage{enumerate}
\usepackage{comment}
\usepackage{algorithm}
\usepackage{algorithmic}
\usepackage{amsfonts}       
\usepackage{dsfont}
\usepackage{enumitem}
\usepackage{mathtools}

\usepackage{pgfplots}
\usepackage{subfig}

\newenvironment{keywords}{\begin{paragraph}{Keywords:}
}
{
\end{paragraph}
}
    
\newenvironment{subclass}{\begin{paragraph}{AMS Subject Classification:}
}
{\end{paragraph}
}

\DeclareMathOperator*{\argmin}{arg\,min}
\DeclareMathOperator*{\epsargmin}{{\varepsilon}\text{--}arg\,min}
\DeclareMathOperator*{\argmax}{arg\,max}

\DeclareMathOperator{\infconv}{\mathbin{\square}}
\DeclareMathOperator{\infdeconv}{\boxminus}

\DeclareMathOperator{\gph}{gph}

\DeclareMathOperator{\SumExp}{SumExp}
\DeclareMathOperator{\Exp}{Exp}
\DeclareMathOperator{\Log}{Log}
\DeclareMathOperator{\dom}{dom}
\DeclareMathOperator{\diag}{Diag}

\DeclareMathOperator{\intr}{int}
\DeclareMathOperator{\bdry}{bdry}

\DeclareMathOperator{\cl}{cl}

\DeclareMathOperator{\ran}{rge}

\DeclareMathOperator{\sign}{sign}
\DeclareMathOperator{\relint}{ri}

\DeclareMathOperator{\id}{id}

\DeclareMathOperator{\logsumexp}{LogSumExp}
\DeclareMathOperator{\vecmax}{vecmax}
\DeclareMathOperator{\Li}{Li}

\newcommand{\bR}{\mathbb{R}}

\newcommand{\bN}{\mathbb{N}}

\newcommand{\exR}{\overline{\mathbb{R}}}

\newcommand{\cC}{\mathcal{C}}

\makeatletter
\newcommand{\aprox}[3][\@nil]{%
  \def\tmp{#1}%
   \ifx\tmp\@nnil
       \operatorname{aprox}_{#3}^{#2}
    \else
         \operatorname{aprox}_{#3}^{#1 \star #2}
    \fi}

\newcommand{\bprox}[3][\@nil]{%
  \def\tmp{#1}%
   \ifx\tmp\@nnil
       \operatorname{bprox}_{#3}^{#2}
    \else
         \operatorname{bprox}_{#3}^{#1 #2}
    \fi}

\makeatother

\newcommand{\gap}[2]{\mathcal{G}_{#2}^{#1}}

\def\ifsvjour{false}
\def\true{true}

\usepackage{hyperref}
\hypersetup{
    colorlinks=true,
    linkcolor=blue,
    filecolor=magenta,
    citecolor =magenta,  
    urlcolor=magenta,
    pdftitle={Anisotropic Proximal Gradient}
    }
\usepackage[capitalize]{cleveref}[0.19]

\crefname{section}{section}{sections}
\crefname{subsection}{subsection}{subsections}
\Crefname{section}{Section}{Sections}
\Crefname{subsection}{Subsection}{Subsections}

\Crefname{figure}{Figure}{Figures}

\crefformat{equation}{\textup{#2(#1)#3}}
\crefrangeformat{equation}{\textup{#3(#1)#4--#5(#2)#6}}
\crefmultiformat{equation}{\textup{#2(#1)#3}}{ and \textup{#2(#1)#3}}
{, \textup{#2(#1)#3}}{, and \textup{#2(#1)#3}}
\crefrangemultiformat{equation}{\textup{#3(#1)#4--#5(#2)#6}}%
{ and \textup{#3(#1)#4--#5(#2)#6}}{, \textup{#3(#1)#4--#5(#2)#6}}{, and \textup{#3(#1)#4--#5(#2)#6}}

\Crefformat{equation}{#2Equation~\textup{(#1)}#3}
\Crefrangeformat{equation}{Equations~\textup{#3(#1)#4--#5(#2)#6}}
\Crefmultiformat{equation}{Equations~\textup{#2(#1)#3}}{ and \textup{#2(#1)#3}}
{, \textup{#2(#1)#3}}{, and \textup{#2(#1)#3}}
\Crefrangemultiformat{equation}{Equations~\textup{#3(#1)#4--#5(#2)#6}}%
{ and \textup{#3(#1)#4--#5(#2)#6}}{, \textup{#3(#1)#4--#5(#2)#6}}{, and \textup{#3(#1)#4--#5(#2)#6}}

\usepackage{crossreftools}
\newtheorem{theorem}{Theorem}[section]
\newtheorem{corollary}[theorem]{Corollary}
\newtheorem{lemma}[theorem]{Lemma}
\newlist{lemenum}{enumerate}{1} 
\setlist[lemenum]{label=(\roman*), ref=\theproposition(\roman*), font=\rm}
\crefalias{lemenumi}{lemma} 

\newtheorem{proposition}[theorem]{Proposition}
\newlist{propenum}{enumerate}{1} 
\setlist[propenum]{label=(\roman*), ref=\theproposition(\roman*), font=\rm}
\crefalias{propenumi}{proposition} 

\newtheorem{definition}[theorem]{Definition}

\theoremstyle{remark}
\newtheorem{remark}[theorem]{Remark}
\newtheorem{example}[theorem]{Example}

\newlist{assumenum}{enumerate}{1} 
\setlist[assumenum]{leftmargin=2.1cm,label=(A\arabic*),font=\bfseries}
\creflabelformat{assumenumi}{#2#1#3}
\crefname{assumenumi}{assumption}{assumptions}
\Crefname{assumenumi}{Assumption}{Assumptions}

\numberwithin{equation}{section}

\title{Anisotropic Proximal Gradient}
\author{Emanuel Laude\thanks{	KU Leuven,
		Department of Electrical Engineering (ESAT-STADIUS),
		Kasteelpark Arenberg 10, 3001 Leuven, Belgium~
		{\tt%
			\href{mailto:emanuel.laude@esat.kuleuven.be}{\{emanuel.laude,}%
			\href{mailto:panos.patrinos@esat.kuleuven.be}{panos.patrinos\}}%
			\href{mailto:emanuel.laude@esat.kuleuven.be,panos.patrinos@esat.kuleuven.be}{@esat.kuleuven.be}%
		}
	} \and Panagiotis Patrinos\footnotemark[1]}

\begin{document}

\maketitle
\begin{abstract}
\input{abstract.tex}
\end{abstract}
\begin{keywords}
Bregman distance $\cdot$ proximal gradient method $\cdot$ duality $\cdot$ nonlinear preconditioning
\end{keywords}
\begin{subclass}
65K05 $\cdot$ 49J52 $\cdot$ 90C30
\end{subclass}

\input{content.tex}

\appendix

\section{Missing proofs}
\input{missing_proofs.tex}

\ifx\ifsvjour\true
\else
\input{missing_proofs_phi_dca.tex}
\fi

\bibliography{references}
\bibliographystyle{abbrv.bst}

\end{document}

%% file: abstract.tex
This paper studies a novel algorithm for nonconvex composite minimization which can be interpreted in terms of dual space nonlinear preconditioning for the classical proximal gradient method. The proposed scheme can be applied to additive composite minimization problems whose smooth part exhibits an anisotropic descent inequality relative to a reference function. It is proved that the anisotropic descent property is closed under pointwise average if the Bregman distance generated by the conjugate reference function is jointly convex. More specifically, for the exponential reference function we prove its closedness under pointwise conic combinations. We analyze the method's asymptotic convergence and prove its linear convergence under an anisotropic proximal gradient dominance condition. Applications are discussed including exponentially regularized LPs and logistic regression with nonsmooth regularization. In numerical experiments we show significant improvements of the proposed method over its Euclidean counterparts.

%% file: content.tex
\section{Introduction}
\subsection{Motivation and contributions}
A classical first-order approach for the minimization of an additive composite problem is the celebrated proximal gradient algorithm. In particular, if the gradient mapping of the smooth part of the cost function is Lipschitz continuous the algorithm converges with a constant step-size which avoids a possibly expensive, in terms of function evaluations, linesearch procedure.
However, there are many functions which do not exhibit a globally Lipschitz continuous gradient mapping and thus proximal gradient does not converge with a constant step-size. For example this is the case for problems involving exponential penalties.
It is well known that the proximal gradient method can be written in terms of a majorize-minimize procedure with a quadratic upper bound. However, in some cases more structure of the cost function is available. This offers the opportunity to construct problem specific and potentially tighter upper approximations that perhaps lead to faster convergence yet preserving the simplicity of a first-order method.
In this paper we study a novel, nonlinear extension of the proximal gradient method, called anisotropic proximal gradient method. In the same way that the Euclidean proximal gradient method at each step minimizes a simple quadratic model the proposed method minimizes a certain problem specific upper bound obtained from an \emph{anisotropic} generalization of the celebrated descent lemma (anisotropic smoothness) \cite{laude2021conjugate,laude2021lower}. This is conceptually similar to the notion of relative smoothness and the Bregman proximal gradient method \cite{birnbaum2011distributed,bauschke2017descent,lu2018relatively}. In fact, in the convex non-composite case the proposed method can be seen as a certain dual counterpart thereof where the roles of reference function and cost function are flipped \cite{maddison2021dual}.
The approach applies but is not limited to entropically regularized LPs, logistic regression and soft maximum-type problems covering problems that are not handled with the Bregman proximal gradient method. In experiments we show promising improvements over Euclidean methods including a recent family of linesearch-free adaptive proximal gradient methods \cite{latafat2023convergence,malitsky2020adaptive}.
More precisely the contributions of this work are summerized as follows:
\begin{itemize}
\item We introduce the notion of anisotropic smoothness, an extension of \cite{laude2021conjugate,laude2021lower}, that is characterized in terms of an inequality akin to (and in fact a generalization of) the well-known Euclidean descent inequality. In particular, we develop a calculus for anisotropic smoothness providing equivalent characterisations and invariances that allow one to assemble more complicated functions from elementary ones. For example in the exponential case we show that the property is closed under pointwise conic combinations which distinguishes it from the notion of Euclidean Lipschitz smoothness and the notion of relative smoothness.
\item Based on the anisotropic descent inequality we propose a novel algorithm that can be seen as a nonlinearly preconditioned version of the Euclidean proximal gradient method. For improved practicality we also consider a linesearch extension. We show that the method converges subsequentially using an unconventional regularized gap function as a measure of stationarity. In particular, this guarantees a sublinear rate wrt. the stationarity gap.
\item We prove the $Q$-linear convergence of the anisotropic proximal gradient method under a non-Euclidean generalization of the proximal gradient dominated condition \cite{karimi2016linear}. This is closely related to the PL-inequality. We prove that the anisotropic proximal gradient dominated condition is implied by anisotropic strong convexity of the smooth part. We show that relative smoothness \cite{bauschke2017descent,lu2018relatively} is a dual characterization of anisotropic strong convexity generalizing \cite{laude2021conjugate} to reference functions $\phi$ which are not necessarily super-coercive. Also note that for the Bregman proximal gradient method only an $R$-linear convergence result is known \cite{lu2018relatively}.
\ifx\ifsvjour\true
\else
\item We also provide an equivalent interpretation in terms of a difference of $\Phi$-convex approach. In particular, this allows us to employ a transfer of smoothness via a certain double-min duality to show linear convergence if the composite part is strongly convex rather than the smooth part. 
\fi
\item In numerical experiments we show that the algorithm improves over its Euclidean counterparts. On the task of regularized logistic regression we show that for small regularization weight the algorithm achieves a better performance than a recent family of linesearch-free adaptive proximal gradient methods \cite{latafat2023convergence,malitsky2020adaptive}.
\end{itemize}
\subsection{Related work}
A popular nonlinear extension of the classical proximal gradient method is the Bregman proximal gradient algorithm \cite{birnbaum2011distributed,bauschke2017descent,lu2018relatively}. Similar to our approach the algorithm applies to composite problems whose smooth part exhibits relative smoothness (in the Bregman sense), a generalization of Lipschitz smoothness introduced by the same authors. For instance this allows one to consider functions whose Hessians exhibit a certain polynomial growth \cite{lu2018relatively}. Yet, there exist examples, such as problems with exponential penalties, for which it is hard to construct ``simple'' reference functions. Simplicity of the reference function is this context refers to an easy to invert gradient mapping and simple proximal mappings both of which are crucial for tractability of the algorithm.
As a remedy we introduce the notion of anisotropic smoothness as an alternative generalization of Lipschitz smoothness. The notion of anisotropic smoothness considered in this paper is an extension of \cite{laude2021conjugate} relaxing the super-coercivity assumption of the reference function. Up to sign-flip, anisotropic smoothness as introduced in \cite{laude2021conjugate} can also be regarded as a globalization of anisotropic prox-regularity \cite{laude2021lower} specialized to smooth functions. 
Alongside relative smoothness \cite{lu2018relatively,bauschke2019linear} also propose the notion of relative strong convexity which allows the authors to study the linear convergence of the Bregman (proximal) gradient method. In particular, \cite{lu2018relatively} pose the open question of the existence of a meaningful conjugate duality for relative smoothness and strong convexity akin to the Euclidean case. Under full domain and super-coercivity of the reference function this is addressed in \cite{laude2021conjugate} where it is shown that such a duality correspondence (in the convex case) holds between relative smoothness and anisotropic strong convexity as well as relative strong convexity and anisotropic smoothness. However, the super-coercivity restriction limits the practical applicability of the framework. In this paper we further generalize these results to reference functions which are not necessarily super-coercive which allows us to include the important exponential reference function. 
Expanding on \cite{laude2021conjugate} we show that anisotropic smoothness is equivalent to the existence of an infimal convolution expression, or equivalently, to a certain notion of generalized concavity \cite{Vil08}, which is an important concept in optimal transport \cite{Vil08}.
In general, however, the pointwise sum of infimal convolutions is not an infimal convolution \cite{laude2021conjugate} limiting the practical applicability. Averages of infimal convolutions also appear in the context of the Bregman proximal average \cite{wang2021bregman}. In particular the authors study the preservation of its convexity for reference functions whose conjugates generate a jointly convex Bregman distance. Building upon these results we show that this property furnishes a sufficient condition for closedness of anisotropic smoothness under pointwise average even in the nonconvex case which is important for practice. Furthermore, for the exponential reference function, we show that anisotropic smoothness is closed under pointwise conic combinations--a property shared with the notion of \emph{quasi self-concordance} \cite{doikov2023minimizing,bach2010self}. In this case, a second-order characterization of anisotropic smoothness reveals a close relation to the recent notion of $(L_0, L_1)$-smoothness \cite{Zhang2020Why}.

\subsection{Paper organization}
The remainder of this paper is organized as follows:

In \Cref{sec:notation} we introduce the notation and some required basic concepts such as Bregman distances.

In \Cref{sec:aproxgrad_method} we introduce the notion of anisotropic smoothness, propose the algorithm and discuss the well-definedness of its iterates under certain constraint qualifications.

In \Cref{sec:calculus_asmooth} we develop a calculus for anisotropic smoothness and provide examples.

In \Cref{sec:aproxgrad_analysis} we analyse the algorithm.

\ifx\ifsvjour\true
\else
In \Cref{sec:phi_dca_and_phi_cvxt} we provide an equivalent interpretation in terms of a difference of $\Phi$-convex approach. 
\fi

In \Cref{sec:apps} we consider applications based on the examples developed in \Cref{sec:calculus_asmooth}.

\section{Notation and preliminaries} \label{sec:notation}
We denote by $\langle \cdot, \cdot\rangle$ the standard Euclidean inner product on $\bR^n$ and by $\|x\|:=\sqrt{\langle x, x \rangle}$ for any $x\in \bR^n$ the standard Euclidean norm on $\bR^n$.
In accordance with \cite{moreau1966fonctionnelles,RoWe98} we extend the classical arithmetic on $\bR$ to the extended real line $\exR:=\bR\cup\{-\infty, +\infty\}$. We define upper addition $-\infty \mathbin{\dot{+}} \infty = \infty$ and lower addition $-\infty \mathbin{\text{\d{\ensuremath{+}}}} \infty = -\infty$, and accordingly upper subtraction $\infty \mathbin{\dot{-}} \infty = \infty$ and lower subtraction $\infty \mathbin{\text{\d{\ensuremath{-}}}} \infty = -\infty$. The effective domain of an extended real-valued function $f : \bR^n \to \exR$ is denoted by $\dom f:=\{x\in\bR^n : f(x)<\infty\}$, and we say that $f$ is proper if $\dom f\neq\emptyset$ and $f(x) > -\infty$ for all $x \in \bR^n$; lower semicontinuous (lsc) if $f(\bar x)\leq\liminf_{x\to\bar x}f(x)$ for all $\bar x\in\bR^n$; super-coercive if $f(x)/\|x\|\to\infty$ as $\|x\|\to\infty$. We define by $\Gamma_0(\bR^n)$ the class of all proper, lsc convex functions $f:\bR^n \to \exR$.
For any functions $f :\bR^n \to \exR$ and $g:\bR^n \to \exR$ we define the infimal convolution or epi-addition of $f$ and $g$ as $(f \infconv g)(x) = \inf_{y \in \bR^n} g(x-y) + f(y)$.
For a proper function $f :\bR^n \to \exR$ and $\lambda \geq 0$ we define the epi-scaling $(\lambda \star f)(x) = \lambda f(\lambda^{-1} x)$ for $\lambda > 0$ and $(\lambda \star f)(x)=\delta_{\{0\}}(x)$ otherwise. We adopt the operator precedence for epi-multiplication and epi-addition from the pointwise case. We denote by $\bR^n_+=[0, +\infty)^n$ the nonnegative orthant and by $\bR^n_{++}=(0, +\infty)^n$ the positive orthant.
Let $g_-:=g(-(\cdot))$ denote the reflection of $g$. Since convex conjugation and reflection commute we adopt the notation: $(g_-)^* = (g^*)_- =:g^*_-$. We adopt the notions of essential smoothness, essential strict convexity and Legendre functions from \cite[Section 26]{Roc70}: We say that a function $f \in \Gamma_0(\bR^n)$ is \emph{essentially smooth}, if $\intr(\dom f) \neq \emptyset$ and $f$ is differentiable on $\intr(\dom f)$ such that $\|\nabla f(x^\nu)\|\to \infty$, whenever $\intr(\dom f) \ni x^\nu \to x \in \bdry\dom f$, and \emph{essentially strictly convex}, if $f$ is strictly convex on every convex subset of $\dom \partial f$, and \emph{Legendre}, if $f$ is both essentially smooth and essentially strictly convex.
We denote by $\bN_0:=\bN \cup\{0\}$.
Otherwise we adopt the notation from \cite{RoWe98}. 

Our concepts and algorithm involve a reference function $\phi$ which complies with the following standing requirement which is assumed valid across the entire paper unless stated otherwise:
\begin{assumenum}
\item \label{assum:a1} $\phi \in \Gamma_0(\bR^n)$ is Legendre 
with $\dom \phi = \bR^n$.
\end{assumenum}
The dual reference function $\phi^*$ is Legendre as well \cite[Theorem 26.3]{Roc70} but does not necessarily have full domain. Instead, thanks to a conjugate duality between super-coercivity and full domain \cite[Proposition 2.16]{bauschke1997legendre}, $\phi^*$ is super-coercive.
The gradient mapping $\nabla \phi : \bR^n \to \intr \dom \phi^*$ is a diffeomorphism between $\bR^n$ and $\intr \dom \phi^*$ with inverse $(\nabla \phi)^{-1} = \nabla \phi^*$ \cite[Theorem 26.5]{Roc70}.
In fact, super-coercivity is a standard assumption for the mirror-descent and Bregman (proximal) gradient algorithm as it ensures well-definedness of the mirror update \cite[Lemma 2(ii)]{bauschke2017descent}. Since the framework that will be developed in this paper corresponds to a certain dual Bregman approach the full-domain restriction is somewhat natural.

The running example considered in this paper is the exponential reference function:
\begin{example}
\ifx\ifsvjour\true
 \label[example]{ex:reference_function}
\else
 \label{ex:reference_function}
\fi
Let $\phi :\bR^n \to \bR$ defined by $\phi(x)=\SumExp(x) :=\sum_{j=1}^n \exp(x_j)$. Then $\phi \in \Gamma_0(\bR^n)$ is Legendre with full domain and thus complies with our standing assumption. The convex conjugate $\phi^*$ is the von Neumann entropy defined by
$$
\phi^*(x) = H(x):= \begin{cases} \sum_{j=1}^n x_j \ln(x_j) - x_j & \text{if $x \in \bR_{+}^n$} \\
+\infty & \text{otherwise,}
\end{cases}
$$
where $0\ln(0):= 0$.

For the gradients we have $\nabla \phi(x) =\Exp(x):=(\exp(x_j))_{j=1}^n$ and $\nabla \phi^*(x^*) =\Log(x^*):=(\ln(x^*_j))_{j=1}^n$.
\end{example}

In the course of the paper we often consider the Bregman distance $D_{\phi^*}$ generated by the dual reference function $\phi^*$:
\begin{definition}[Bregman distance]
We define the Bregman distance generated by $\phi^*$ as:
$$
D_{\phi^*}(x, y) := \begin{cases} \phi^*(x) - \phi^*(y) - \langle \nabla \phi^*(y), x-y \rangle & \text{if $x \in \dom \phi^*, y \in \intr \dom \phi^*$} \\
+\infty & \text{otherwise.}
\end{cases}
$$
\end{definition}
Thanks to \cite[Theorem 3.7(iv)]{bauschke1997legendre} we have that $D_{\phi^*}(x, y) = 0$ if and only if $x=y$.

For $\phi=\SumExp$ being the exponential reference function as in \cref{ex:reference_function} the Bregman distance $D_{\phi^*}$ is the classical KL-divergence.
Thanks to \cite[Theorem 3.7(v)]{bauschke1997legendre} have the following key relation between $D_{\phi^*}$ and $D_{\phi}$:
\begin{lemma} \label{thm:dual_bregman}
We have that the identity
$$
D_{\phi^*}(x, y) = D_{\phi}(\nabla \phi^*(y), \nabla \phi^*(x)),
$$
holds true for any $x,y \in \intr \dom \phi^*$.
\end{lemma}

\section{Anisotropic proximal gradient} \label{sec:aproxgrad_method}
\subsection{Anisotropic descent inequality}
The algorithm that is developed in this paper is based on the following anisotropic descent property generalizing \cite[Remark 3.10]{laude2021conjugate} to reference functions $\phi$ which are not necessarily super-coercive.
Up to sign-flip it can also be regarded as a globalization of anisotropic prox-regularity \cite[Definition 2.13]{laude2021lower} specialized to smooth functions.
\begin{definition}[anisotropic descent inequality (anisotropic smoothness)]
\ifx\ifsvjour\true
\label[definition]{def:adescent}
\else
\label{def:adescent}
\fi
Let $f \in \mathcal{C}^1(\bR^n)$ such that the following constraint qualification holds true
\begin{align} \label{eq:cq_adescent}
\ran \nabla f \subseteq \ran \nabla \phi.
\end{align}
Then we say that $f$ satisfies the anisotropic descent property (is anisotropically smooth) relative to $\phi$ with constant $L>0$ if for all $\bar x \in \bR^n$
\begin{equation}\label{eq:adescent}
f(x) \leq f(\bar x) + \tfrac{1}{L} \star \phi(x-\bar x + L^{-1}\nabla\phi^*(\nabla f(\bar x))) - \tfrac{1}{L} \star \phi(L^{-1}\nabla\phi^*(\nabla f(\bar x))) \quad \forall x\in\bR^n.
\end{equation}
If $L=1$ we say that $f$ satisfies the anisotropic descent property relative to $\phi$ without further mention of $L$.
\end{definition}
The following remarks are in order:
\begin{remark}[generalization of Euclidean descent lemma]
\ifx\ifsvjour\true
 \label[remark]{rem:euclidean_descent_lemma}
\else
 \label{rem:euclidean_descent_lemma}
\fi
For $\phi=\tfrac{1}{2}\|\cdot\|^2$ by expanding the square since $\nabla \phi^* = \id$ and $\tfrac1 L \star \phi = \frac{L}{2}\|\cdot\|^2$ the upper bound inequality \cref{eq:adescent} specializes to the classical descent lemma:
\begin{align*}
f(x) &\leq f(\bar x) + \tfrac{L}{2}\|x-\bar x + L^{-1} \nabla f(\bar x)\|^2 - \tfrac L 2 \|L^{-1} \nabla f(\bar x)\|^2 \\
&= f(\bar x) + \langle \nabla f(\bar x), x - \bar x \rangle + \tfrac{L}{2}\|x - \bar x\|^2.
\end{align*}
\end{remark}

\begin{remark}[proximal linearization representation of anisotropic smoothness]
Using the definition of the Bregman distance the descent inequality \cref{eq:adescent} can be rewritten in terms of a proximal linearization as follows:
$$
f(x) \leq f(\bar x) + \langle \nabla f(\bar x), x- \bar x \rangle + D_{L^{-1} \star \phi}(x - \bar y, \bar x - \bar y),
$$
for all $x,\bar x \in \bR^n$ and $\bar y = \bar x - L^{-1}\nabla \phi^*(\nabla f(\bar x))$. 
The inequality is reminiscent of the Bregman descent inequality \cite{birnbaum2011distributed,bauschke2017descent,lu2018relatively} where the arguments in the Bregman distance are translated by the point $\bar y = \bar x - L^{-1}\nabla \phi^*(\nabla f(\bar x))$. As we shall see this guarantees shift-invariance, a property which is in general lacking in the classical Bregman descent inequality.
\end{remark}

\begin{remark}
Since $\nabla (L^{-1} \star \phi)^* = L^{-1}\nabla \phi^*$, saying that $f$ satisfies the anisotropic descent property relative to $\phi$ with constant $L>0$ means equivalently that $f$ satisfies the anisotropic descent property relative to $\frac1L \star \phi$ (with constant $1$).
\end{remark}

The constraint qualification \cref{eq:cq_adescent} ensures that the expression $\nabla\phi^*(\nabla f(\bar x)) \in \bR^n$ is well defined.
If we choose $x:=\bar x$ the descent inequality \cref{eq:adescent} collapses to $f(\bar x) \leq f(\bar x)$ and thus 
$$
x \mapsto f(\bar x) + \tfrac{1}{L} \star \phi(x-\bar x + L^{-1}\nabla\phi^*(\nabla f(\bar x))) - \tfrac{1}{L} \star \phi(L^{-1}\nabla\phi^*(\nabla f(\bar x))),
$$
as a function in $x$ is an upper bound or majorizer of $f$ at $\bar x$.
With some abuse of terminology we often refer to a function that only has anisotropic upper bounds as anisotropically smooth. This differs from \cite[Definition 3.8(iii)]{laude2021conjugate} which assumes the existence of both, upper and lower bounds. For convex functions, however, the lower bound inequality holds automatically in which case the two notions coincide.

The upper bounds obtained from the descent inequality suggest a majorize-minimize procedure akin to the classical proximal gradient scheme that is discussed in the next subsection. Practical examples and a calculus for anisotropic smoothness are developed in \Cref{sec:calculus_asmooth}.

\subsection{Algorithm definition, well-definedness and standing assumptions}
In this paper our goal is to develop a novel algorithm for nonconvex composite minimization problems of the form:
\begin{align} \label{eq:opt_prob} \tag{P}
\text{minimize}~ F := f + g,
\end{align}
where 
\begin{assumenum}[resume*]
\item \label{assum:a2} $f\in \mathcal{C}^1(\bR^n)$ has the anisotropic descent property with constant $L>0$;
\item \label{assum:a3} $g:\bR^n \to \exR$ is proper lsc;  
\item \label{assum:a4} $\inf F > -\infty$.
\end{assumenum}

This suggests the following iterative procedure for minimizing the cost function: In each iteration we minimize the anisotropic upper bound at the current iterate $x^k$ to compute the next iterate $x^{k+1}$:
\begin{align} \label{eq:majorize_minimize}
x^{k+1} &= \argmin_{x \in \bR^n} \left\{f(x^k) + \tfrac{1}{L} \star \phi(x-x^k + L^{-1}\nabla\phi^*(\nabla f(x^k))) - \tfrac{1}{L} \star \phi(L^{-1}\nabla\phi^*(\nabla f(x^k))) + g(x) \right\}.
\end{align}
Choose \(\lambda=1/L\). We will show later that what follows is actually valid for any \(\lambda \in (0, 1/L]\). Introducing a new variable $y^k := x^k - \lambda \nabla \phi^*(\nabla f(x^k))$ we can rewrite the update \cref{eq:majorize_minimize} in terms of a forward-backward like scheme listed in \cref{alg:aproxgrad}.
\begin{algorithm}[H]
\caption{Anisotropic proximal gradient (anisoPG)}
\label{alg:aproxgrad}
\begin{algorithmic}
\REQUIRE Let $\lambda >0$. Let $x^0 \in \bR^n$.
\FORALL{$k=0, 1, \dots$}
   \STATE \begin{align}
	y^k &= x^k - \lambda \nabla \phi^*(\nabla f(x^k)) \label{eq:forward} \\
   	x^{k+1} &\in \argmin_{x \in \bR^n} \left\{g(x) + \lambda \star \phi(x - y^k)\right\}, \label{eq:backward}
	\end{align}
\ENDFOR
\end{algorithmic}
\end{algorithm}
We refer to \cref{eq:forward} as the anisotropic \emph{forward-step} and \cref{eq:backward} as the anisotropic \emph{backward-step}.

The following remarks are in order:
\begin{remark}[generalization of dual space preconditioning and descent direction]
\ifx\ifsvjour\true
 \label[remark]{rem:dual_space_pc}
\else
 \label{rem:dual_space_pc}
\fi
Suppose that $g\equiv 0$. Minimizing the anisotropic upper bound at $x^k$ wrt. $x$ and assuming that $0 \in \ran \nabla \phi$ we obtain by first-order stationarity:
\begin{align}
0 = \nabla \phi(\lambda^{-1}(x^{k+1} - x^k) + \nabla\phi^*(\nabla f(x^k))),
\end{align}
which reads after rearranging
\begin{align}
x^{k+1}= x^k - \lambda \nabla \phi^*(\nabla f(x^k)) + \lambda\nabla \phi^*(0).
\end{align}
By replacing $\phi^*$ with $\phi^*- \langle \nabla \phi^*(0), \cdot \rangle$ we can assume without loss of generality that $\nabla \phi^*(0)=0$ or equivalently that $\phi$ is minimized at $0$ and thus the update simplifies to
\begin{align}
x^{k+1}= x^k - \lambda \nabla \phi^*(\nabla f(x^k)).
\end{align}
For convex $f$ this is \emph{dual space preconditioning} for gradient descent \cite{maddison2021dual}.
Furthermore, note that $-\nabla \phi^*(\nabla f(x))$ is always a descent direction of $f$ at $x$ since by strict monotonicity of $\nabla \phi^*$ and the identity $\nabla \phi^*(0)=0$ one has:
\begin{align*}
\langle \nabla \phi^*(\nabla f(x)), \nabla f(x) \rangle &= \langle \nabla \phi^*(\nabla f(x)) - \nabla \phi^*(0), \nabla f(x) - 0 \rangle \geq 0,
\end{align*}
with equality if and only if $\nabla f(x)=0$.
\end{remark}

\begin{remark}[generalization of scaled proximal gradient descent]
\ifx\ifsvjour\true
 \label[remark]{rem:euclidean_pgd}
\else
 \label[remark]{rem:euclidean_pgd}
\fi
For $\phi \in \Gamma_0(\bR^n)$ being a squared Euclidean norm defined as $\phi(x)=\frac12\|x\|_M^2:=\frac12\langle x, M x\rangle$ with $M$ symmetric positive definite we have $\nabla \phi^*(x)=M^{-1}x$. Thus the algorithm becomes a scaled version of the Euclidean proximal gradient method:
\begin{align}
y^k &= x^k - \lambda M^{-1}\nabla f(x^k) \\
x^{k+1} &= \argmin_{x \in \bR^n} \; \left\{\tfrac{1}{2\lambda}\|x - y^k\|_M^2 + g(x) \right\},
\end{align}
where the backward-step is a scaled Euclidean proximal mapping.
\end{remark}

More generally, the backward-step amounts to the \emph{left} anisotropic proximal mapping of $g$ at $y^k$ which is formally defined as follows:
\begin{definition}[left and right anisotropic proximal mapping and Moreau envelope]
Let $\lambda >0 $ and $x \in \bR^n$. Then the right anisotropic proximal mapping of $g$ with parameter $\lambda$ at $x$ is defined as:
\begin{align}
\aprox[\lambda]{\phi}{g}(x) := \argmin_{y \in \bR^n} \;\{\lambda \star \phi(x - y) + g(y)\},
\end{align}
and the right anisotropic Moreau envelope is the infimal convolution of $g$ and $\lambda \star \phi$:
\begin{align}
(g \infconv \lambda \star \phi)(x) = \inf_{y \in \bR^n}\;\{\lambda \star \phi(x - y) + g(y)\}.
\end{align}
The left anisotropic proximal mapping and Moreau envelope of $g$ with parameter $\lambda$ are defined as
\begin{align}
y \mapsto \argmin_{x \in \bR^n} \;\{\lambda\star \phi(x - y) + g(x)\} \quad \text{and} \quad y \mapsto \inf_{x \in \bR^n} \;\{\lambda\star \phi(x - y) + g(x)\}.
\end{align}
The reflection $\phi_-$ of $\phi$ allows us to express the left anisotropic proximal mapping and Moreau envelope in terms of the right anisotropic proximal mapping and Moreau envelope:
\begin{align}
\argmin_{x \in \bR^n} \;\{\lambda \star \phi(x-y) + g(x)\} &= \argmin_{x \in \bR^n}\;\{ \lambda \star \phi_-(y-x) + g(x) \}= \aprox[\lambda]{\phi_-}{g} (y), \\
\inf_{x \in \bR^n} \;\{\lambda \star \phi(x-y) + g(x) \}&= \inf_{x \in \bR^n}\; \{\lambda \star \phi_-(y-x) + g(x) \}=(g \infconv \lambda \star \phi_-)(y). 
\end{align}
If $\phi(x)=\frac{1}{2}\|x\|^2$ left and right anisotropic proximal mapping and Moreau envelope both coincide with the Euclidean proximal mapping and Moreau envelope.
\end{definition}
For well-definedness of the backward-step we assume that for any input $y$ the minimum in $\argmin_{x \in \bR^n} \;\{\lambda \star \phi(x-y) + g(x)\}$ is attained. This is guaranteed if $g + \lambda \star \phi(\cdot-y)$ is level-bounded for any $y$.

\begin{remark}[threshold of anisotropic prox-boundedness]
\ifx\ifsvjour\true
 \label[remark]{rem:level_boundedness}
\else
 \label{rem:level_boundedness}
\fi
Without loss of generality we can assume that $\phi(0)=0$ since a replacement of $\phi$ with the perturbed reference function $\phi(x)-\phi(0)$ does not affect the proximal mapping. Then, thanks to \cite[Lemma 4.4]{burke2013epi}, $\lambda_1 \star \phi(x-y) \leq \lambda_2 \star \phi(x-y)$ for $\lambda_2 < \lambda_1$ and thus level-boundedness of $g + \lambda \star \phi(\cdot-y)$ for $\lambda=\lambda_0$ implies level-boundedness of $g + \lambda \star \phi(\cdot-y)$ for any $\lambda < \lambda_0$.
This leads us to define the threshold of anisotropic prox-boundedness:
\begin{align}
\lambda_g:= \sup \{\lambda >0 \mid \text{$g + \lambda \star \phi(\cdot-y)$ is level-bounded for any $y\in \bR^n$} \}.
\end{align}
Note that $\lambda_g=+\infty$ if $g$ (or $\phi$) is bounded from below and $\phi$ (or $g$) is coercive.
\end{remark}
We make the following additional assumption on $\lambda_g$ and $\lambda$ to ensure well-definedness of the proximal mapping:
\begin{assumenum}[resume*]
\item \label{assum:a5} Assume that $\lambda_g >0$ and choose $\lambda>0$ such that $\lambda < \lambda_g$ and $\lambda \leq 1/L$.
\end{assumenum}
Next we provide characterizations of the anisotropic proximal mapping and Moreau envelope in the convex case. In particular we provide an expression of the anisotropic proximal mapping in terms of an anisotropic resolvent justifying the term backward-step. In addition we state a non-Euclidean Moreau decomposition that reveals an interesting duality between the anisotropic proximal mapping and the well-established Bregman proximal mapping. The following result is a refinement of \cite[Theorem 3.1(ii)]{combettes2013moreau} and \cite[Proposition 2.16]{wang2021bregman}, originally due to \cite[Theorem 3.2]{Teboulle92}:

\begin{proposition}[Bregman--Moreau decomposition and anisotropic resolvent]
\ifx\ifsvjour\true
 \label[proposition]{thm:moreau_decomposition}
\else
 \label{thm:moreau_decomposition}
\fi
Let $g\in \Gamma_0(\bR^n)$. Assume that $\emptyset \neq \dom g^* \cap \intr \dom \phi^*$. Define by $\bprox[\lambda]{\phi^*}{g^*}:=\argmin_{y \in \bR^n} \{g^*(y) + \lambda D_{\phi^*}(y, \cdot)\}$ the left Bregman proximal mapping of $g^*$ with parameter $\lambda^{-1}$ and reference function $\phi^*$. Then the following statements hold:
\begin{propenum}
\item \label{thm:moreau_decomposition:decomp}
$\id = \aprox[\lambda]{\phi}{g} + \lambda \nabla \phi^* \circ \bprox[\lambda]{\phi^*}{g^*}\circ \nabla \phi \circ \lambda^{-1} \id$; 
\item \label{thm:moreau_decomposition:smooth} $g \infconv \lambda \star \phi \in \mathcal{C}^1(\bR^n)$ is convex with 
$
\nabla (g \infconv \lambda \star \phi) = \nabla \phi \circ \lambda^{-1}(\id - \aprox[\lambda]{\phi}{g})
$;
\item \label{thm:moreau_decomposition:bprox} $\bprox[\lambda]{\phi^*}{g^*} \circ \nabla(\lambda \star \phi) =(\partial g^* + \lambda \nabla \phi^*)^{-1} = \nabla (g \infconv 
\lambda \star \phi)$;
\item \label{thm:moreau_decomposition:aprox} $\aprox[\lambda]{\phi}{g} = (\id + \lambda \nabla \phi^* \circ \partial g)^{-1}$.
\end{propenum}
\end{proposition}
The proof is omitted for brevity since it follows closely the ones of the results in the given references invoking basic properties of infimal convolutions, in particular \cite[Proposition 18.7]{BaCo110}.
Note that for $\phi=\frac{1}{2}\|\cdot\|^2$ \cref{thm:moreau_decomposition:decomp} specializes to the well known Moreau decomposition of the classical Euclidean proximal mapping where the expressions in \cref{thm:moreau_decomposition:bprox,thm:moreau_decomposition:aprox} yield the Euclidean resolvents $(\id + \lambda^{-1} \partial g^*)^{-1}$ resp. $(\id + \lambda \partial g)^{-1}$ of the subdifferentials of $g^*$ and $g$.
The following remarks are in order.
\begin{remark}[relation to the Bregman proximal gradient algorithm \cite{bauschke2017descent,lu2018relatively}] 
\ifx\ifsvjour\true
 \label[remark]{rem:relation_bregman_pgd}
\else
 \label{rem:relation_bregman_pgd}
\fi
In this remark we shall clarify the differences between \cref{alg:aproxgrad} and the Bregman proximal gradient algorithm \cite{bauschke2017descent,lu2018relatively}. In the Bregman proximal gradient algorithm the forward-step takes the form $y^k = \nabla \phi^*(\nabla \phi(x^k) - \lambda\nabla f(x^k))$ which differs from the forward-step \cref{eq:forward} of our method unless $\phi(x)=\frac{1}{2}\langle x, Mx \rangle$, as in \cref{rem:euclidean_pgd}, in which case both methods collapse to the same algorithm. Despite the duality relation between the backward-step \cref{eq:backward} of our method and the backward-step of the Bregman proximal gradient method given in the lemma above the anisotropic proximal gradient algorithm is not equivalent to the Bregman proximal gradient method applied to the Fenchel--Rockafellar dual problem. In fact, even in the Euclidean case the proximal gradient algorithm is not self-dual under Fenchel--Rockafellar duality.
\ifx\ifsvjour\true\else
Instead the anisotropic proximal gradient method is self-dual under a certain double-min duality as shown in \Cref{sec:dca}.
\fi
\end{remark}
\begin{remark}
If $g$ is convex, in light of the above result, well-definedness of the anisotropic proximal mapping and Moreau envelope is implied by validity of the constraint qualification:
\begin{assumenum}[resume*]
\item \label{assum:a5'} $g$ is convex and $\emptyset \neq \dom g^* \cap \intr (\dom \phi_-^*)$ and $\lambda \leq 1/L$.
\end{assumenum}
Henceforth, whenever $g$ is assumed to be convex we impose \cref{assum:a5'} in place of \cref{assum:a5}. In fact, this implies that $f + \lambda \star \phi(\cdot - y)$ is level-bounded for any $y$ and for any $\lambda >0$ and thus in particular $\lambda_g=+\infty$.
Further note that for $g\equiv 0$ the CQ becomes $0 \in \intr \dom \phi^*$ and thus $0 \in \ran \nabla \phi$; also see \cref{rem:dual_space_pc}
\end{remark}

\section{Anisotropic smoothness: calculus and examples} \label{sec:calculus_asmooth}
\subsection{Basic calculus}
In this subsection we show that anisotropic smoothness is a property satisfied by many practically relevant problems such as logistic regression, exponential regression or problems involving $\logsumexp$. For that purpose we show how to assemble more complicated functions from simple ones by developing a calculus for anisotropic smoothness.
To begin with we shall provide equivalent characterizations of anisotropic smoothness. In particular we show that anisotropic smoothness of $f\in \mathcal{C}^1(\bR^n)$ relative to $\phi$ means equivalently that $f$ can be written in terms of a smooth infimal convolution with $\phi$. This generalizes \cite{laude2021conjugate} to reference functions $\phi$ which are not necessarily super-coercive.
In the convex case we also provide dual characterizations of anisotropic smoothness in terms of relative strong convexity in the Bregman sense
refining \cite[Lemma 4.2]{wang2021bregman}.
In the smooth case relative strong convexity was considered in \cite{lu2018relatively,bauschke2019linear}.
\begin{proposition}[characterizations of anisotropic smoothness]
\ifx\ifsvjour\true
 \label[proposition]{thm:phi_convex_asmooth}
\else
 \label{thm:phi_convex_asmooth}
\fi
Let $f:\bR^n \to \exR$ be proper lsc and $L>0$. 
Then the following are equivalent:
\begin{propenum}
\item \label{thm:phi_convex_asmooth:asmooth} $f$ has the anisotropic descent property relative to $\phi$ with constant $L$;
\item \label{thm:phi_convex_asmooth:infconv} $f\in \mathcal{C}^1(\bR^n)$ with $\ran \nabla f \subseteq \ran \nabla \phi$ and $f = \xi \infconv L^{-1}\star \phi$ for some $\xi:\bR^n \to \exR$.
\end{propenum}
Any of the above statements implies that the infimal convolution in the second statement can be taken to be exact for some $\xi:\bR^n \to \exR$ proper, lsc, where exactness means that for any $x \in\bR^n$ there is $y \in \bR^n$ such that $(\xi \infconv L^{-1}\star \phi)(x)=L^{-1}\star \phi(x-y) + \xi(y)$.

If, in addition, $f$ is convex the assumption $f\in \mathcal{C}^1(\bR^n)$ is superfluous in the second statement, $\xi$ can be taken to be convex lsc and the following items can be added to the list
\begin{propenum}[resume]
\item \label{thm:phi_convex_asmooth:str_cvx} $f^* = \psi + L^{-1}\phi^*$ for some $\psi \in \Gamma_0(\bR^n)$ with $\dom \psi \cap \intr \dom\phi^* \neq \emptyset$;
\item \label{thm:phi_convex_asmooth:str_cvx_relint} $\dom f^* \subseteq \dom \phi^*$ with $\relint \dom f^* \cap \intr \dom \phi^* \neq \emptyset$ and $f^* \mathbin{\dot{-}} L^{-1}\phi^*$ is convex on $\relint \dom f^*$, i.e., $f^*$ is strongly convex relative to $\phi^*$ with constant $L^{-1}$ in the Bregman sense.
\end{propenum}
If, furthermore, $f,\phi$ are Legendre and $\mathcal{C}^2(\bR^n)$ with nonsingular Hessian the following item can be added to the list
\begin{propenum}[resume]
\item \label{thm:phi_convex_asmooth:cvx_c2} $\ran \nabla f \subseteq \ran \nabla \phi$ and $\nabla^2 f(x) \preceq L \nabla^2 \phi(\nabla \phi^*(\nabla f(x)))$ for all $x \in \bR^n$.
\end{propenum}
\end{proposition}
We highlight that the existence of an infimal convolution expression is equivalent to the notion of \emph{generalized convexity} introduced in \Cref{sec:generalized_conjugacy}; see \cref{thm:inf_conv_rem}. In particular, generalized convexity is a key ingredient for the proof of the equivalence between the first two items which follows along the lines of the proof of \cite[Proposition 4.5]{laude2021conjugate} replacing $f$ with $-f$. The proof of the additional equivalences under convexity follows along the lines of the proofs of \cite[Theorem 4.6]{laude2021conjugate} and \cite[Lemma 4.2]{wang2021bregman}. The last equivalence follows by \cite[Proposition 3.4]{maddison2021dual} that invokes the inverse function theorem. We provide a complete proof in \Cref{sec:thm:phi_convex_asmooth}.

\begin{remark}[sufficiency of Hessian bound for anisotropic smoothness in more general cases]
\ifx\ifsvjour\true
 \label[remark]{rem:sufficiency_hessian}
\else
 \label{rem:sufficiency_hessian}
\fi
Under twice differentiability of $f,\phi$, by equating the second-order Taylor expansions of $f$ and the anisotropic upper bound \cref{eq:adescent} it can be proved that the Hessian characterization \cref{thm:phi_convex_asmooth:cvx_c2} is always a necessary condition for anisotropic smoothness (despite a lack of convexity and nonsingularity of the Hessians everywhere).
Unfortunately, sufficiency is hard to verify in general as the Hessian bound is a local pointwise condition that needs to be transformed into a global one.
In case $f$ is convex but not necessarily Legendre with invertible Hessian the above result can be generalized by replacing $f$ with a certain regularized version that is Legendre convex with invertible Hessian as explored in \cite[Proposition 3.7]{maddison2021dual}.
More generally, one can leverage the interpretation of anisotropic smoothness in terms of $\Phi$-convexity, \cref{thm:inf_conv_rem}:
The sufficiency of the Hessian bound for $\Phi$-convexity is studied in \cite[Theorem 3.16]{léger2023gradient} and \cite[Theorem 12.46]{Vil08} for general couplings $\Phi$ that exhibit nonnegative cross-curvature \cite[Definition 2.8]{léger2023gradient}: In particular, if $\phi=\SumExp$, by \cite[Example 2.13(4)]{léger2023gradient}, $\Phi: (x,y)\mapsto \SumExp(x-y) \in \mathcal{C}^4(\bR^n \times \bR^n)$ has nonnegative cross-curvature and hence by \cite[Theorem 3.16]{léger2023gradient} the Hessian characterization \cref{thm:phi_convex_asmooth:cvx_c2} is sufficient for anisotropic smoothness even in the nonconvex case. In this case we have $\nabla^2 \phi(y) = \diag \Exp(y)$ and $\nabla \phi^*(\nabla f(x)) = \Log(\nabla f(x))$ and hence the condition reads $\nabla^2 f(x) \preceq L \diag(\nabla f(x))$ for $\ran \nabla f \subseteq \bR^n_{++}$. This is reminiscent of $(L_0, L_1)$-smoothness \cite{Zhang2020Why} from machine learning and a first-order version of \emph{quasi self-concordance} \cite{doikov2023minimizing,bach2010self}.
\end{remark}


Next we shall study operations that preserve anisotropic smoothness. To begin with we show that in contrast to Bregman relative smoothness, anisotropic smoothness enjoys a translation invariance property:
\begin{proposition}[translation invariance]
\ifx\ifsvjour\true
 \label[proposition]{thm:translation_invariance}
\else
 \label{thm:translation_invariance}
\fi
Let $f \in\mathcal{C}^1(\bR^n)$ have the anisotropic descent property relative to $\phi$ with constant $L$ and let $b\in\bR^n$. Then $f(\cdot - b)$ has the anisotropic descent property relative to $\phi$ with constant $L$ as well.
\end{proposition}
\begin{proposition}[epi-scaling]
\ifx\ifsvjour\true
 \label[proposition]{thm:epi_scaling_asmooth}
\else
 \label{thm:epi_scaling_asmooth}
\fi
Let $f \in\mathcal{C}^1(\bR^n)$ have the anisotropic descent property relative to $\phi$ with constant $L$. Let $\sigma >0$. Then $\sigma \star f$ has the anisotropic descent property relative to $\phi$ with constant $L/\sigma$.
\end{proposition}
Using the characterization of anisotropic smoothness in terms of an infimal convolution the proofs are elementary.

Next we show that like its Euclidean counterpart anisotropic smoothness is separable:
\begin{proposition}[separability of anisotropic smoothness]
\ifx\ifsvjour\true
 \label[proposition]{thm:separable}
\else
 \label{thm:separable}
\fi

Let $f_i \in\mathcal{C}^1(\bR^{n_i})$ have the anisotropic descent property relative to $\phi_i:\bR^{n_i} \to \bR$ with constants $L_i$ for $1\leq i \leq m$. Then $f:\bR^{n_1} \times \cdots \times \bR^{n_m} \to \bR$ defined as $f(x):=\sum_{i=1}^m f_i(x_i)$ has the anisotropic descent property relative to $\phi:\bR^{n_1} \times \cdots \times \bR^{n_m} \to \bR$ defined by $\phi(x):=\sum_{i=1}^m L_i^{-1} \star \phi_i(x_i)$.
\end{proposition}
\begin{proof}
The result follows by summing the individual descent inequalities.
\ifx\ifsvjour\true
\qed
\fi
\end{proof}

We have the following elementary property:
\begin{proposition}[pointwise scaling and tilting]
Let $f \in\mathcal{C}^1(\bR^n)$ have the anisotropic descent property relative to $\phi$ with constant $L$ and let $c \in \bR^n$ and $\alpha >0$. Then $\alpha f+\langle c,\cdot \rangle$ has the anisotropic descent property relative to $\alpha \phi+\langle c,\cdot \rangle$ with constant $L$.
\end{proposition}
\begin{proof}
This follows by definition of the anisotropic descent property noting that $(\alpha \phi+\langle c,\cdot \rangle)^*= (\alpha \star \phi^*)(\cdot - c)$ and $\nabla (\alpha \star \phi^*)(\cdot - c) = \nabla \phi^*((\cdot )\alpha^{-1} - c)$.
\ifx\ifsvjour\true
\qed
\fi
\end{proof}

Based on the infimal convolution interpretation of anisotropic smoothness we can prove the following elementary transitivity property of anisotropic smoothness:
\begin{proposition}[transitivity of anisotropic smoothness]
\ifx\ifsvjour\true
 \label[proposition]{thm:transitivity}
\else
 \label{thm:transitivity}
\fi

Let $f \in \mathcal{C}^1(\bR^n)$ and $\phi_1,\phi_2 \in \Gamma_0(\bR^n)$ be Legendre with full domain. Assume that $f$ is anisotropically smooth relative to $\phi_1$ and let $\phi_1$ be anisotropically smooth relative to $\phi_2$. Then $f$ has the anisotropic descent property relative to $\phi_2$.
\end{proposition}
\begin{proof}
Thanks to \cref{thm:phi_convex_asmooth} anisotropic smoothness of $f$ relative to $\phi_1$ implies that $h = \xi \infconv \phi_1$ for some $\xi:\bR^n \to \exR$ proper lsc and that $\ran \nabla f \subseteq \ran \nabla \phi_1$. Anisotropic smoothness of $\phi_1$ relative to $\phi_2$ implies that $\phi_1 = \psi \infconv \phi_2$ for some $\psi:\bR^n \to \exR$ proper lsc and that $\ran \nabla \phi_1 \subseteq \ran \nabla \phi_2$. Thus we have $f = \xi \infconv (\psi \infconv \phi_2) = (\xi \infconv \psi) \infconv \phi_2$ and $\ran \nabla f \subseteq \ran \nabla \phi_1 \subseteq  \ran \nabla \phi_2$. Again invoking \cref{thm:phi_convex_asmooth} $f$ is anisotropically smooth relative to $\phi_2$.
\ifx\ifsvjour\true
\qed
\fi
\end{proof}

Next we show that anisotropic smoothness with constant $L_1$ is preserved for any $L_2>L_1$. While this is obvious in the Euclidean case, due to the epi-scaling used in anisotropic smoothness, a proof is required in the general case:
\begin{proposition}[monotonicity under epi-scaling]
\ifx\ifsvjour\true
 \label[proposition]{thm:episcaling_asmoothness}
\else
 \label{thm:episcaling_asmoothness}
\fi
Let $f \in \mathcal{C}^1(\bR^n)$ and $\bar x, x \in \bR^n$ with $x \neq \bar x$ and suppose that
\begin{align*}
f(x) &\leq f(\bar x) + \tfrac{1}{L_1} \star \phi(x-\bar x + {L_1}^{-1}\nabla\phi^*(\nabla f(\bar x))) - \tfrac{1}{{L_1}} \star \phi({L_1}^{-1}\nabla\phi^*(\nabla f(\bar x))).
\end{align*}
Then we have for any $L_2>L_1$ that
\begin{align*}
f(x) < f(\bar x) + \tfrac{1}{L_2} \star \phi(x-\bar x + {L_2}^{-1}\nabla\phi^*(\nabla f(\bar x))) - \tfrac{1}{{L_2}} \star \phi({L_2}^{-1}\nabla\phi^*(\nabla f(\bar x))).
\end{align*}
In particular, this implies the following monotonicity property of anisotropic smoothness under epi-scaling: If $f$ has the anisotropic descent property relative to $\phi$ with constant $L_1>0$, then $f$ has the anisotropic descent property relative to $\phi$ for any $L_2>L_1$.
\end{proposition}
A proof is provided in \cref{sec:thm_episcaling_asmoothness}.

Naturally, anisotropic smoothness is closed under epi-addition; see \cite[Corollary 4.9]{laude2021conjugate}. In contrast, the pointwise sum of anisotropically smooth functions is in general not anisotropically smooth, see \cite[Example 4.10]{laude2021conjugate}, unless the dual Bregman divergence is jointly convex: This is a direct consequence of the convexity of the recently introduced Bregman proximal average \cite[Theorem 5.1(i)]{wang2021bregman}.
Next we will generalize this to nonconvex infimal convolutions. In the subsequent section we then consider a refinement for the exponential reference function showing closedness under pointwise conic combinations. This is false in general for Euclidean Lipschitz smoothness with the same Lipschitz constant.

\begin{proposition}[closedness of anisotropic smoothness under pointwise average]
\ifx\ifsvjour\true
 \label[proposition]{thm:closed_addition}
\else
 \label{thm:closed_addition}
\fi

Let $f_i \in \mathcal{C}^1(\bR^n)$ for $1\leq i \leq m$ be anisotropically smooth relative to $\phi$ with constants $L_i$ and assume that $D_{\phi^*}$ is jointly convex. Let $\alpha_i \geq 0$ and $\sum_{i=1}^m \alpha_i = 1$. Then $\sum_{i=1}^m \alpha_i f_i$ is anisotropically smooth relative to $\phi$ with constant $L:=\max_{1 \leq i \leq m} L_i$.
\end{proposition}
The proof adapts some key steps of \cite[Theorem 5.1(ii)]{wang2021bregman} and treats the nonconvex case. A proof is provided in \cref{sec:thm_closed_addition}.

Using the calculus rules devised above we can show in the next example that the loss function in logistic regression is anisotropically smooth relative to a softmax approximation of the absolute value function:
\begin{example}
\ifx\ifsvjour\true
 \label[example]{ex:logistic}
\else
 \label{ex:logistic}
\fi
Let $f(x)=\frac{1}{m} \sum_{i=1}^m \ell(\langle a_i, x \rangle)$ with $a_i \in [-1, 1]^n$ and $\ell(t)=\ln(1 + \exp(t))$. Then $f$ is anisotropically smooth relative to the symmetrized logistic loss $\phi(x)=\sum_{j=1}^n h(x_j)$ for $h(t)=2\ln(1+ \exp(t)) - t$ with constant $L = \max_{1\leq i \leq m} \|a_i\|^2$. Note that for $L\to \infty$, $\frac1L\star h$ converges pointwisely to the absolute value function.
\end{example}
A proof is provided in \Cref{sec:ex_logistic}.

\subsection{Exponential smoothness}
In this subsection we specialize $\phi=\SumExp$ which is an important instance of our framework. In this case we refer to anisotropic smoothness as \emph{exponential smoothness}.
As a motivation we first provide a simple example that shows that relative smoothness in the Bregman sense \cite{birnbaum2011distributed,bauschke2017descent,lu2018relatively} with the exponential lacks certain important invariances:
\begin{example}[lack of scale invariance in Bregman smoothness]
Choose $f(x)=\exp(ax -b)$ for $a > 0$ and $b \in \bR$. Then $f$ is not relatively smooth wrt $\phi=\exp$ for any $L$ unless $a= 1$. To see this we consider the second derivatives:
$f''(x) = a^2 \exp(ax-b)$ and $\phi''(x)=\exp(x)$. Consider $x \mapsto L\phi''(x)-f''(x)=L\exp(x)-a^2\exp(-b)\exp(a x)$. It is easy to see that for any $L>0$, $\exp(ax)$ exhibits a different growth than $\exp(x)$: If $a < 1$ for $x<0$ sufficiently small we have $L\exp(x)-a^2\exp(-b)\exp(a x) < 0$. If $a>1$ for $x>0$ sufficiently large again $L\exp(x)-a^2\exp(-b)\exp(a x) < 0$ and thus $f$ is not Bregman smooth relative to $\exp$.
\end{example}
As we shall see exponential smoothness is a more promising candidate enjoying powerful invariances and a rich calculus.
Considering the Hessian characterization
\begin{align}
\nabla^2 f(x) \preceq L \diag(\nabla f(x)),
\end{align}
with $\ran \nabla f \subseteq \bR^n_{++}$ (given in \cref{thm:phi_convex_asmooth:cvx_c2,rem:sufficiency_hessian}) exponential smoothness bears similarities to $(L_0, L_1)$-smoothness \cite{Zhang2020Why} and a first-order version of \emph{quasi self-concordance} \cite{doikov2023minimizing,bach2010self}. In contrast to $(L_0, L_1)$-smoothness and classical Lipschitz smoothness (with same constants), exponential smoothness is closed under addition and positive scaling--two essential properties it shares with quasi self-concordant functions. This is shown in the next proposition which is as a refinement of \cref{thm:closed_addition}. In \cref{sec:plusminus} we illustrate that the restriction $\ran \nabla f \subseteq \bR^n_{++}$ can be lifted using the plus-minus trick which is easy to implement.

\begin{proposition}[closedness under pointwise conic combinations]
\ifx\ifsvjour\true
 \label[proposition]{thm:closed_addition_exp}
\else
 \label{thm:closed_addition_exp}
\fi
Let $f_i \in \mathcal{C}^1(\bR^n)$ be exponentially smooth with constant $L_i$. Let $\alpha_i \geq 0$ such that $\sum_{i=1}^m \alpha_i >0$. Then $\sum_{i=1}^m \alpha_i f_i$ is exponentially smooth with constant $L:=\max_{1 \leq i \leq m} L_i$.
In particular this shows that $\mathcal{C}_{\Exp}^L:=\{f \in \mathcal{C}^1(\bR^n) : \text{f is exponentially smooth with constant $L$}\}$ is a convex cone.
\end{proposition}
The proof is based on the infimal convolution characterization of exponential smoothness noting the identity $\exp(x-y_1) + \exp(x-y_2)= \exp(x +\ln(\exp(-y_1) + \exp(-y_2)))$. A complete proof is provided in \cref{sec:thm_closed_addition_exp}

Using the invariance under translation and positive scaling \cref{thm:translation_invariance,thm:closed_addition_exp}, we can show the following property:
\begin{proposition}[closedness under affine transformation and scaling]
\ifx\ifsvjour\true
 \label[proposition]{thm:exponential_smoothness_penalty}
\else
\label{thm:exponential_smoothness_penalty}
\fi
Let $h$ be exponentially smooth with constant $L >0$. Let $A \in \bR^{m \times n}_+$ such that each column contains at least one nonzero element, $a >0$ and $b \in \bR^m$.
Then $f(x) := a h(Ax-b)$ is exponentially smooth with constant $M=L\|A\|_{\infty}$ where $\|A\|_{\infty}:=\max_{1 \leq i \leq m} \sum_{j=1}^n |A_{ij}|$ is the $\infty$-norm of $A$.
\end{proposition}
A proof is provided in \cref{sec:thm_exponential_smoothness_penalty}.
The assumption that the entries of $A$ are nonnegative might seem restrictive. However, this can be mitigated by exploiting the plus-minus trick; see \Cref{sec:plusminus}.
\begin{example}[LogSumExp function] 
\ifx\ifsvjour\true
  \label[example]{ex:logsumexp}
\else
  \label{ex:logsumexp}
\fi
Choose $f(x)= \sigma \star \logsumexp(Ax-b)$ for $\logsumexp(z)=\ln(\sum_{j=1}^m \exp(z_j))$, $A \in \bR_+^{m \times n}$ such that each column contains at least one nonzero element and $b\in\bR^m$.
We show that $\logsumexp=\vecmax \infconv \SumExp$ for $\vecmax(z)= \max_{1 \leq j \leq m} z_j$. It can be readily checked that $\vecmax^* = \delta_{\Delta_m}$
 for $\Delta_m:=\{z \in \bR^m_+ : \sum_{j=1}^m z_j = 1\}$ being the $m$-dimensional unit-simplex. Hence
 \begin{align*}
(\vecmax \infconv \SumExp)(z) &= (\delta_{\Delta_m} + H)^*(z) =\sup_{y \in \Delta_m}\langle z,y \rangle - H(y).
 \end{align*}
Introducing a Lagrange multiplier for the sum-to-one constraint it can be readily checked that the supremum is attained for $y_j=\frac{\exp(z_j)}{\sum_{i=1}^m \exp(z_i)}$ which verifies the claim. Thus in light of \cref{thm:phi_convex_asmooth,thm:epi_scaling_asmooth}, $\sigma \star \logsumexp$ is exponentially smooth with constant $1/\sigma$.
Then in light of the previous result $f$ is exponentially smooth with constant $\|A\|_{\infty}/\sigma$.
Note that $f$ is also Lipschitz smooth but with the constant $\|A\|_2^2/\sigma$. Depending on the structure of $A$ this can lead to worse constants.
\end{example}
Clearly, the quadratic function $f=\frac{1}{2}\|x\|_2^2$ is not exponentially smooth. However, as shown in the next example we can decompose $f(x)=\Theta(x)+\Theta(-x)$ for a function $\Theta$ which is exponentially smooth. Then the gradient $\nabla f(x)$ of $f$ at $x$ is decomposed into a positive part $\nabla \Theta(x)$ and a negative part $-\nabla \Theta(-x)$. Invoking the plus-minus trick this allows one to minimize quadratic costs $x \mapsto \|Ax\|_2^2$ using the exponential reference function:
\begin{example}[plus-minus decomposition of quadratic via polylogarithm]
\ifx\ifsvjour\true
  \label[example]{ex:polylog}
\else
  \label{ex:polylog}
\fi
We show that $\frac{1}{2}\|x\|^2$ can be decomposed as $\frac{1}{2}\|x\|^2=\Theta(x)+\Theta(-x)$ for 
\begin{align}
\Theta(x) = \sum_{i=1}^n \vartheta(x_i), \quad \vartheta(t):=\int_0^t \ln(1+\exp(\tau))  \mathrm{d} \tau = -\Li_2(-\exp(t)) - \frac{\pi^2}{12},
\end{align}
where $\Li_2$ denotes the polylogarithm of order $2$. In particular, $\Theta \in \cC^2$ is convex and exponentially smooth with constant $L=1$.
\end{example}
A proof is provided in \cref{sec:ex_polylog}.

\section{Analysis of the anisotropic proximal gradient method} \label{sec:aproxgrad_analysis}
\subsection{Subsequential convergence and linesearch}
Unless stated otherwise, in this section we assume validity of \cref{assum:a1,assum:a2,assum:a3,assum:a4,assum:a5}. Furthermore recall the definition of the threshold of anisotropic prox-boundedness $\lambda_g$ and assume that $\lambda < \lambda_g$.
The main goal of this subsection is to analyse the subsequential convergence of the anisotropic proximal gradient method. For that purpose we shall introduce the following regularized gap function as a measure of stationarity: 
\begin{align}
\gap{\phi}{F}(\bar x, \lambda) := \tfrac{1}{\lambda}\big(F(\bar x) - F_{\lambda}(\bar x) \big),
\end{align}
for any $\lambda < \lambda_g$ and
\begin{align}
F_{\lambda}(\bar x) &:= \inf_{x \in \bR^n} f(\bar x) + \lambda \star \phi(x-\bar x + \lambda \nabla\phi^*(\nabla f(\bar x))) - \lambda \star \phi(\lambda \nabla\phi^*(\nabla f(\bar x))) + g(x) \notag \\
&= f(\bar x) + (g \infconv \lambda \star \phi_-)(\bar x - \lambda \nabla\phi^*(\nabla f(\bar x)))  - \lambda  \star \phi(\lambda \nabla\phi^*(\nabla f(\bar x))),
\end{align}
is an anisotropic generalization of the forward-backward envelope \cite{patrinos2013proximal,stella2017forward}. The regularized gap function is a generalization of \cite[Equation 13]{karimi2016linear} to the anisotropic case and is closely related to regularized gap functions for solving variational inequalities \cite{fukushima1992equivalent}. The particular scaling $\frac{1}{\lambda}$ will be important subsequently for proving a linear convergence result under anisotropic proximal gradient dominance.

Next we show the following key properties of $\gap{\phi}{F}(\bar x, \lambda)$ to justify its choice as a measure of stationarity:
\begin{lemma}
\ifx\ifsvjour\true
 \label[lemma]{thm:continuity_aprox_grad}
\else
 \label{thm:continuity_aprox_grad}
\fi
For any $\lambda < \lambda_g$ the gap function $\gap{\phi}{F}(\cdot, \lambda) \geq 0$ is lsc and for any $x^\star \in \bR^n$ with $\gap{\phi}{F}(x^\star, \lambda) = 0$ we have that $x^\star$ is a stationary point of $F$, i.e., $0 \in \nabla f(x^\star) + \partial g(x^\star)$. If, in addition, $g\equiv 0$, the gap simplifies to $\gap{\phi}{F}(\cdot, \lambda)=D_{\phi^*}(0, \nabla f(\cdot))$.
\end{lemma}
\begin{proof}
Let $\lambda < \lambda_g$. Let $\bar x \in \bR^n$ and define
$$
\xi(x, \bar x) := f(\bar x) + \lambda \star \phi(x-\bar x + \lambda \nabla\phi^*(\nabla f(\bar x))) - \lambda \star \phi(\lambda \nabla\phi^*(\nabla f(\bar x))) + g(x).
$$
Then $F_\lambda(\bar x)= \inf_{x \in \bR^n} \xi(x, \bar x) \leq \xi(\bar x, \bar x)= F(\bar x)$ implying that $\gap{\phi}{F}(\bar x, \lambda) \geq 0$.
Let $x^\star \in \bR^n$ and assume that $\gap{\phi}{F}(x^\star, \lambda ) = 0$, i.e., $F(x^\star) = F_{\lambda}(x^\star)$. Therefore $F_{\lambda}(x^\star) = \inf_{x \in \bR^n} \xi(x, x^\star) \leq \xi(x^\star, x^\star) = F(x^\star) = F_{\lambda}(x^\star)$ and thus the infimum in the definition of $F_{\lambda}(x^\star)$ is attained at $x^\star$.
By Fermat's rule \cite[Theorem 10.1]{RoWe98} we have
$$
0 \in \partial \xi(\cdot, x^\star)(x^\star) = \nabla \phi(\nabla \phi^*(\nabla f(x^\star))) + \partial g(x^\star) = \nabla f(x^\star) + \partial g(x^\star).
$$
Since by assumption for any $y \in \bR^n$, $g + \lambda \star \phi(\cdot-y)$ is level-bounded, for any $y \in \bR^n$ there exists $x \in \bR^n$ such that $(g \infconv \lambda \star \phi_-)(y)=g(x) + \lambda \star \phi(x-y)$ is finite. Furthermore, $-(g \infconv \lambda \star \phi_-)$ is a pointwise supremum over continuous functions and hence lsc. Since $\id - \lambda\nabla \phi^* \circ \nabla f$ is continuous, $-F_\lambda$ is lsc as well. Consequently, since $F=f+g$ is lsc we have that $\gap{\phi}{F}(\bar x, \lambda) := \frac{1}{\lambda}(F(\bar x) - F_{\lambda}(\bar x))$ is lsc.

If, in particular, $g\equiv 0$ we have for any $\bar x \in \bR^n$
\begin{align*}
\gap{\phi}{F}(\bar x, \lambda) &= -\frac{1}{\lambda} \min_{x \in \bR^n} \lambda \star \phi(x-\bar x + \lambda \nabla\phi^*(\nabla f(\bar x))) - \lambda  \star \phi(\lambda \nabla\phi^*(\nabla f(\bar x))) \\
&= \phi(\nabla\phi^*(\nabla f(\bar x))) -\min_{x \in \bR^n} \phi(\lambda^{-1}(x - \bar x) + \nabla\phi^*(\nabla f(\bar x))) \\
&=\phi(\nabla\phi^*(\nabla f(\bar x))) - \phi(\nabla \phi^*(0)) - \langle \nabla \phi(\nabla \phi^*(0)), \nabla\phi^*(\nabla f(\bar x)) - \nabla \phi^*(0) \rangle \\
&=D_{\phi}(\nabla \phi^*(\nabla f(\bar x)), \nabla \phi^*(0))= D_{\phi^*}(0, \nabla f(\bar x)),
\end{align*}
where the identities hold due to the constraint qualification $0\in\intr\dom\phi^*$ and the fact that $\hat x$ minimizes $x \mapsto \phi(\lambda^{-1}(x - \bar x) + \nabla\phi^*(\nabla f(\bar x)))$ if and only if $\nabla \phi^*(0) = \lambda^{-1}(\hat x - \bar x) + \nabla\phi^*(\nabla f(\bar x))$, the identity $\nabla \phi(\nabla \phi^*(0))=0$ and the identity in \cref{thm:dual_bregman}.
\ifx\ifsvjour\true
\qed
\fi
\end{proof}
We first prove the following sufficient decrease property:
\begin{lemma}
\ifx\ifsvjour\true
 \label[lemma]{thm:sufficient_descent}
\else
 \label{thm:sufficient_descent}
\fi
Let $\{x^k\}_{k =0}^\infty$ be the sequence of iterates generated by \cref{alg:aproxgrad}.
Then the following sufficient decrease property holds true for all $k \in \bN_0$:
\begin{align}
F(x^{k+1}) \leq F(x^k) - \lambda \gap{\phi}{F}(x^k, \lambda).
\end{align}
\end{lemma}
\begin{proof}
We have:
\begin{align*}
F(x^{k+1}) &= f(x^{k+1}) +g(x^{k+1}) \\
&\leq f(x^k) + \tfrac{1}{L} \star \phi(x^{k+1} - x^k + L^{-1}\nabla \phi^*(\nabla f(x^k)))- \tfrac{1}{L} \star \phi(L^{-1}\nabla \phi^*(\nabla f(x^k))) + g(x^{k+1}) \\
&\leq f(x^k) + \lambda \star \phi(x^{k+1} - x^k + \lambda \nabla \phi^*(\nabla f(x^k))) - \lambda \star \phi(\lambda \nabla \phi^*(\nabla f(x^k))) + g(x^{k+1}) \\
&= F_{\lambda}(x^k) = F(x^k) - (F(x^k) -F_{\lambda}(x^k)) = F(x^k) - \lambda \gap{\phi}{F}(x^k, \lambda),
\end{align*}
where the first inequality follows by the anisotropic descent inequality, the second inequality by \cref{thm:episcaling_asmoothness} and the choice $\lambda^{-1} \geq L$ and the second equality by the update of $x^{k+1}$ and the definition of $F_{\lambda}(x^k)$.
\ifx\ifsvjour\true
\qed
\fi
\end{proof}

\begin{theorem}
\ifx\ifsvjour\true
 \label[theorem]{thm:asymptotic_convergence}
\else
 \label{thm:asymptotic_convergence}
\fi

Let $\{x^k\}_{k =0}^\infty$ be the sequence of iterates generated by \cref{alg:aproxgrad}.
The sequence $\{F(x^k)\}_{k =0}^\infty$ is nonincreasing and convergent. In addition, every limit point $x^\star$ of $\{x^k\}_{k =0}^\infty$ is a stationary point of $F$, i.e., $0 \in \nabla f(x^\star) + \partial g(x^\star)$.
In particular the minimum over the past regularized gaps vanishes sublinearly at rate $\mathcal{O}(1/K)$:
$$
\min_{k \in \{0,1,\ldots, K-1\}} \gap{\phi}{F}(x^k, \lambda) \leq \frac{1}{\lambda K} \big(F(x^0) - \inf F \big).
$$
\end{theorem}
\begin{proof}
Thanks to \cref{thm:sufficient_descent} we have
\begin{align*}
F(x^{k+1}) &\leq F(x^k) - \lambda \gap{\phi}{F}(x^k, \lambda).
\end{align*}
This implies that $F(x^{k})$ is monotonically decreasing. Since $-\infty < \inf F \leq F(x^{k})$ is bounded from below this means that $F(x^{k}) \to F^*$ converges to some $F^*$.
Summing the inequality we obtain since $-\infty < \inf F \leq F(x^{K})$ is bounded from below:
\begin{align}
-\infty < \inf F - F(x^0) &\leq F(x^{K}) -F(x^0) \notag\\
&= \sum_{k=0}^{K-1} F(x^{k+1}) -F(x^k)\leq - \lambda \sum_{k=0}^{K-1} \gap{\phi}{F}(x^k, \lambda) \label{eq:inequality_telescope}.
\end{align}
This implies that $\{\sum_{k=0}^{K-1} \gap{\phi}{F}(x^k, \lambda) \}_{K \in \bN}$ is convergent and thus we have
\begin{align} \label{eq:limit_convergence}
\lim_{k\to \infty} \gap{\phi}{F}(x^k, \lambda) = 0.
\end{align}
Let $x^\star$ be a limit point of $\{x^k\}_{k =0}^\infty$. Let $x^{k_j} \to x^\star$ be a corresponding subsequence.
In view of \cref{thm:continuity_aprox_grad} $\gap{\phi}{F}(\cdot, \lambda)$ is lsc and thus we have
$$
0 \leq \gap{\phi}{F}(x^\star, \lambda) \leq \lim_{j \to \infty} \gap{\phi}{F}(x^{k_j}, \lambda) = 0,
$$
where the last equality follows from \cref{eq:limit_convergence}.
Then invoking \cref{thm:continuity_aprox_grad} again this implies that $x^\star$ is a stationary point of $f+g$, i.e., $0 \in \nabla f(x^\star) + \partial g(x^\star)$.

Thanks to \cref{eq:inequality_telescope} we have
$$
K \cdot \min_{k \in \{0,1,\ldots, K-1\}} \gap{\phi}{F}(x^k, \lambda) \leq \sum_{k=0}^{K-1} \gap{\phi}{F}(x^k, \lambda) \leq \frac{1}{\lambda}(F(x^0) - \inf F).
$$
Dividing by $K$ we obtain the claimed sublinear rate.
\ifx\ifsvjour\true
\qed
\fi
\end{proof}
\begin{remark}
To our knowledge feasibility of the step-size $\lambda=1/L$ for nonconvex $g$ is also new in the Euclidean case as existing results typically require $\lambda <1/L$. In the Euclidean case (uniform) level-boundedness is implied by prox-boundedness with some threshold $\lambda_g$; see \cite[Theorem 1.25]{RoWe98}. In the anisotropic case this holds for any $\lambda >0$ if for example $g$ (or $\phi$) is bounded from below and $\phi$ (or $g$) is coercive in which case we have $\lambda_g = \infty$.
\end{remark}


If $L$ is unknown (or only an over-estimate is available) we can still apply our algorithm by estimating a (sharper) step-size $\lambda_k$ in each iteration $k$ using a backtracking linesearch: 
Given $0<\lambda_{k-1}< \lambda_g$ and a constant $0<\alpha < 1$ we can backtrack $\lambda_k = \alpha^t \lambda_{k-1}$ until for some $t \in \bN_0$ the descent inequality \cref{eq:adescent} with constant $L=1/\lambda_k$ is valid at $x^k$ and $x^{k+1}$. Thanks to the scaling property of the anisotropic descent inequality \cref{thm:episcaling_asmoothness}, the linesearch always terminates with some $\lambda_k>0$.
The complete scheme is listed in \cref{alg:linesearch_aproxgrad}.

\begin{algorithm}[H]
\caption{Linesearch anisotropic proximal gradient (LS-anisoPG$^\alpha$)}
\label{alg:linesearch_aproxgrad}
\begin{algorithmic}
\REQUIRE Let $0< \alpha < 1$ and choose some $\lambda_g >\lambda_{-1} >0$. Let $x^0\in \bR^n$.
\FORALL{$k=0, 1, \ldots$}
\STATE $t\gets0$
\REPEAT
   \STATE $\lambda_k\gets\alpha^t \cdot \lambda_{k-1}$
   \STATE $y^k\gets x^{k} - \lambda_k \nabla \phi^*(\nabla f(x^k))$
   \STATE $x^{k+1} \gets\argmin_{x \in \bR^n} \;\{\lambda_k \star \phi(x - y^k) + g(x)\}$
   \STATE $t\gets t+1$
\UNTIL{$f(x^{k+1}) \leq \lambda_k \star \phi(x^{k+1} - y^k) -\lambda_k \star \phi(\lambda_k \nabla \phi^*(\nabla f(x^k))) + f(x^k)$}
\IF{$\tfrac{\lambda_k}{\alpha} < \lambda_g$}
\STATE $\lambda_{k} \gets \tfrac{\lambda_k}{\alpha}$
\ENDIF
\ENDFOR
\end{algorithmic}
\end{algorithm}

In the following theorem we show that \cref{alg:linesearch_aproxgrad} converges subsequentially:
\begin{theorem}[convergence under backtracking linesearch]
\ifx\ifsvjour\true
 \label[theorem]{thm:linesearch}
\else
 \label{thm:linesearch}
\fi
Let $\{x^k\}_{k =0}^\infty$ be the sequence of iterates generated by \cref{alg:linesearch_aproxgrad} and denote by $\{\lambda_k\}_{k =0}^\infty$ the corresponding sequence of step-sizes.
Then the sequence $\{F(x^k)\}_{k =0}^\infty$ is nonincreasing and convergent. In addition, every limit point $x^\star$ of $\{x^k\}_{k =0}^\infty$ is a stationary point of $F$, i.e., $0 \in \nabla f(x^\star) + \partial g(x^\star)$.
In particular the minimum over the past regularized gaps vanishes sublinearly at rate $\mathcal{O}(1/K)$:
$$
\min_{k \in \{0,1,\ldots, K-1\}} \gap{\phi}{F}(x^k, \lambda_{\min}) \leq \frac{1}{K\lambda_{\min}} \big(F(x^0) - \inf F \big),
$$
where $\lambda_{\min}:=\min \{\lambda_{-1}, \alpha/L \}$.
\end{theorem}
\begin{proof}
Thanks to \cref{thm:episcaling_asmoothness} in each iteration $k \in \bN_0$ the linesearch terminates with some $\lambda_k$ with $\lambda_{\min} \leq \lambda_k$ such that
$$
f(x^{k+1}) \leq f(x^k) + \lambda_k \star \phi(x^{k+1} - x^k + \lambda_k \nabla \phi^*(\nabla f(x^k))) - \lambda_k\star \phi(\lambda_k \nabla \phi^*(\nabla f(x^k))).
$$
Adding $g(x^{k+1})$ to both sides of the inequality we have:
\begin{align*}
F(x^{k+1}) &= f(x^{k+1}) +g(x^{k+1}) \\
&\leq f(x^k) + \lambda_k \star \phi(x^{k+1} - x^k + \lambda_k \nabla \phi^*(\nabla f(x^k))) - \lambda_k \star \phi(\lambda_k \nabla \phi^*(\nabla f(x^k))) + g(x^{k+1}) \\
&=\inf_{x \in \bR^n} f(x^k) + \lambda_k \star \phi(x - x^k + \lambda_k\nabla \phi^*(\nabla f(x^k))) - \lambda_k \star \phi(\lambda_k \nabla \phi^*(\nabla f(x^k))) + g(x) \\
&\leq \inf_{x \in \bR^n} f(x^k) + \lambda_{\min} \star \phi(x - x^k + \lambda_{\min}\nabla \phi^*(\nabla f(x^k))) - \lambda_{\min} \star \phi(\lambda_{\min} \nabla \phi^*(\nabla f(x^k))) + g(x) \\
&=F_{\lambda_{\min}}(x^k) = F(x^k) - (F(x^k) -F_{\lambda_{\min}}(x^k)) \\
&= F(x^k) - \lambda_{\min} \gap{\phi}{F}(x^k, \lambda_{\min}),
\end{align*}
where the second equality holds by the $x$-update and the last inequality holds by \cref{thm:episcaling_asmoothness} and the inequality $\lambda_{\min}^{-1} \geq \lambda_k^{-1}$ and the last equalities by definition of $F_{\lambda_{\min}}$.

%
%
Summing the inequality from $k=0$ to $k=K-1$ and adapting the proof of \cref{thm:asymptotic_convergence} we obtain the claimed result.
\ifx\ifsvjour\true
\qed
\fi
\end{proof}

\subsection{$Q$-linear convergence under anisotropic proximal gradient dominance}
The following definition and subsequent theorem is a generalization of the proximal PL-inequality due to \cite{karimi2016linear} to the anisotropic case:
\begin{definition}[anisotropic proximal gradient dominance]
We say that $F:=f+g$ satisfies the anisotropic proximal gradient dominance condition relative to $\phi$ with constant $\mu>0$ for parameter $0<\lambda\leq \min\{\mu^{-1},\lambda_g\}$ if 
for all $\bar x \in \bR^n$
\begin{align} \label{eq:anisotropic_PL}
\mu(F(\bar x) - \inf F) \leq \gap{\phi}{F}(\bar x, \lambda).
\end{align}
\end{definition}
In light of \cref{thm:continuity_aprox_grad} for $g \equiv 0$ and $\phi=\frac{1}{2}\|\cdot\|^2$ the gap function $\gap{\phi}{F}(\bar x, \lambda) = \frac{1}{2}\|\nabla f(\bar x)\|^2$ and thus \cref{eq:anisotropic_PL} specializes to the classical PL-inequality $\mu(f(\bar x) - \inf f) \leq \frac{1}{2}\|\nabla f(\bar x)\|^2$.

Next we establish the $Q$-linear convergence\footnote{We say that a sequence $\{x^k\}_{k =0}^\infty$ converges $Q$-linearly to $L$ with rate $\mu$ if $\lim_{k\to \infty} \frac{\|x^{k+1}- L\|}{\|x^k - L\|} \leq \mu$. We say that $\{x^k\}_{k =0}^\infty$ converges $R$-linearly to $L$ with rate $\mu$ if $\|x^k -L\| \leq \varepsilon_k$ and $\varepsilon_k$ converges $Q$-linearly to $0$.} of our algorithm under anisotropic proximal gradient dominance. This follows immediately from \cref{thm:sufficient_descent}:
\begin{theorem}
\ifx\ifsvjour\true
 \label[theorem]{thm:linear_convergence}
\else
 \label{thm:linear_convergence}
\fi
Let $\{x^k\}_{k =0}^\infty$ be the sequence of backward-steps generated by \cref{alg:aproxgrad}.
Let $F:=f+g$ satisfy the anisotropic proximal gradient dominance condition relative to $\phi$ with constant $0<\mu<L$ for parameter $1/L < \lambda_g$. Let $\lambda = 1/L$. Then $\{F(x^k)\}_{k =0}^\infty$ converges $Q$-linearly to $\inf F$ with rate $(1 - \tfrac\mu L )$, in particular,
\begin{equation*}
F(x^{k+1}) - \inf F \leq (1 - \tfrac\mu L ) (F(x^{k}) - \inf F).
\end{equation*}
\end{theorem}
\begin{proof}
Thanks to \cref{thm:sufficient_descent} we have for any $k \in \bN_0$ that
\begin{align*}
F(x^{k+1}) - \inf F \leq F(x^{k}) - \inf F - \tfrac{1}{L} \gap{\phi}{F}(x^{k}, L^{-1}).
\end{align*}
The anisotropic gradient dominance condition yields that
\begin{equation*}
\begin{aligned}
F(x^{k+1}) - \inf F &\leq F(x^{k}) - \inf F - \tfrac{\mu}{L}(F(x^{k}) - \inf F) \\
&=  (1 - \tfrac\mu L)(F(x^{k}) - \inf F)
\end{aligned}
\end{equation*}
which is the claimed rate.
\end{proof}
Note that for the Bregman proximal gradient method only a $R$-linear convergence result is known \cite{lu2018relatively}.
If $g\equiv 0$ our method specializes to dual space preconditioning for gradient descent \cite{maddison2021dual} for which the $Q$-linear convergence was established under different conditions \cite[Theorem 3.9]{maddison2021dual}. As shown in \cite[Proposition 3.3]{maddison2021dual} these conditions are equivalent to Bregman strong convexity and smoothness of $f^*$ relative to $\phi^*$. As we shall see next this is a special case of our setting if $g\equiv 0$.

Next we show that $F=f+g$ satisfies the anisotropic proximal gradient dominance condition if $f$ is anisotropically strongly convex. Anisotropic strong convexity is defined in terms of the following subgradient inequality which generalizes \cite[Definition 3.8(ii)]{laude2021conjugate} to reference functions $\phi$ which are possibly not super-coercive. Consequently, the gradient mapping $\nabla \phi^*$ does not have full domain in general. Therefore we restrict the subgradient inequality to hold only at pairs $(\bar x, \bar v) \in \gph \partial f$ for which $\bar v \in \intr \dom \phi^*$. In addition we assume that $\ran \partial f \supseteq \ran \nabla \phi$.
\begin{definition}[anisotropic strong convexity]
\ifx\ifsvjour\true
 \label[definition]{def:astrong_convexity}
\else
 \label{def:astrong_convexity}
\fi
Let $f \in \Gamma_0(\bR^n)$ such that $\ran \partial f \supseteq \ran \nabla \phi$. Then we say that $f$ is anisotropically strongly convex relative to $\phi$ with constant $\mu$ if for all $(\bar x, \bar v) \in \gph \partial h \cap (\bR^n \times \intr \dom \phi^*)$ the following inequality holds true
\begin{align}\label{eq:a-strongly}
f(x) &\geq f(\bar x) + \tfrac1 \mu \star \phi(x-\bar x + \mu^{-1} \nabla\phi^*(\bar v)) - \tfrac1 \mu \star \phi(\mu^{-1} \nabla\phi^*(\bar v)), \quad \forall x\in\bR^n.
\end{align}
\end{definition}
\begin{remark}
For $\phi=\tfrac{1}{2}\|\cdot\|^2$ by expanding the square since $\nabla \phi^* = \id$ and $\tfrac1 \mu \star \phi = \frac{\mu}{2}\|\cdot\|^2$ the subgradient inequality \cref{eq:a-strongly} specializes to the classical Euclidean strong convexity inequality, as is the case for the anisotropic descent inequality \cref{rem:euclidean_descent_lemma}:
\begin{align*}
f(x) &\geq f(\bar x) + \tfrac{\mu}{2}\|x-\bar x + \mu^{-1} \bar v\|^2 - \tfrac \mu 2 \|\mu^{-1} \bar v\|^2 = h(\bar x)+\langle \bar v, x - \bar x \rangle + \tfrac{\mu}{2}\|x-\bar x\|^2.
\end{align*}
\end{remark}
Next we show that anisotropic strong convexity of $f$ implies anisotropic proximal gradient dominance of $F=f+g$.
\begin{proposition}[anisotropic strong convexity implies anisotropic proximal gradient dominance]
\ifx\ifsvjour\true
 \label[proposition]{thm:strong_implies_pg_dominance}
\else
 \label{thm:strong_implies_pg_dominance}
\fi

Let $f$ be anisotropically strongly convex relative to $\phi$ with constant $\mu$ and $g \in \Gamma_0(\bR^n)$. Then $\ran \nabla f = \ran \nabla \phi$ and $F:=f+g$ is coercive and strictly convex relative to its effective domain implying that it has a unique minimizer $x^\star= \argmin F$. In particular, $F$ satisfies the anisotropic proximal gradient dominance condition relative to $\phi$ with constant $\mu$ for any parameter $0<\lambda\leq \mu^{-1}$, i.e., 
$$
\mu(F(\bar x) - F(x^\star) ) \leq \gap{\phi}{F}(\bar x, \lambda).
$$
\end{proposition}
A proof is provided in \cref{sec:thm_strong_implies_pg_dominance}

Next we provide a conjugate duality correspondence between anisotropic strong convexity and relative smoothness in the Bregman sense \cite{birnbaum2011distributed,bauschke2017descent,lu2018relatively} generalizing \cite[Theorem 4.3]{laude2021conjugate} to reference functions which are not necessarily super-coercive. This turns out helpful in the subsequent \cref{ex:a_strongly_convex_1,ex:a_strongly_convex_2}.
\begin{proposition}[characterizations of anisotropic strong convexity]
\ifx\ifsvjour\true
 \label[proposition]{thm:conjugate_duality_astrong}
\else
 \label{thm:conjugate_duality_astrong}
\fi
Let $f \in \Gamma_0(\bR^n)$ and $\mu > 0$.
Then the following are equivalent:
\begin{propenum}
\item \label{thm:conjugate_duality_astrong:astrong} $\ran \partial f \supseteq \ran \nabla \phi$ and $f$ satisfies the anisotropic strong convexity inequality for all $(\bar x, \bar v) \in \gph \partial h \cap (\bR^n \times \intr \dom \phi^*)$, i.e., the following inequality holds true
\begin{align}\label{eq:a_strongly_mu1}
f(x) &\geq f(\bar x) +\tfrac{1}{\mu} \star \phi(x-\bar x + \mu^{-1}\nabla\phi^*(\bar v)) - \tfrac{1}{\mu} \star \phi(\mu^{-1}\nabla\phi^*(\bar v)) \quad \forall x\in\bR^n;
\end{align}
\item \label{thm:conjugate_duality_astrong:bsmooth} $\dom f^* \supseteq \dom \phi^*$ and $f^*$ is smooth relative to $\phi^*$ in the Bregman sense with constant $\mu^{-1}$, i.e., $\mu^{-1}\phi^* \mathbin{\dot{-}} f^*$ is convex on $\intr \dom \phi^*$.
\end{propenum}
\end{proposition}
A proof is provided in \Cref{sec:thm:conjugate_duality_astrong}. It follows along the lines of the proof of \cite[Theorem 4.3]{laude2021conjugate} incorporating the constraint qualifications $\ran \partial f \supseteq \ran \nabla \phi$ and $\dom f^* \supseteq \dom \phi^*$, the latter being the standard CQ in Bregman smoothness. A key ingredient for the proof of the result above is the notion of \emph{generalized convexity} introduced in \Cref{sec:generalized_conjugacy}.

The next examples show that the absolute value function and the quadratic function are both anisotropically strongly convex relative to a softmax approximation of the absolute value function:
\begin{example}
\ifx\ifsvjour\true
 \label[example]{ex:a_strongly_convex_1}
\else
 \label{ex:a_strongly_convex_1}
\fi
Let $g(x):=\nu |x|$ with $\nu \geq 1$. Then $g$ is anisotropically strongly convex relative to the symmetrized logistic loss $\phi(x):=2\ln(1+ \exp(x)) - x$ with any constant $\mu>0$.
\end{example}
\begin{proof}
The individual conjugates amount to $g^*(x)=\delta_{[-\nu,\nu]}(x)$ and $\phi^*(x)=(x + 1)\ln((x + 1)/2) + (1-x)\ln((1-x)/2)$ with $x\ln x + (1-x)\ln(1-x)=0$ for $x\in\{0,1\}$ and $\dom\phi^*=[-1, 1]$.
Since $\dom g^* = [-\nu,\nu] \supseteq [-1,1]$ thanks to \cref{thm:conjugate_duality_astrong} anisotropic strong convexity of $g$ is implied by convexity of $\tfrac{1}{\mu}\phi^* \mathbin{\dot{-}} g^*$ on $(-1, 1)$. Since $g^*\equiv 0$ on $(-1, 1)$ this is valid for any $\mu>0$.
\ifx\ifsvjour\true
\qed
\fi
\end{proof}
\begin{remark}[Disclaimer]
It should be noted that for $\nu \geq 1$ \cref{ex:a_strongly_convex_1} is rather of theoretical interest: For the smooth part $f$ we have the restriction $\ran f' \subseteq \intr\dom\phi^*=(-1,1)$ and thus the first-order optimality condition $-f'(x^\star)\in \partial g(x^\star)$ is only valid at $x^\star = 0$.
\end{remark}

\begin{example}
\ifx\ifsvjour\true
 \label[example]{ex:a_strongly_convex_2}
\else
 \label{ex:a_strongly_convex_2}
\fi
Let $g(x):=\tfrac{\nu}{2}x^2$ with $\nu >0$. Then we show that $g$ is anisotropically strongly convex relative to $\phi(x):=2\ln(1+ \exp(x)) - x$ with constant $\mu:=2\nu$.
\end{example}
\begin{proof}
The individual conjugates amount to $g^*(x)=\tfrac{1}{2\nu}x^2$ and $\phi^*(x)=(x + 1)\ln((x + 1)/2) + (1-x)\ln((1-x)/2)$ with $x\ln x + (1-x)\ln(1-x)=0$ for $x\in\{0,1\}$ and $\dom\phi^*=[-1, 1]$.
Since $\dom g^* = \bR^n$ thanks to \cref{thm:conjugate_duality_astrong} anisotropic strong convexity of $g$ is implied by convexity of $\tfrac{1}{\mu}\phi^* \mathbin{\dot{-}} g^*$ on $(-1, 1)$. This is implied by $(\phi^*)''(x)=2/(1-x^2)\geq \tfrac{\mu}{\nu}$ for all $x \in (-1,1)$ which is valid for $\mu=2\nu$.
\ifx\ifsvjour\true
\qed
\fi
\end{proof}
\ifx\ifsvjour\true
Leveraging an equivalent reformulation of our algorithm in terms of a difference of $\Phi$-convex approach examined in
an extended version of this manuscript \cite{laude2022anisotropic} we show the algorithm's invariance under a certain double-min duality. This allows us to derive linear convergence if $g$ in place of $f$ is anisotropically strongly convex.
\else
\section{Difference of \texorpdfstring{$\Phi$}{Phi}-convex approach} \label{sec:phi_dca_and_phi_cvxt}
\subsection{Generalized convexity and generalized conjugacy} \label{sec:generalized_conjugacy}
In this section we shall discuss an equivalent reformulation of our algorithm in terms of a difference of $\Phi$-convex approach and show the algorithm's invariance under a certain double-min duality. This allows us to derive linear convergence if $g$ in place of $f$ is anisotropically strongly convex.

\input{generalized_conjugacy.tex}

\subsection{Difference of \texorpdfstring{$\Phi$}{Phi}-convex approach and transfer of smoothness} \label{sec:dca}
\input{phi_dca.tex}

\fi

\section{Applications} \label{sec:apps}
\subsection{Plus-minus trick} \label{sec:plusminus}
In our algorithm the gradient $\nabla f(x^k)$ has to be a compatible input to the preconditioner $\nabla \phi^*$. However, in some important cases $\nabla \phi^*$ does not have full domain. In particular this is the case if $\phi= \SumExp$ and hence $\nabla \phi^*=\Log$ is the component-wise logarithm with $\dom \nabla \phi^* =\bR_{++}^n$. As we shall demonstrate this is not restrictive by employing the plus-minus trick. Furthermore, this reformulation leads to a generalization of the so-called parallel update algorithm \cite{collins2002logistic}, a powerful approach for exponential and logistic regression.
Henceforth assume $\phi=\SumExp$. In a nutshell the key idea is to decompose the gradient $\nabla f(x)=T_+(x)-T_-(x)$ into a positive part $T_+(x)$ and a negative part $-T_-(x)$ such that $T_+(x),T_-(x) \in \bR_{++}^n$. Applying the logarithmic preconditioner $\Log$ to both components $T_+(x),T_-(x)$ separately, the resulting vector $-R(x)=\Log(T_-(x))-\Log(T_+(x))$ remains a descent direction of the cost $f$ as in \cref{rem:dual_space_pc}: In fact, by strict monotonicity of $\nabla \phi^*=\Log$ we have:
\begin{align}
\langle -R(x), \nabla f (x) \rangle &= \langle \nabla \phi^*(T_-(x))-\nabla \phi^*(T_+(x)), T_+(x) -T_-(x) \rangle < 0.
\end{align}
This decomposition can be framed in terms of the anisotropic proximal gradient method by introducing an additional variable $x_-$ along with the linear sharing constraint $x_- = -x$ enforced via $g$.
For that purpose assume that $f(x)=f_{\pm}(x,x_-)$ such that $f_{\pm}(x,x_-)$ is exponentially smooth jointly in $(x,x_-)$, i.e., anisotropically smooth relative to $\phi(x, x_-):=\SumExp(x)+\SumExp(x_-)$ with constant $L$.
Via the choices $T_+(x)=\nabla_x f_\pm(x, x_-)$ and $T_-(x)=\nabla_{x_-} f_\pm(x, x_-)$ the gradient $\nabla f_\pm(x, -x)=T_+(x)-T_-(x)=\nabla f(x)$ furnishes the sought decomposition into two positive parts.
Then the $x$-update reads $(x^{k+1},x_-^{k+1}) = \aprox[\lambda]{\phi}{g}(y^k, y_-^k)$ for $y^k = x^k - \lambda \Log(T_+(x^k))$ and $y_-^k = x_-^k - \lambda \Log(T_-(x_-^k))$. Thanks to \cref{thm:moreau_decomposition:decomp} it can be expressed in terms of the Bregman projection $\bprox[\lambda]{\phi^*}{g^*}$ for $g^*=\delta_{C^*}$ onto the consensus subspace $C^*=\{(x_1, x_2) \in \bR^n \times \bR^n \mid x_1=x_2\}$ which thanks to \cite[Example 3.16(i)]{BC03} admits a simple closed form solution $(u,u) = \bprox[\lambda]{\phi^*}{g^*}(w, w_-)$ in terms of the geometric mean $u = \Exp((\Log(w) + \Log(w_-)) / 2)$. Overall, this yields $\aprox[\lambda]{\phi}{g}(y,y_-)=((y - y_-)/2, (y_- - y)/2)$ and hence, by eliminating $x_-^k$ and $y^k,y_-^k$ from the algorithm the $x$-update can be written compactly as 
\begin{equation*}
x^{k+1} = x^k - \tfrac{\lambda}{2}(\Log(T_+(x^k))- \Log(T_-(x^k) ).
\end{equation*}
Next we show how to obtain $f_\pm$ from $f$: In many applications $f(x)=h(Ax)+\langle c,x \rangle$ for $h$ being exponentially smooth with constant $1/\sigma$, $A \in \bR^{m \times n}$ and $c \in \bR^n$.
We define $A_+,A_- \in \bR_+^{m \times n}$ as
\begin{align}
(A_+)_{ij} := \begin{cases} A_{ij} & \text{if $A_{ij} \geq 0$} \\
0 & \text{otherwise,}
\end{cases}
\qquad (A_-)_{ij} := \begin{cases} -A_{ij} & \text{if $A_{ij} \leq 0$} \\
0 & \text{otherwise} \end{cases}
\end{align}
and $A_{\pm}=\begin{pmatrix} A_+ & A_-\end{pmatrix}$.
Analogously, we define $c_+, c_-$ such that overall $A_+ - A_- = A$ and $c_+ - c_-= c$.
Then we choose $f_{\pm}(x, x_-):=h(A_+ x + A_- x_-) + \langle x, c_+\rangle + \langle x_-, c_-\rangle$. By adding some $\varepsilon > 0$ to the entries of $A_+, A_-, c_+, c_-$ we can ensure that $A_+, A_-,c_+, c_- \in \bR_{++}^n$ without changing the optimization problem. In light of \cref{thm:exponential_smoothness_penalty}, $(x,x_-)\mapsto h(A_+ x + A_- x_-)$ is exponentially smooth with constant $L=\|A_\pm\|_{\infty}/\sigma$. Thanks to \cref{thm:phi_convex_asmooth} $\langle x, c_+\rangle + \langle x_-, c_-\rangle$ is exponentially smooth for any $L>0$. Invoking \cref{thm:closed_addition_exp} we deduce that $f_\pm$ remains exponentially smooth with constant $L$.
Actually, it suffices to add $\varepsilon > 0$ to the entries of $c_+,c_-$ only. Then we can take $L=\|A\|_{\infty}/\sigma$.
The complete algorithm reads
\begin{align}
T_+(x^k) &= A_+^\top \nabla h(Ax^k) + c_+ \label{eq:parallel_update_p} \\
T_-(x^k)  &= A_-^\top \nabla h(Ax^k) + c_- \label{eq:parallel_update_m} \\
x^{k+1} &= x^k - \tfrac{\lambda}{2}(\Log(T_+(x^k))- \Log(T_-(x^k) ), \label{eq:parallel_update_pm}
\end{align}
for $\lambda = 1/L$.
In the subsequent sections we refer to this algorithm as anisoPG$_{\pm}$ and LS-anisoPG$_\pm^\alpha$ for the linesearch version.
For the special case that $c=0$, $\sigma = 1$, $h=\SumExp$ or $h$ being the logistic loss and $A \in [-1,1]^{m \times n}$ with $\sum_{j=1}^n |A_{ij}| \leq 1$ for every $i$ the smoothness constant and step-size can be chosen to be $\lambda=L=1$ and anisoPG$_{\pm}$ specializes to the parallel update optimization algorithm \cite{collins2002logistic} for AdaBoost \cite{freund1997decision,schapire1998improved} and logistic regression with $q_0=(1,1,\ldots, 1) \in \bR_+^n$.

In the following subsections we present numerical evaluations for diverse problems. The code for reproducing these experiments is publicly available\footnote{\url{https://github.com/EmanuelLaude/anisotropic-proximal-gradient}}.

\begin{figure}[!t]
\centering
\subfloat{
        \centering
	\includegraphics[width=0.32\textwidth]{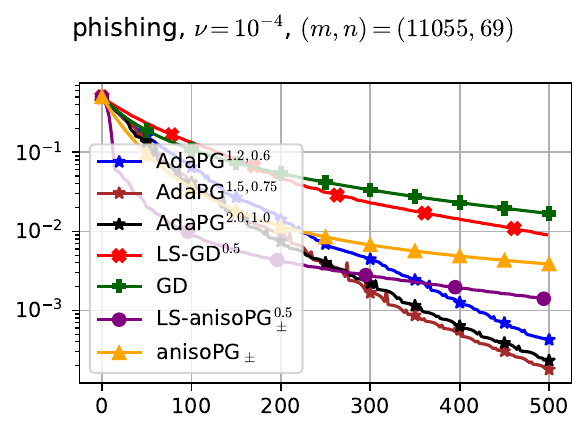}
} 
\subfloat{
        \centering
	\includegraphics[width=0.32\textwidth]{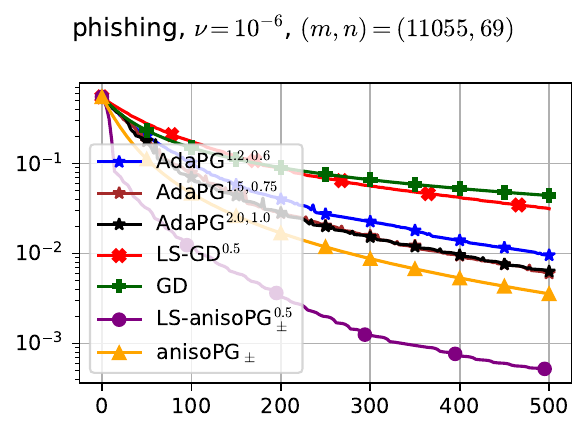}
} 
\subfloat{
        \centering
	\includegraphics[width=0.32\textwidth]{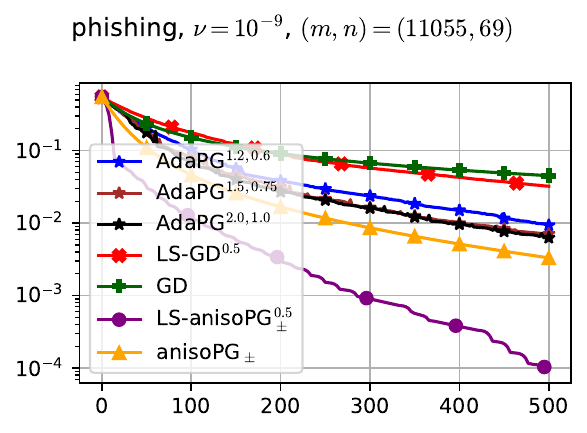}
} \\ 
\subfloat{
        \centering
	\includegraphics[width=0.32\textwidth]{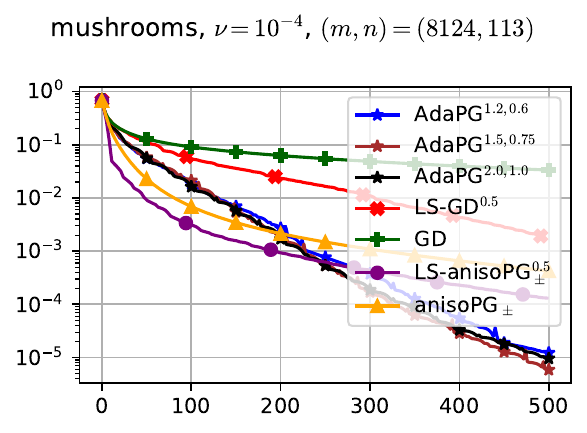}
} 
\subfloat{
        \centering
	\includegraphics[width=0.32\textwidth]{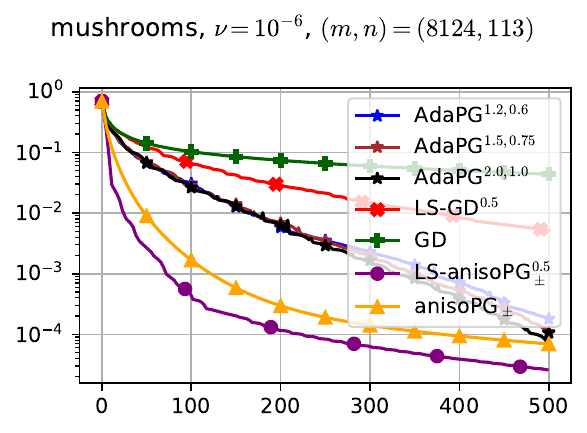}
} 
\subfloat{
        \centering
	\includegraphics[width=0.32\textwidth]{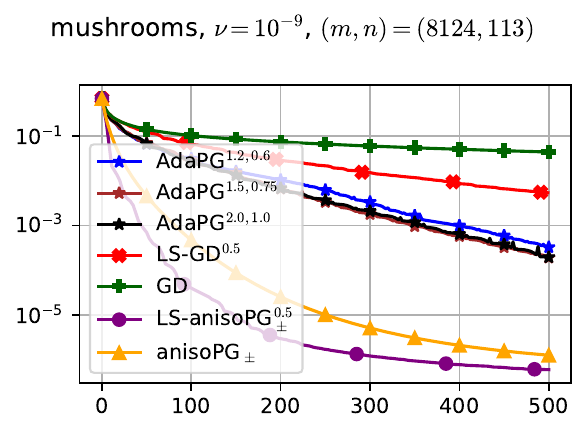}
} \\ 
\subfloat{
        \centering
	\includegraphics[width=0.32\textwidth]{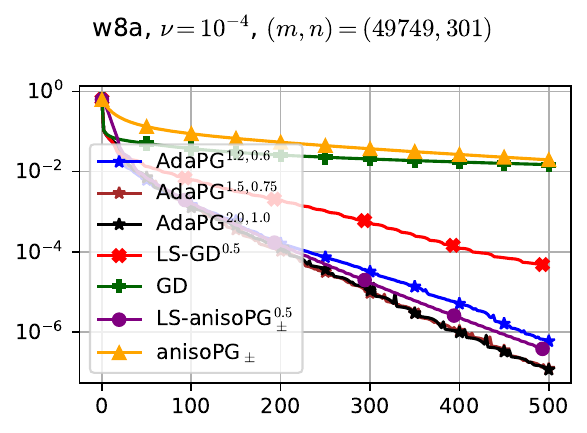}
} 
\subfloat{
        \centering
	\includegraphics[width=0.32\textwidth]{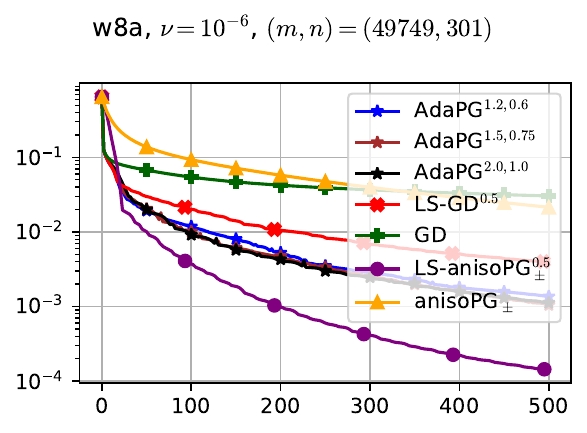}
} 
\subfloat{
        \centering
	\includegraphics[width=0.32\textwidth]{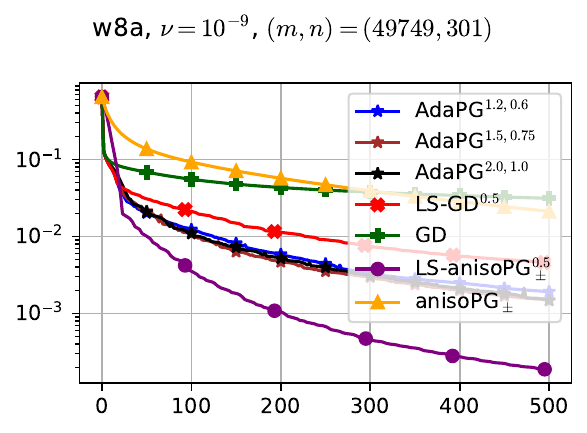}
} \\ 
\subfloat{
        \centering
	\includegraphics[width=0.32\textwidth]{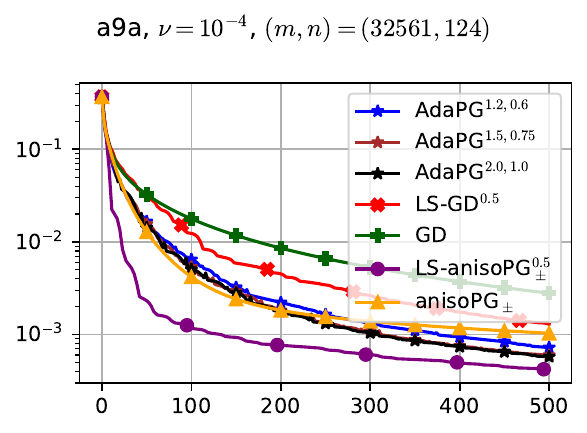}
} 
\subfloat{
        \centering
	\includegraphics[width=0.32\textwidth]{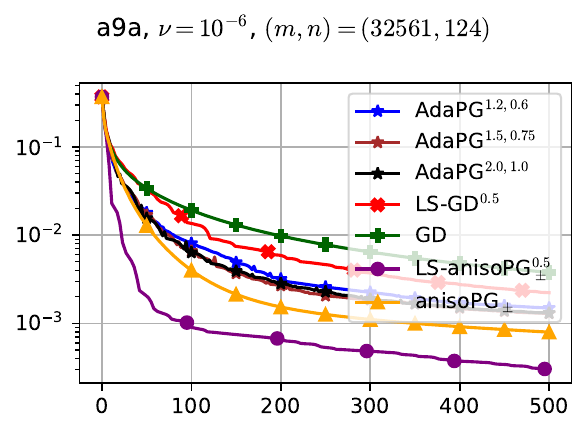}
} 
\subfloat{
        \centering
	\includegraphics[width=0.32\textwidth]{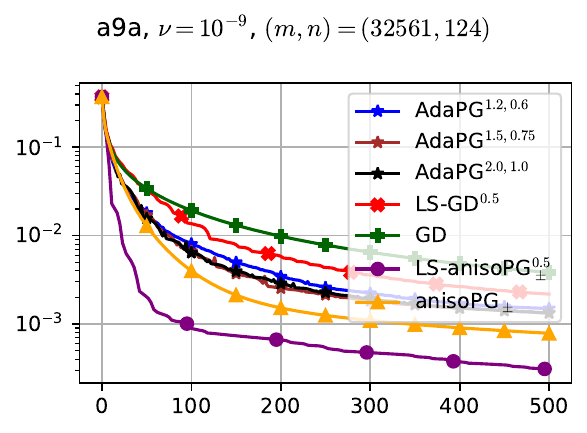}
} \\ 
\subfloat{
        \centering
	\includegraphics[width=0.32\textwidth]{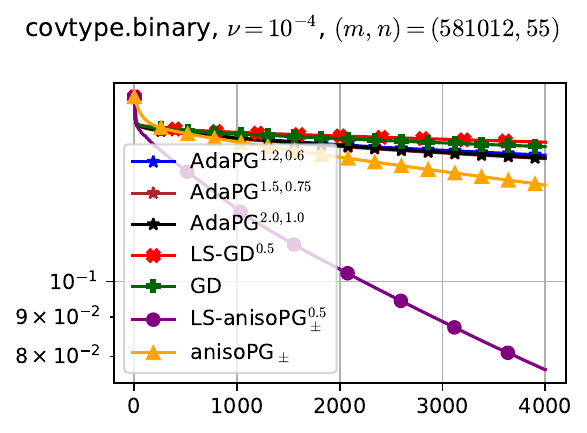}
} 
\subfloat{
        \centering
	\includegraphics[width=0.32\textwidth]{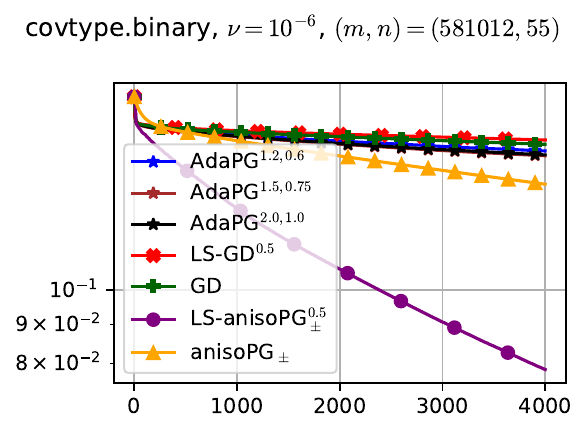}
} 
\subfloat{
        \centering
	\includegraphics[width=0.32\textwidth]{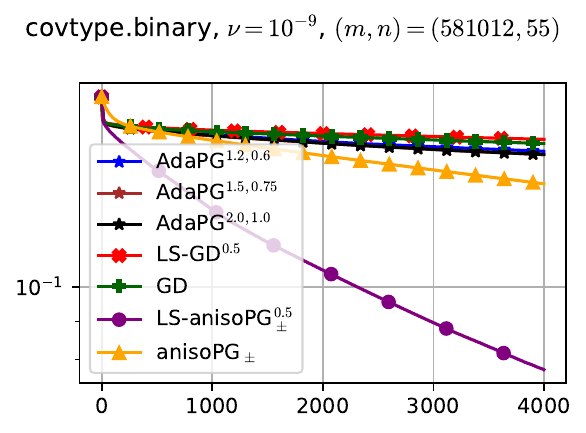}
} 
\caption{Quadratically regularized logistic regression. On the $y$-axis we plot the function gap $F(x) - F(x^\star)$ and on the $x$-axis the number of calls to $A,A^\top$.}
\label{fig:logistic}
\end{figure}
\begin{figure}[!t]
\centering
\subfloat[\label{fig:logistic_comparison_a}]{
        \centering
	\includegraphics[width=0.49\textwidth]{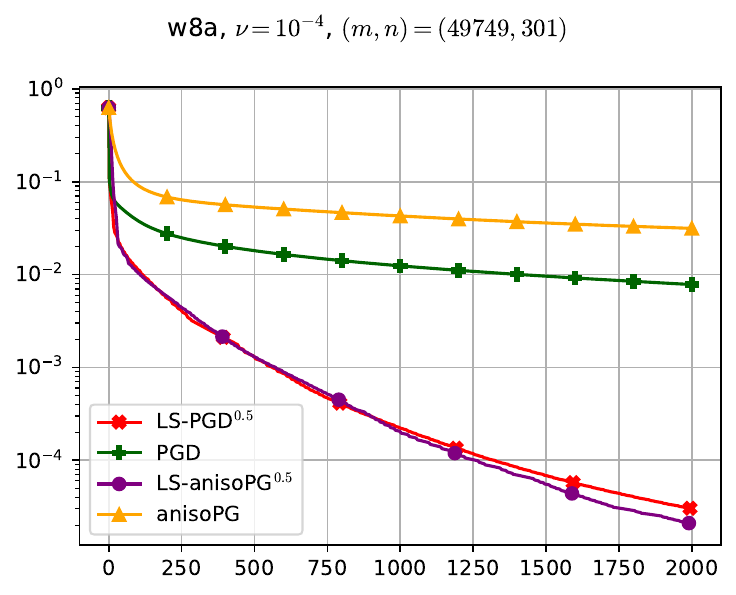}
} 
\subfloat[\label{fig:logistic_comparison_b}]{
        \centering
	\includegraphics[width=0.49\textwidth]{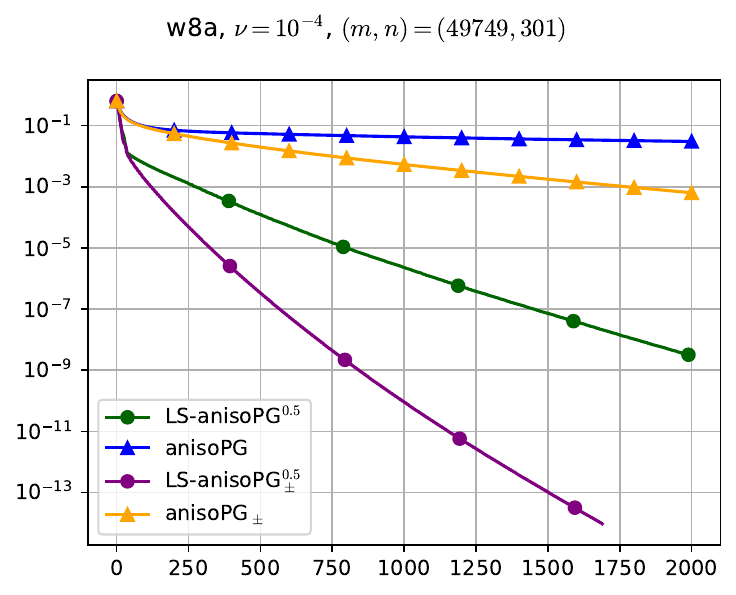}
} 
\caption{(a) Comparison of Euclidean methods with anisoPG using the symmetrized logistic reference function on the task of $1$-norm regularized logistic regression. (b) Comparison of two variants of anisoPG (with linesearch): exponential smoothness with plus-minus decomposition anisoPG$_\pm$ vs symmetrized logistic reference function anisoPG. On the $y$-axis we plot the function gap $F(x) - F(x^\star)$ and on the $x$-axis the number of calls to $A,A^\top$.}
\label{fig:logistic_comparison}
\end{figure}
\subsection{Regularized logistic regression}
In this subsection we consider regularized logistic regression. Here, one is interested in the minimization of the following cost function:
\begin{align} \label{eq:logistic_regression}
\min_{x \in \bR^n} \frac{1}{m} \sum_{i=1}^m \ln(1+\exp( -b_i \langle a_i, x \rangle)) + R(x),
\end{align}
for $m$ input output training pairs $(a_i, b_i) \in [-1, 1]^n \times \{-1, 1\}$ and $R$ is a possibly nonsmooth regularizer. First we choose $R(x)=\frac{\nu}{2}\|x\|^2$.
It is easy to see that $f_i(x):=\ln(1+\exp( -b_i \langle a_i, x \rangle)) = \logsumexp(0,  \langle -b_i a_i, x \rangle)$. Using \cref{thm:exponential_smoothness_penalty,ex:logsumexp,thm:closed_addition_exp} and the plus-minus trick described above we can write the loss function $f(x)=\frac{1}{m}\sum_{i=1}^m f_i(x)$ as $f_{\pm}(x,x_-)$ which is exponentially smooth with constant $L=\|A\|_{\infty}$. Invoking the polylogarithmic plus-minus decomposition of $\frac{1}{2}\|\cdot\|^2$, \cref{ex:polylog}, the quadratic regularizer can be written as $\nu \Theta(x) + \nu \Theta(x_-)$ which is exponentially smooth in the variables $(x,x_-)$ with constant $1$. Hence the sum $f_{\pm}(x,x_-) + \nu \Theta(x) + \nu \Theta(x_-)$ is a plus-minus decomposition of the regularized loss. In particular it is exponentially smooth with constant $L=\max\{1, \|A\|_{\infty}\}$.
In \cref{fig:logistic} we compare anisotropic proximal gradient with constant step-size (anisoPG$_\pm$) and anisotropic proximal gradient with linesearch for $\alpha=0.5$ (LS-anisoPG$_\pm^{0.5}$) to standard gradient descent (GD), gradient descent with linesearch (LS-GD$^{0.5}$) and a recent family of linesearch-free adaptive proximal gradient methods \cite{latafat2023convergence,malitsky2020adaptive} AdaPG$^{q,r}$ with parameters $q$ and $r$ using the publicly available LIBSVM datasets\footnote{\url{https://www.csie.ntu.edu.tw/~cjlin/libsvmtools/datasets/}} for classification.
We implement LS-GD$^\alpha$ as an instance of LS-anisoPG$^\alpha$ with $\phi=\frac{1}{2}\|\cdot\|^2$ and Lipschitz constant $\mathrm{lip}=\|A\|_2^2/4 + \nu$. We choose $\lambda_{\mathrm{-1}}=1/L$ for the exponential case and $\lambda_{\mathrm{-1}}=1.99/{\mathrm{lip}}$ for the Euclidean case and choose $\alpha=0.5$ in both cases.
For a fair comparison between the linesearch-based algorithms and the methods without linesearch we count the number of calls to $A,A^\top$ which in our setup dominate the computational burden. On the $y$-axis we plot the suboptimality gap $F(x^k)-F(x^\star)$ where $x^\star$ is a high-accuracy solution precomputed with L-BFGS.
It can be seen that the anisotropic variants almost consistently perform better than their Euclidean counterparts wrt. the number of calls to $A,A^\top$. In particular for small regularization weights LS-anisoPG$_\pm^{0.5}$ achieves substantially better performance than the state-of-the-art linesearch-free adaptive proximal gradient methods \cite{latafat2023convergence,malitsky2020adaptive} AdaPG$^{q,r}$. Furthermore, LS-anisoPG$_\pm^{0.5}$ performs significantly better than the version without linesearch (anisoPG$_\pm$) which in some cases also performs better than AdaPG$^{q,r}$. LS-anisoPG$_\pm^{0.5}$ can be seen as a linesearch extension of the parallel update algorithm \cite{collins2002logistic}.

Thanks to \cref{ex:logistic} $f(x)=\tfrac{1}{m} \sum_{i=1}^m \ln(1+\exp( -b_i \langle a_i, x\rangle))$ is also anisotropically smooth relative to the symmetrized logistic loss $\phi(x)=\sum_{j=1}^n h(x_j)$ for $h(t)=2\ln(1+ \exp(t)) - t$ with constant $\|A\|_{\infty, 2}$.
A potential advantage of this alternative reference function is that the nonsmooth part $g$ does not involve the sharing constraint as in the plus-minus decomposition. As such it allows for a broader choice of regularizers: We discuss the following choices of regularizers $g:=R=\nu\|\cdot\|_1$ and $g:=R=\tfrac{\nu}{2}\|\cdot\|^2$. For both choices it holds that $\dom g^*\cap \intr \dom \phi_-^* \neq \emptyset$ and thus the algorithm can be applied.
For $g=\nu\|\cdot\|_1$ a closed form solution to the anisotropic proximal mapping is obtained via the Bregman--Moreau decomposition \cref{thm:moreau_decomposition:decomp}. Then the anisotropic proximal mapping can be computed in terms of the Bregman proximal mapping of $g^*=\delta_{[-\nu,\nu]^n}$ whose solution involves a simple clipping operation. The anisotropic proximal mapping of $g$ thus amounts to the following soft-thresholding operation:
$
[\aprox[\lambda]{\phi_-}{g}(y)]_i 
= \sign(y_i)\max\{|y_i|- \rho, 0\},
$
for $\rho = \lambda (h^*)'(\nu)$.
For $g=\tfrac{\nu}{2}\|\cdot\|^2$ by a change of variable the anisotropic proximal mapping of $g$ can be written as the Euclidean proximal mapping of the logistic loss $x \mapsto \ln(1+ \exp(x))$. Thanks to \cite[Proposition 2]{briceno2019random} its solution can be obtained in closed form via the generalized Lambert $W$ function \cite{mezHo2017generalization,maignan2016fleshing}.
\begin{figure}[!t]
\centering
\subfloat[\label{fig:logsumpexp_sumexp_a}]{
        \centering
	\includegraphics[width=0.49\textwidth]{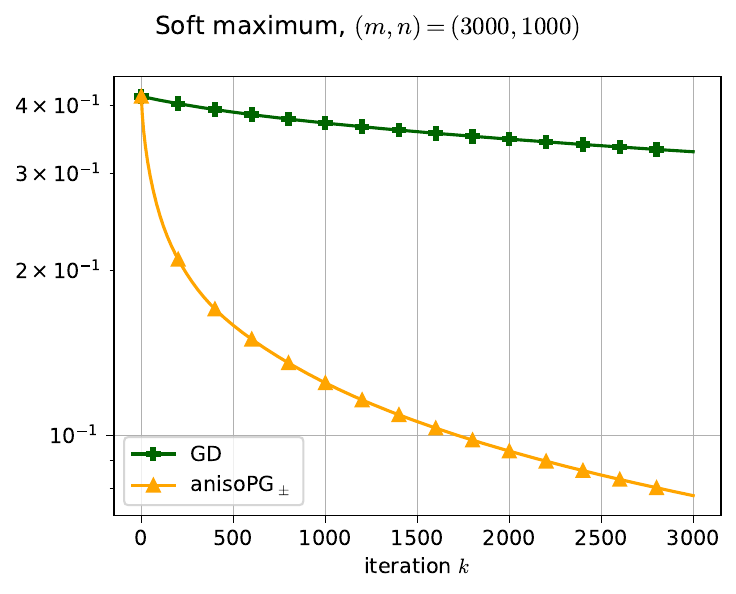}
} 
\subfloat[\label{fig:logsumpexp_sumexp_b}]{
        \centering
	\includegraphics[width=0.49\textwidth]{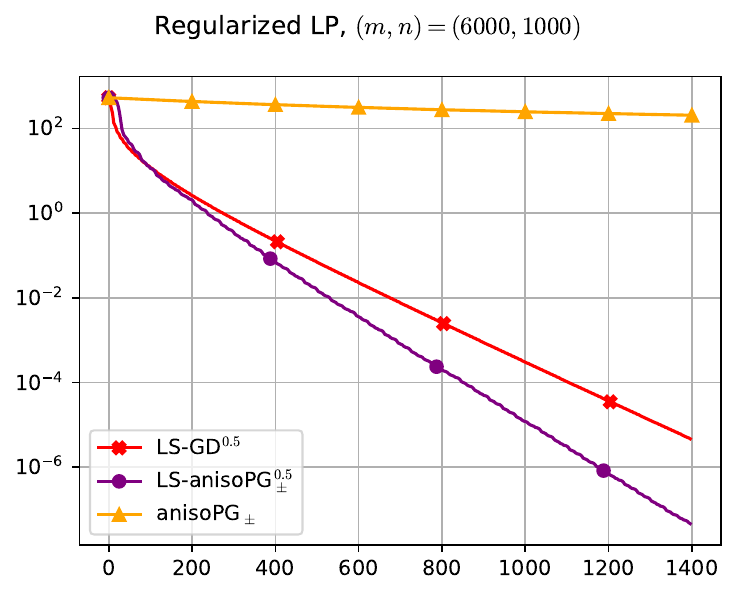}
} 
\caption{(a) Soft maximum with $\logsumexp$. (b) exponentially regularized LP. On the $y$-axis we plot the function gap $F(x) - F(x^\star)$.}
\label{fig:logsumpexp_sumexp}
\end{figure}

In \cref{fig:logistic_comparison_a} we compare the anisotropic proximal gradient method as described above with the Euclidean counterparts, i.e., \cref{alg:linesearch_aproxgrad} with $\phi=\frac{1}{2}\|\cdot\|^2$ on the task of $\ell_1$-regularized logistic regression, i.e., $R=\nu\|\cdot\|_1$. It can be seen that LS-anisoPG$^{0.5}$ with the symmetrized logistic reference function performs on par with the Euclidean proximal gradient method with linesearch LS-PGD$^{0.5}$. In \cref{fig:logistic_comparison_b} we compare the performance of anisoPG with the exponential reference function and the plus-minus decomposition (LS-anisoPG$_\pm^{0.5}$, anisoPG$_\pm$) against the versions with the symmetrized logistic reference function for $R=\frac\nu2\|\cdot\|_2^2$. It can be seen that anisoPG with the exponential reference function achieves superior performance.

\subsection{Soft maximum}
In this subsection we consider soft maximum problems that take the form
\begin{align}
\min_{x \in \bR^n} \sigma \star \logsumexp(Ax -b),
\end{align}
for a linear operator $A \in \bR^{m \times n}$ and vector $b\in \bR^m$. We generate a random matrix that has condition number $150$ and norm $\|A\|_2 = 100$. By the calculus of exponential smoothness the plus-minus decomposition of the cost is exponentially smooth with constant $L=\|A\|_{\infty}/\sigma$. In \cref{fig:logsumpexp_sumexp_a} we compare vanilla anisotropic proximal gradient without linesearch with stepsize $1/L$ against its Euclidean counterpart with stepsize $1.99/\mathrm{lip}$ with $\mathrm{lip}=\|A\|_2^2/\sigma$ showing significant improvements. Althought it is hard to compare the update directions of both algorithms note the significant difference between the smoothness constants $L=19,449.51$ and $\mathrm{lip}=200,000.00$.

\subsection{Exponential regularization for linear programs}
In this subsection we consider exponential smoothing for linear programming. Given a \emph{linear program} (LP) in standard dual form with inequality constraints $Ax - b \leq 0$ we consider an unconstrained approximation to the original LP which takes the form
\begin{align} \label{eq:LP_exp}
\min_{x \in \bR^n}~ \langle c, x \rangle + \sigma \star \SumExp(Ax -b),
\end{align}
for $A \in \bR^{m \times n}$, $c \in \bR^n$ and $b \in \bR^m$ and $\sigma >0$ is a regularization parameter. The function $\SumExp(Ax-b)$ can be interpreted as a penalty function for the constraints. Exponential smoothing is dual to entropic regularization which is a common technique for solving large-scale optimal transport problems \cite{cuturi2013sinkhorn}. Problems of this form also arise as subproblems in the exponential multiplier method \cite{tseng1993convergence} aka the Bregman augmented Lagrangian method \cite{eckstein2003practical,solodov2000inexact} when applied to linear programs. Furthermore, the exponential loss function appears in boosting for machine learning \cite{freund1997decision,schapire1998improved}.

In \cref{fig:logsumpexp_sumexp_b} we compare anisotropic proximal gradient (with linesearch) with the exponential reference function to Euclidean gradient descent with linesearch on minimizing the regularized LP cost. We generate a random LP in dual form such that $A\in \bR^{6000 \times 1000}$ has condition number $10$ and $\|A\|_2=1$. We choose $\sigma=1$. Although $\SumExp$ is not globally Lipschitz smooth, assuming that the iterates stay in a compact set we can run linesearch gradient descent as an instance of LS-anisoPG$^\alpha$ with $\phi=\frac{1}{2}\|\cdot\|^2$. In the exponential case we choose $\lambda_{-1}=1/\|A\|_{\infty}$ as the inverse sharp smoothness constant. In the Euclidean case we choose $\lambda_{-1}=100$. 
Besides the existence of a global smoothness constant $L$ in the exponential case which allows us to terminate the linesearch loop safely at $\lambda_k \leq 1/L$, we highlight that the anisotropic proximal gradient algorithm also overcomes a computational disadvantage of its Euclidean counterpart: Note that for small values of $\sigma$ the gradient of the exponential penalties $\sigma \SumExp((Ax -b)/\sigma)$ can suffer from numerical overflows. This is avoided by the logarithmic preconditioner that wraps the update direction within a component-wise logarithm: Then the components of the update direction take the form of a weighted $\logsumexp$ which can be evaluated robustly even for large values by noting the identity $\logsumexp_{\alpha}(x) = \ln(\sum_{i=1}^n \alpha_i \exp(x_i - \vecmax(x)) + \vecmax(x)$.

\section{Conclusion}
In this paper we have considered dual space nonlinear preconditioning for forward-backward splitting. The algorithm is an extension of dual space preconditioning for gradient descent \cite{maddison2021dual} to the nonconvex and composite setting.
In our case, the method is derived via an anisotropic descent inequality relative to a reference function 
whose inverse gradient takes the role of a nonlinear preconditioner in the proximal gradient scheme.
We develop a calculus for anisotropic smoothness and show invariances which allow one to apply the algorithm to a class problems that are not amenable to the Bregman proximal gradient method.
We prove subsequential convergence of the method using a regularized gap function which vanishes at rate $\mathcal{O}(1/K)$ and we analyze the algorithm's linear convergence under an anisotropic generalization of the proximal PL-inequality \cite{karimi2016linear}.
The method generalizes existing classical algorithms for optimal transport, logistic regression and boosting. This provides a dual view onto the framework of entropic subspace projections that is typically used to derive these algorithms.
We also discuss examples which go beyond these existing methods. In numerical simulations we show substantial improvements of the anisotropic proximal gradient algorithm over its Euclidean counterparts. Our experiments on logistic regression demonstrate that the performance of our methods is almost unaffected by the regularization weight when compared to a recent family of linesearch-free adaptive proximal gradient methods \cite{latafat2023convergence,malitsky2020adaptive} that perform favorably only for larger regularization weight.
An interesting open problem is the complexity of the method in convex case.
More recently, the algorithm has been generalized to Banach spaces \cite{kim2023mirror} and adapted to measure spaces in continuous time \cite{bonet2024mirror}.

\subsection*{Acknowledgements} We would like to thank three anonymous referees for their helpful comments.
%


%% file: generalized_conjugacy.tex
\ifx\ifsvjour\true
  We
\else
  First we
\fi
provide a self-contained overview of generalized conjugacy and generalized convexity \cite{moreau1970inf}; also see \cite{balder1977extension,dolecki1978convexity}.
Classically, these notions appear in the context of eliminating duality gaps in nonconvex and nonsmooth optimization or optimal transport theory; see, e.g., \cite{rockafellar1974augmented,penot1990strongly,Vil08,bauermeister2021lifting}. In our case these notions are used heavily in some proofs.
\ifx\ifsvjour\true\else
In particular, they give rise to a difference of $\Phi$-convex approach interpretation of our algorithm discussed in the next subsection.
\fi
For the remainder of this section let $X$ and $Y$ be nonempty sets and $\Phi: X \times Y \to \bR$ a real-valued coupling. Let $f: X \to \exR$.
\begin{definition}
We say that $f$ is $\Phi$-convex if there exists $\xi : Y \to \exR$ such that $f(x) = \sup_{y \in Y} \Phi(x, y) - \xi(y)$. In addition, we define the $\Phi$-conjugate of $f$ at $y \in Y$ and $\Phi$-bionjugate of $f$ at $x \in X$ as
$$
f^\Phi(y) = \sup_{x \in X} \Phi(x,y) - f(x) \quad \text{and} \quad f^{\Phi\Phi}(x) = \sup_{y \in Y} \Phi(x,y) - f^\Phi(y).
$$
We say that $y \in Y$ is a $\Phi$-subgradient of $f$ at $\bar x \in X$ if
\begin{align}
f(x) \geq f(\bar x) + \Phi(x, y) - \Phi(\bar x, y) \quad \forall x \in X,
\end{align}
or in other words $\bar x \in \argmax_{x \in X}\{ \Phi(x, y) -f(x)\}$. We call the set $\partial_\Phi f(\bar x)$ of all such $y \in Y$ the $\Phi$-subdifferential of $f$ at $\bar x$.
\end{definition}

We have the following result:
\begin{proposition}
\ifx\ifsvjour\true
 \label[proposition]{thm:phi_envelope_equivalence}
\else
 \label{thm:phi_envelope_equivalence}
\fi
$f$ is $\Phi$-convex on $X$ if and only if $f(x) = f^{\Phi\Phi}(x)$ for all $x \in X$.
\end{proposition}

The following equivalent statements are standard in literature, see, e.g., \cite{balder1977extension,dolecki1978convexity} and references therein:
\begin{proposition}
\ifx\ifsvjour\true
 \label[proposition]{thm:phi_subgradients}
\else
 \label{thm:phi_subgradients}
\fi
For any $\bar x \in X$ and $\bar y \in Y$ the following statements are equivalent:
\begin{propenum}
\item $\bar y \in \partial_\Phi f(\bar x)$;
\item $f(\bar x) + f^\Phi(\bar y) = \Phi(\bar x, \bar y)$;
\item $\bar x \in \argmin_{x \in X} \;\{f(x) - \Phi(x, \bar y)\}$;
\end{propenum}
where any of the above equivalent statements implies that $f(\bar x) = f^{\Phi\Phi}(\bar x)$ and $\bar x \in \partial_\Phi f^\Phi(\bar y)$. If, in addition, $f$ is $\Phi$-convex, any of the above statements is equivalent to $\bar x \in \partial_\Phi f^\Phi(\bar y)$.
\end{proposition}

As illustrated in the following remark $\Phi$-convexity of $-f$ is equivalent to the existence of an infimal convolution expression of $f$:
\begin{remark}[$\Phi$-convexity and infimal convolutions]
\ifx\ifsvjour\true
 \label[remark]{thm:inf_conv_rem}
\else
 \label{thm:inf_conv_rem}
\fi
Choose $X=Y=\bR^n$ and the coupling 
\begin{align}
\Phi(x,y):=-\lambda \star \phi(x-y),
\end{align}
for $\lambda :=L^{-1}$.
Let $f:\bR^n \to \bR$. By definition $\Phi$-convexity of $-f$ means that there is $\xi:\bR^n \to \exR$ such that
$$
f(x)=-\sup_{y \in \bR^n} -\lambda \star \phi(x-y) - \xi(y) = \inf_{y \in \bR^n} \lambda \star \phi(x-y) + \xi(y) = (\xi \infconv \lambda \star \phi)(x).
$$
\end{remark}

%% file: phi_dca.tex
In the Euclidean case forward-backward splitting is invariant under transfer of strong convexity: More precisely, forward-backward splitting applied to $F=f+g$ produces the same sequence of iterates $\{x^k\}_{k =0}^\infty$ as forward-backward splitting applied to $F=(f+\frac{\mu}{2}\|\cdot\|^2) + (g - \frac{\mu}{2}\|\cdot\|^2)$. If $g$ in place of $f$ is strongly convex this transformation yields an equivalent splitting in which the smooth part $f+\frac{\mu}{2}\|\cdot\|^2$ is $(L+\mu)$-Lipschitz smooth and $\mu$-strongly convex and $g - \frac{\mu}{2}\|\cdot\|^2$ is convex. Thus \cref{thm:linear_convergence,thm:strong_implies_pg_dominance} can be invoked to prove linear convergence.
In the anisotropic case, however, this is no longer true as the following counterexample reveals:
\begin{example} \label{ex:counter_example_transfer_str_cvx}
Let $f(x)=ax$ with $a \in \bR_{++}$ and $g=\phi=\exp$. Then $f$ is anisotropically smooth relative to $\phi$ and $g$ is anisotropically strongly convex relative to $\phi$. However, $f+\phi$ is not anisotropically strongly convex relative to $\phi$ since $(a, +\infty) = \ran (f + \phi)' \subset \ran \phi' = (0, \infty)$.
\end{example}

Although a primal transfer of strong convexity fails in general, a similar property can be derived for a certain ``dual'' problem. To this end we provide an equivalent interpretation of the algorithm in terms of a \emph{difference of $\Phi$-convex approach} and show that the algorithm is invariant under a certain double-min duality (DC-duality). This allows us to transfer smoothness from $f$ to $g$ by epi-graphical addition and subtraction of $\phi$. Then it can be shown that the DC-dual problem has the anisotropic proximal gradient dominance condition and the algorithm attains linear convergence.

In view of \cref{thm:inf_conv_rem,thm:phi_convex_asmooth,thm:episcaling_asmoothness} anisotropic smoothness of $f$ relative to $\phi$ with constant $L$ implies that $-f$ is $\Phi$-convex relative to $\Phi(x,y)=-\lambda \star \phi(x-y)$ for any $\lambda \leq 1/L$ and via \cref{thm:phi_subgradients} that $(-f)(x)=(-f)^{\Phi\Phi}(x)=\sup_{y \in \bR^n} -\lambda \star\phi(x-y) -(-f)^{\Phi}(y)$. By invoking the notion of the $\Phi$-subdifferential this allows us to provide an alternative interpretation of our algorithm in terms of a difference of $\Phi$-convex approach applied to $F=g - (-f)$:
We first show the following lemma which illustrates that the anisotropic descent inequality can be understood in terms of a $\Phi$-subgradient inequality due to the following correspondence between $\Phi$-subgradients and classical gradients.
\begin{lemma}[correspondence between $\Phi$-subgradients and gradients] \label{thm:phi_subgradient_gradient_correspondence}
Let $f \in \mathcal{C}^1(\bR^n)$ and let $L>0$. Define $\Phi(x,y):=-\lambda\star \phi(x-y)$ for $\lambda:=L^{-1}$. Assume that $\ran \nabla f \subseteq \ran \nabla \phi$.
Then the following statements are equivalent:
\begin{propenum}
\item \label{thm:phi_subgradient_gradient_correspondence:asmooth} $f$ satisfies the anisotropic descent inequality relative to $\phi$ with constant $L$;
\item \label{thm:phi_subgradient_gradient_correspondence:phi_subgrad} $\partial_\Phi (-f)= \id - \lambda\nabla \phi^* \circ \nabla f$.
\end{propenum}
In particular, any of the above equivalent statements implies that $\partial_\Phi (-f)$ is single-valued.
\end{lemma}
A proof can be found in \cref{sec:thm_phi_subgradient_gradient_correspondence}.

Unless stated otherwise for the remainder of this section we assume validity of \cref{assum:a1,assum:a2,assum:a3,assum:a4,assum:a5}. Furthermore recall the definition of the threshold of anisotropic prox-boundedness $\lambda_g$ and assume that $\lambda < \lambda_g$.
Thanks to \cref{thm:phi_subgradient_gradient_correspondence,thm:phi_subgradients} the updates in \cref{alg:aproxgrad} can be rewritten:
\begin{align} 
y^k &= x^k - \lambda\nabla\phi^*(\nabla f(x^k)) \in \partial_\Phi(-f)(x^k) \label{eq:forward_phi} \\
x^{k+1} &\subseteq \argmin_{x \in \bR^n} \;\{\lambda\star \phi(x - y^k) + g(x)\}= \partial_\Phi g^\Phi(y^k) \label{eq:backward_phi}.
\end{align}
Up to replacing the classical subdifferential with the $\Phi$-subdifferential this algorithm exactly resembles the structure of the classical \emph{difference of convex approach} (DCA).

Next we shall discuss a double-min duality also called DC-duality in the classical DCA; see, e.g., \cite{tao1997convex}. In particular we show that the algorithm is invariant under this duality.
Based on the identity $-f = (-f)^{\Phi\Phi}$ we can rewrite \cref{eq:opt_prob} as:
\begin{align}
\inf_{x \in \bR^n} ~g(x) + f(x) &= \inf_{x \in \bR^n} ~g(x) -(-f)^{\Phi\Phi}(x)\notag \\
&= \inf_{x \in \bR^n}  ~g(x) - \sup_{y \in \bR^n} -\lambda \star\phi(x-y) -(-f)^{\Phi}(y)\notag \\
&= \inf_{x \in \bR^n} \inf_{y \in \bR^n} ~g(x) + \lambda \star\phi(x-y) +(-f)^{\Phi}(y)\notag \\
&= \inf_{y \in \bR^n}\inf_{x \in \bR^n}  ~g(x) + \lambda \star\phi(x-y) +(-f)^{\Phi}(y)\label{eq:double_min} \\
&= \inf_{y \in \bR^n} (-f)^{\Phi}(y) -\sup_{x \in \bR^n} -\lambda \star \phi(x-y)-g(x)\notag \\
&= \inf_{y \in \bR^n} (-f)^{\Phi}(y) -g^\Phi(y) \label{eq:dc_dual_plain},
\end{align}
where reversing the order of minimization is also called double-min duality; see \cite[Theorem 11.67]{RoWe98}. In the spirit of classical DC-duality we refer to \cref{eq:dc_dual_plain} as the DC-dual problem of \cref{eq:opt_prob}.

As shown above anisotropic smoothness means that $-f$ is $\Phi$-convex and equivalently $f=-(-f)^{\Phi\Phi}=(-f)^{\Phi}\infconv \lambda \star \phi$ is an infimal convolution. For the corresponding conjugate transform we have
\begin{align}
(-f)^{\Phi}
=-(-f)\infconv\lambda \star \phi_-= \sup_{u \in \bR^n} f(\cdot + u) - \lambda \star \phi(u)=:f \infdeconv \lambda \star \phi,
\end{align}
which is also called the \emph{infimal deconvolution} of $f$ and $\phi$ due to \cite{hiriart1986formulations}.

If, in addition, $f \in \Gamma_0(\bR^n)$ thanks to Hiriart-Urruty \cite[Theorem 2.2]{hiriart1986general} the infimal deconvolution admits the form
\begin{align}
(-f)^\Phi = f \infdeconv \phi = (f^* \mathbin{\dot{-}} \phi^*)^*,
\end{align}
which is thus convex, proper and lsc.

Further rewriting $-g^\Phi$ respectively $(-f)^{\Phi}$ in terms of the infimal convolution respectively infimal deconvolution the DC-dual problem of \cref{eq:opt_prob} is rewritten as
\begin{align}
\text{minimize}~ G:= 
(g \infconv \lambda \star \phi_-) + (f \infdeconv \lambda \star \phi) \tag{D} \label{eq:dc_dual}.
\end{align}
The epi-graphical addition and subtraction of the reference function leads to a transfer of smoothness in the DC-dual and corresponds to pointwise addition and subtraction of the conjugate reference function in the Fenchel--Rockafellar dual. In the Euclidean case this is also related to primal and dual regularization in classical DC-programming \cite[Section 5.2]{tao1997convex}.

Next we show that up to interchanging the roles of forward- and backward-step our algorithm is invariant under double-min duality as is known for the classical DCA. The following result is a generalization of \cite[Theorem 3.8]{laude2021lower} for the Euclidean case to the anisotropic case:
\begin{proposition}[interchange of forward- and backward-step] \label{thm:invariance_double_min}
Let $g$ be convex. Then $g \infconv \lambda \star \phi_-$ is anisotropically smooth relative to $\phi_-$ with constant $\lambda^{-1}$ and for any $y \in \bR^n$ a backward-step on $g$ at $y$ is a forward-step on $g \infconv \lambda \star \phi_-$ at $y$, i.e.,
$$
y - \lambda \nabla \phi_-^*(\nabla (g \infconv \lambda \star \phi_-)(y)) = \aprox[\lambda]{\phi_-}{g}(y).
$$
For any $x \in \bR^n$ a forward-step on $f$ at $x$ is backward-step on $f \infdeconv \lambda \star \phi$ at $x$, i.e.,
$$
\aprox[\lambda]{\phi}{f \infdeconv \lambda \star \phi}(x) = x - \lambda \nabla \phi^*(\nabla f(x)).
$$
If, in addition, $f$ is convex we have $f \infdeconv \lambda \star \phi \in \Gamma_0(\bR^n)$ with $\dom (f \infdeconv \lambda \star \phi)^* \cap \intr \dom\phi^* \neq \emptyset$, i.e., the constraint qualification for the anisotropic proximal mapping in the convex case is satisfied.
\end{proposition}
A proof can be found in \cref{sec:thm_invariance_double_min}.

Alternatively, our algorithm can be understood in terms of Gauss--Seidel minimization applied to \cref{eq:double_min} for which the same self-duality in the DC-sense is true. Next, we exploit the transfer of smoothness in combination with the self-duality to show linear convergence wrt $G$ if $g$ is anisotropically strongly convex.
\begin{proposition}[transfer of smoothness] \label{thm:strongly_convex}
Let $g$ be anisotropically strongly convex with constant $\mu$ relative to $\phi_-$. Then $g \infconv \frac{1}{L} \star \phi_-$ is anisotropically strongly convex with constant $\sigma = 1/(L^{-1}+\mu^{-1}) < \mu$ and has the anisotropic descent property with constant $L$.
\end{proposition}
A proof can be found in \cref{sec:thm_strongly_convex}.

Thanks to the self-duality of our algorithm in the DC-sense and the previous proposition it can be shown that the method converges linearly wrt to the DC-dual cost $G=(g\infconv \lambda \star \phi_-) + (f \infdeconv \lambda \star \phi)$ evaluated at the forward-steps $y^k$:
\begin{proposition} \label{thm:linear_convergence_dual}
Let $f$ be convex and $g$ be anisotropically strongly convex with constant $\mu$ relative to $\phi_-$. Let $\{y^k\}_{k =0}^\infty$ be the sequence of forward-steps generated by \cref{alg:aproxgrad} with $\lambda=1/L$. Then $G$ has a unique minimizer $y^\star$ and $\{G(y^k)\}_{k =0}^\infty$ converges $Q$-linearly to $G(y^\star)$, in particular,
$$
G(y^{k+1}) - G(y^\star) \leq (1 - \tfrac\mu{\mu + L} ) (G(y^k) - G(y^\star) ).
$$
\end{proposition}
A proof can be found in \cref{sec:thm_linear_convergence_dual}.

The above result can be translated to primal $R$-linear convergence by showing that $G(y^k) = F_\lambda(x^k)$ which relates the dual cost $G$ and the forward-backward envelope $F_\lambda$ to each other. This relation was observed previously in the Euclidean setting \cite[Lemma 6]{themelis2020envelope}:
\begin{proposition}[DC-dual vs. forward-backward envelope] \label{thm:fbe_dc}
Let $\{x^k\}_{k =0}^\infty$ be the sequence of backward-steps and $\{y^k\}_{k =0}^\infty$ be the sequence of forward-steps generated by \cref{alg:aproxgrad}. Then it holds that
$$
G(y^k)=(g\infconv \lambda \star \phi_-)(y^k) + (f \infdeconv \lambda \star \phi)(y^k)=F_\lambda(x^k),
$$
for all $k \in \bN_0$.
\end{proposition}
A proof is provided in \cref{sec:thm_fbe_dc}.
Now we are ready to state the primal $R$-linear convergence:
\begin{corollary} \label{thm:linear_convergence_primal}
Let $f$ be convex and $g$ be anisotropically strongly convex with constant $\mu$ relative to $\phi_-$. Let $\{x^k\}_{k =0}^\infty$ be the sequence of backward-steps generated by \cref{alg:aproxgrad} with step-size $\lambda=L^{-1}$. Then $\{F(x^{k+1}) \}_{k =0}^\infty$ converges $R$-linearly, in particular,
$$
F(x^{k+1}) - \inf F \leq (1 - \tfrac\mu{\mu + L} )^{k} (F(x^0) - \inf F).
$$
\end{corollary}
We provide a proof in \cref{sec:thm_linear_convergence_primal}.

%% file: missing_proofs.tex
\subsection{Proof of \cref{thm:phi_convex_asmooth}} \label{sec:thm:phi_convex_asmooth}
The proof combines and extends the proofs of \cite{laude2021conjugate} for $\phi$ which are possibly not super-coercive. 

\begin{proof}[Proof of \cref{thm:phi_convex_asmooth}]
Without loss of generality we assume that $L=1=\lambda$ by replacing $\phi$ with $\lambda \star \phi$. Choose \begin{align}
\Phi(x,y):=-\lambda \star \phi(x-y),
\end{align}
and recall \cref{thm:inf_conv_rem}

``\labelcref{thm:phi_convex_asmooth:asmooth} $\Rightarrow$ \labelcref{thm:phi_convex_asmooth:infconv}'': By definition $f \in \mathcal{C}^1(\bR^n)$. Let $\bar x \in \bR^n$. Then we have that for any $x$
$$
-f(x) \geq -\phi(x - (\bar x -\nabla \phi^*(\nabla f(\bar x)))) + \phi(\bar x - ( \bar x - \nabla \phi^*(\nabla f(\bar x)))) -f(\bar x),
$$
which means by definition of the $\Phi$-subdifferential that $\bar x - \nabla \phi^*(\nabla f(\bar x)) \in \partial_\Phi (-f)(\bar x)$. Invoking \cref{thm:phi_subgradients} this means that $(-f)(\bar x) = (-f)^{\Phi\Phi}(\bar x)$. Since $\bar x \in \bR^n$ was arbitrary in light of \cref{thm:inf_conv_rem} we have that $f=-(-f)^{\Phi\Phi} = \xi \infconv \phi$ for $\xi:=(-f)^\Phi$.

``\labelcref{thm:phi_convex_asmooth:infconv} $\Rightarrow$ \labelcref{thm:phi_convex_asmooth:asmooth}'': Suppose that $f \in \mathcal{C}^1(\bR^n)$ and $f=\xi \infconv \phi$ for some $\xi:\bR^n \to \exR$. Let $\bar x \in \bR^n$. Since $f \in \mathcal{C}^1(\bR^n)$ it is in particular finite-valued and we have $f(\bar x)=\inf_{y \in \bR^n} \xi(y) + \phi(\bar x - y) \in \bR$.
This means that for any $\varepsilon > 0$ there exists $\bar y_\varepsilon$ such that
\begin{align} \label{eq:eps_inf}
f(\bar x) + \varepsilon \geq \xi(\bar y_\varepsilon) + \phi(\bar x - \bar y_\varepsilon),
\end{align}
while for any $x \in \bR^n$ we have that $f(x) \leq \xi(\bar y_\varepsilon) + \phi(x - \bar y_\varepsilon)$. Combining the identities and eliminating $\xi(\bar y_\varepsilon)$ we obtain that for any $x \in \bR^n$:
\begin{align} \label{eq:ineq_ekeland}
\phi(x - \bar y_\varepsilon)-f(x) + \varepsilon \geq \phi(\bar x - \bar y_\varepsilon) -f(\bar x),
\end{align}
showing that $\bar x \in \epsargmin \{\phi(\cdot - \bar y_\varepsilon)-f \}$. Ekeland’s variational principle with $\delta=\sqrt{\varepsilon}$, see \cite[Proposition 1.43]{RoWe98}, yields the existence of a point $\bar x_\varepsilon \in \overline{\mathrm{B}}(\bar x; \sqrt{\varepsilon})$ with $\phi(\bar x_\varepsilon - \bar y_\varepsilon)-f(\bar x_\varepsilon) \leq \phi(\bar x - \bar y_\varepsilon) -f(\bar x)$ and $\bar x_\varepsilon \in \argmin \{\phi(\cdot - \bar y_\varepsilon)-f + \sqrt{\varepsilon}\|\cdot - \bar x_\varepsilon\|\}$. By Fermat's rule \cite[Theorem 10.1]{RoWe98} and the fact that $\partial \|\cdot\|(0)=\overline{\mathrm{B}}(0; 1)$ we have $0 \in \nabla \phi(\bar x_\varepsilon - \bar y_\varepsilon) -\nabla f(\bar x_\varepsilon) + \overline{\mathrm{B}}(0; \sqrt{\varepsilon})$ and thus
\begin{align} \label{eq:u_eps}
\nabla f(\bar x_\varepsilon) - \nabla \phi(\bar x_\varepsilon - \bar y_\varepsilon) =: u_\varepsilon \in \overline{\mathrm{B}}(0; \sqrt{\varepsilon}).
\end{align}
Since $\ran \nabla \phi = \intr(\dom \phi^*)$ we have that $\nabla f(\bar x_\varepsilon) -u_\varepsilon = \nabla \phi(\bar x_\varepsilon - \bar y_\varepsilon) \in \intr(\dom \phi^*)$.
Passing to $\varepsilon \to 0$, by continuity, $\intr(\dom \phi^*)\ni \nabla f(\bar x_\varepsilon) -u_\varepsilon \to \nabla f(\bar x) \in \intr(\dom \phi^*)$, where the inclusion follows by the constraint qualification $\ran \nabla f \subseteq \ran \nabla \phi = \intr(\dom \phi^*)$.
Since $\nabla\phi^*$ is continuous relative to $\intr(\dom \phi^*)$ we have by rewriting~\cref{eq:u_eps}
$$
\bar y_\varepsilon =\bar x_\varepsilon - \nabla \phi^*(\nabla f(\bar x_\varepsilon) -u_\varepsilon) \to \bar x -  \nabla \phi^*(\nabla h(\bar x)) =: \bar y,
$$ 
for $\varepsilon \to 0$.
Thus, passing to $\varepsilon \to 0$ in \cref{eq:ineq_ekeland} we obtain again by continuity that
$$
\phi(x - \bar x +\nabla \phi^*(\nabla f(\bar x)))-f(x) \geq \phi(\nabla \phi^*(\nabla f(\bar x))) -f(\bar x),
$$
which is the anisotropic descent inequality.

Thanks to the previous part of the proof the function $\xi$ in the expression $f = \xi \infconv \phi$ can be taken as $\xi=(-f)^\Phi$ which is lsc as it is a pointwise supremum over continuous functions. $\xi$ is also proper by properness of $f$.
Recall that $\bar y_{\varepsilon} \to \bar x -  \nabla \phi^*(\nabla f(\bar x)) = \bar y$ for $\varepsilon \to 0$. Then we have:
$$
\inf_{y \in \bR^n} \xi(y) + \phi(\bar x - y) \leq \xi(\bar y) + \phi(\bar x - \bar y) \leq \lim_{\varepsilon \searrow 0} \xi(\bar y_{\varepsilon}) + \phi(\bar x - \bar y_{\varepsilon}) =f(\bar x) = \inf_{y \in \bR^n} \xi(y) + \phi(\bar x - y),
$$
where the second inequality holds since $\xi$ is proper lsc and the second last equality holds thanks to \cref{eq:eps_inf}.

For the remainder of the proof assume that $f \in \Gamma_0(\bR^n)$.

``\labelcref{thm:phi_convex_asmooth:infconv} $\Rightarrow$ \labelcref{thm:phi_convex_asmooth:str_cvx}'': Assume that $f = \xi \infconv \phi$ for some $\xi :\bR^n \to \exR$. Then by definition of the convex conjugate we have that $f^*=\xi^* + \phi^*$ and $\xi^* \in \Gamma_0(\bR^n)$. By assumption we have $\dom \partial f^* =\ran \partial f \subseteq \ran \nabla \phi = \intr \dom \phi^*$ implying that $\emptyset \neq \dom f^* \cap \intr \dom \phi^*$. Since $\dom f^* = \dom \xi^* \cap \dom \phi^*$ we have $\emptyset \neq (\dom \xi^* \cap \dom \phi^*) \cap \intr \dom \phi^* = \dom \xi^* \cap \intr \dom \phi^*$.

``\labelcref{thm:phi_convex_asmooth:str_cvx} $\Rightarrow$ \labelcref{thm:phi_convex_asmooth:infconv}'': Let $f^* = \psi + \phi^*$ for some $\psi \in \Gamma_0(\bR^n)$ such that $\dom \psi \cap \intr \dom\phi^* \neq \emptyset$. Then we can invoke \cite[Proposition 6.19(vii)]{BaCo110} and \cite[Theorem 15.3]{BaCo110} to obtain that $f=f^{**} =(\psi + \phi^*)^*= \psi^* \infconv \phi \in \Gamma_0(\bR^n)$. 
By \cref{thm:moreau_decomposition:smooth} $\psi^* \infconv \phi \in \mathcal{C}^1(\bR^n)$ with $\nabla (\psi^* \infconv \phi) = \nabla \phi \circ (\id - \aprox{\phi}{\psi^*})$. This implies that $\ran \partial (\psi^* \infconv \phi) \subseteq \ran \nabla \phi$.

``\labelcref{thm:phi_convex_asmooth:str_cvx} $\Rightarrow$ \labelcref{thm:phi_convex_asmooth:str_cvx_relint}'': Let $f^* = \psi + \phi^*$ for some $\psi \in \Gamma_0(\bR^n)$ such that $\dom \psi \cap \intr \dom\phi^* \neq \emptyset$. We have that $\dom f^* = \dom \psi \cap \dom \phi^* \subseteq \dom \phi^*$. Choose $x_0 \in \relint \dom f^*$ and $x \in \dom \psi \cap \intr \dom\phi^* \subseteq \cl \dom f^*$. Then we have by the line segment principle \cite[Theorem 6.1]{Roc70} that $x_\tau :=(1-\tau) x_0 + \tau x \in \relint \dom f^*$ for every $\tau \in [0,1)$. Since $x_\tau \to x$ for $\tau \to 1$ and $\intr \dom \phi^*$ is open we have that $x_\tau \in \relint \dom f^* \cap \intr \dom \phi^*$ for $\tau$ close to $1$.
In particular $\psi \equiv f^* \mathbin{\dot{-}} \phi^*$ is finite-valued and convex on $\relint \dom f^* \subseteq \dom f^* \subseteq \dom \phi^*$.

``\labelcref{thm:phi_convex_asmooth:str_cvx_relint} $\Rightarrow$ \labelcref{thm:phi_convex_asmooth:str_cvx}'': Assume that $\dom f^* \subseteq \dom \phi^*$ with $\relint \dom f^* \cap \intr \dom \phi^* \neq \emptyset$ and $f^* \mathbin{\dot{-}} \phi^*$ is convex on $\relint \dom f^*$.
Define $\zeta(x) := f^*(x) - \phi^*(x) \in\bR$ if $x \in \relint \dom f^*$ and $+\infty$ otherwise. Then $\dom \zeta = \relint \dom f^*$. Since $\zeta$ is proper, convex, by \cite[Theorem 7.4]{Roc70} $\cl \zeta \in \Gamma_0(\bR^n)$ with $\zeta \equiv \cl \zeta$ on $\relint \dom f^*$. Next we prove that $f^* \equiv \cl \zeta + \phi^*$ on $\cl \dom f^*$:

For every $x \in \relint \dom f^* \subseteq \dom \phi^*$ we have $(\cl \zeta)(x) + \phi^*(x) = \zeta(x) +\phi^*(x)=f^*(x)$. Now take $x_0 \in \relint \dom f^*\cap \intr \dom \phi^*$ which exists by assumption and any $x \in \cl\dom f^*$.
Then we have by the line segment principle \cite[Theorem 6.1]{Roc70} that $(1-\tau) x_0 + \tau x \in \relint \dom f^*$ for every $\tau \in [0,1)$. In view of \cite[Corollary 7.4.1]{Roc70} $\relint(\dom \cl \zeta) = \relint \dom \zeta = \relint \dom f^*$. Therefore $x_0 \in \relint(\dom \cl \zeta) \cap \intr \dom \phi^*=\relint (\dom \cl \zeta \cap \dom \phi^*) = \relint(\dom (\cl \zeta + \phi^*))$ by \cite[Theorem 6.5]{Roc70} and $\cl \zeta + \phi^* \in \Gamma_0(\bR^n)$ by \cite[Theorem 9.3]{Roc70}. Then we can invoke \cite[Theorem 7.5]{Roc70} and obtain
\begin{align*}
(\cl \zeta)(x)+\phi^*(x)&=\lim_{\tau \nearrow 1} (\cl \zeta)((1-\tau) x_0 + \tau x) + \phi^*((1-\tau) x_0 + \tau x) \\
&= \lim_{\tau \nearrow 1} f^*((1-\tau) x_0 + \tau x) = f^*(x),
\end{align*}
where the second equality holds since $(1-\tau) x_0 + \tau x \in \relint \dom f^*$ and $\cl \zeta + \phi^* \equiv f^*$ on $\relint \dom f^*$, and the last equality again by \cite[Theorem 7.5]{Roc70} since $f^* \in \Gamma_0(\bR^n)$ and $x_0 \in \relint \dom f^*$.
This implies that $\cl \zeta + \phi^* \equiv f^*$ on $\cl \dom f^*$. By \cite[Theorem 7.4]{Roc70} $\cl \zeta$ differs from $\zeta$ at most at relative boundary points of $\dom \zeta = \relint \dom f^*$, i.e., at points $x \in \cl \dom \zeta \setminus \relint \dom \zeta = \cl \dom f^* \setminus \relint \dom f^*$. This implies that $\cl \zeta \equiv \zeta \equiv +\infty$ on $\bR^n \setminus \cl \dom f^*$ and thus $\cl \zeta + \phi^* \equiv f^* \equiv +\infty$ on $\bR^n \setminus \cl \dom f^*$.
Altogether this implies that $\psi + \phi^* = f^*$ for $\psi := \cl \zeta \in \Gamma_0(\bR^n)$.
Finally, since $f^*=\psi + \phi^*$ we have by assumption that $\emptyset \neq \dom f^* \cap \intr \dom \phi^*= (\dom \psi \cap \dom \phi^*) \cap \intr \dom \phi^* = \dom \psi \cap \intr \dom \phi^*$.



For the last equivalence assume that $f,\phi\in\mathcal{C}^2(\bR^n)$ are Legendre with invertible Hessian.

``\labelcref{thm:phi_convex_asmooth:str_cvx_relint} $\Leftrightarrow$ \labelcref{thm:phi_convex_asmooth:cvx_c2}'': Since $f$ is Legendre $\emptyset \neq \intr \dom f^* = \dom \nabla f^*$. In particular, $\dom f^* \subseteq \dom \phi^*$ implies $\ran \nabla f=\dom \nabla f^*\subseteq \dom \nabla \phi^* = \ran \nabla \phi$ and by the line segment principle also the opposite implication holds true and hence the characterization in \labelcref{thm:phi_convex_asmooth:str_cvx_relint} is equivalent to convexity of
\begin{equation*}
f^*-\phi^*
\end{equation*}
on $\intr \dom f^*$. 
Equivalently this means that
\begin{equation} \label{eq:proof_matrix_ineq}
\nabla^2 f^*(x^*) \succeq \nabla^2 \phi^*(x^*)
\end{equation}
for all $x^* \in \intr \dom f^*$. By the inverse function theorem we have that $\nabla^2 \phi(\nabla \phi^*(x^*))^{-1} = \nabla^2 \phi^*(x^*)$ and $\nabla^2 f(\nabla f^*(x^*))^{-1} = \nabla^2 f^*(x^*)$.
Hence, by inversion the matrix inequality \cref{eq:proof_matrix_ineq} is equivalent to 
\begin{equation*}
\nabla^2 f(\nabla f^*(x^*)) \preceq \nabla^2 \phi(\nabla \phi^*(x^*)),
\end{equation*}
for all $x^* \in \intr \dom f^*$.
Since $\ran \nabla f=\intr \dom f^*$ we can introduce a change of variable $\nabla f(x) = x^*$ and thus, since $\nabla f^* \circ \nabla f = \id$ the matrix inequality is equivalent to
\begin{equation*}
\nabla^2 f(x) \preceq \nabla^2 \phi(\nabla \phi^*(\nabla f(x))),
\end{equation*}
for all $x \in \bR^n$, as claimed.

This verifies all equivalences and concludes the proof.
\ifx\ifsvjour\true
\qed
\fi
\end{proof}
\subsection{Proof of \cref{thm:episcaling_asmoothness}} \label{sec:thm_episcaling_asmoothness}
The proof requires the following lemma:
\begin{lemma}
\ifx\ifsvjour\true
 \label[lemma]{thm:scaling_helper_lemma}
\else
 \label{thm:scaling_helper_lemma}
\fi
Let $(y_0, \beta_0)\in \bR^n \times \bR$. Then for any $0 < \lambda_2 < \lambda_1$ and any $\bar x \in \bR^n$ and
$$
y:=\bar x + \lambda_2\tfrac{\bar x - y_0}{\lambda_1}, \qquad \beta:=(\lambda_2-\lambda_1 )\phi(\tfrac{\bar x-y_0}{\lambda_1} ) + \beta_0,
$$
the inequality
\begin{align} \label{eq:scaling_inequality}
\lambda_1 \star \phi(x-y_0) -\beta_0 < \lambda_2 \star \phi(x-y) - \beta,
\end{align}
holds for all $x \neq \bar x$ while it holds with equality at $x=\bar x$.
\end{lemma}
\begin{proof}
Let $\bar x \in \bR^n$. Define $\psi(z):=\phi(z + \tfrac{\bar x - y_0}{\lambda_1} ) - \phi(\frac{\bar x - y_0}{\lambda_1} )$. Then clearly, $\psi(0)=0$ and $\psi$ is strictly convex. Thanks to \cite[Lemma 4.4]{burke2013epi} adapted to strict convexity we have that $\lambda_1 \star \psi(z) < \lambda_2 \star \psi(z)$ for any $z \neq 0$ which reads by expanding the epi-scaling:
\begin{equation*}
\lambda_1 \big(\phi(\tfrac{z + \bar x - y_0}{\lambda_1} ) - \phi(\tfrac{\bar x- y_0}{\lambda_1} )\big) < \lambda_2 \big(\phi(\tfrac{z}{\lambda_2} + \tfrac{\bar x - y_0}{\lambda_1} ) - \phi(\tfrac{\bar x - y_0}{\lambda_1} )\big).
\end{equation*}
Via the change of variable $x :=z+\bar x$ we have for any $x \neq \bar x$:
\begin{equation*}
\lambda_1 \big(\phi(\tfrac{x-y_0}{\lambda_1} ) - \phi(\tfrac{\bar x-y_0}{\lambda_1} )\big) < \lambda_2 \big(\phi(\tfrac{x - \bar x}{\lambda_2} + \tfrac{\bar x - y_0}{\lambda_1} ) - \phi(\tfrac{\bar x-y_0}{\lambda_1} )\big).
\end{equation*}
By definition of the epi-scaling this reads:
\begin{equation*}
\lambda_1 \star \phi(x-y_0) - \lambda_1 \phi(\tfrac{\bar x-y_0}{\lambda_1} ) < \lambda_2 \star \phi(x - \bar x + \lambda_2\tfrac{\bar x - y_0}{\lambda_1} ) - \lambda_2 \phi(\tfrac{\bar x-y_0}{\lambda_1} ).
\end{equation*}
Rearranging and adding $\beta_0$ to both sides of the inequality yields:
\[
\lambda_1 \star \phi(x-y_0) - \beta_0< \lambda_2 \star \phi(x - \bar x + \lambda_2\tfrac{\bar x - y_0}{\lambda_1} ) - (\lambda_2-\lambda_1 )\phi(\tfrac{\bar x-y_0}{\lambda_1} ) - \beta_0. \ifx\ifsvjour\true
\qed
\else
\qedhere
\fi
\]
\end{proof}

\begin{proof}[Proof of \cref{thm:episcaling_asmoothness}]
Let $L_2>L_1 >0$ and assume that $f$ is anisotropically smooth relative to $\phi$ with constant $L_1$. Let $\bar x, x \in \bR^n$ with $\bar x \neq x$. Then we have that
\begin{equation*}
f(x) \leq f(\bar x) + \tfrac{1}{L_1} \star \phi(x - \bar x + \tfrac{1}{L_1} \nabla \phi^*(\nabla f(\bar x))) -  \tfrac{1}{L_1} \phi(\nabla \phi^*(\nabla f(\bar x))).
\end{equation*}
Let $\lambda_1 :=\frac{1}{L_1}$ and $\lambda_2 :=\frac{1}{L_2}$. Choose $y_0:=\bar x - \lambda_1 \nabla \phi^*(\nabla f(\bar x))$ and $\beta_0:= \lambda_1 \phi(\nabla \phi^*(\nabla f(\bar x)))$. Then the claimed result follows from \cref{thm:scaling_helper_lemma} for $y = \bar x + \lambda_2 \nabla \phi^*(\nabla f(\bar x))$ and $\beta:=\lambda_2\phi(\nabla \phi^*(\nabla f(\bar x)))$.

In particular this proves the claimed implication in the second part of the statement which follows alternatively via \cref{thm:transitivity}.
\ifx\ifsvjour\true
\qed
\else
\qedhere
\fi
\end{proof}

\subsection{Proof of \cref{thm:closed_addition}} \label{sec:thm_closed_addition}
We first shall prove the following auxiliary result:
\begin{lemma}
\ifx\ifsvjour\true
 \label[lemma]{thm:helper_closed_addition}
\else
 \label{thm:helper_closed_addition}
\fi
Let $y \in (\bR^n)^m$ for $m \geq 1$ and $\alpha_i \geq 0$ with such that $\sum_{i=1}^m \alpha_i = 1$. Assume that $D_{\phi^*}$ is jointly convex. Then $f=\sum_{i=1}^m \alpha_i \phi(\cdot - y_i)$ is anisotropically smooth.
\end{lemma}
\begin{proof}
We first show by induction over $m$ that $\nabla (\sum_{i=1}^{m} \alpha_i \phi(\cdot - y_i))$ is $\nabla \phi^*$-firmly nonexpansive on $\bR^n$ in the sense of \cite[Definition 4.1]{wang2021bregman} and $\dom  (\sum_{i=1}^{m} \alpha_i \phi(\cdot - y_i))^* = \dom \phi^*$. Then the claim follows by \cite[Lemma 4.2]{wang2021bregman} and \cref{thm:phi_convex_asmooth}. Let $m=2$. We have that $\phi(\cdot - y_i) = (\langle \cdot, y_i \rangle + \phi^*)^*$ for $i \in \{1, 2\}$. Thus we can invoke \cite[Lemma 4.2]{wang2021bregman} and obtain that $\nabla \phi(\cdot - y_i)$ is $\nabla \phi^*$-firmly nonexpansive on $\bR^n$ for $i\in \{1,2\}$. Since $D_{\phi^*}$ is jointly convex, in view of \cite[Remark 4.6]{wang2021bregman} we can invoke \cite[Lemma 4.4]{wang2021bregman} and obtain that $\nabla (\alpha_1 \phi(\cdot - y_1) + \alpha_2 \phi(\cdot - y_2)) = \alpha_1 \nabla \phi(\cdot - y_1)+ (1-\alpha_1) \nabla \phi(\cdot - y_2)$ is $\nabla \phi^*$-firmly nonexpansive on $\bR^n$ as well.
By convexity of $\dom \phi^*$ we have that 
\begin{align*}
\dom (\alpha_1 \phi(\cdot - y) + \alpha_2 \phi(\cdot - y_2))^* &= \dom \alpha_1 \star (\langle \cdot, y_1 \rangle + \phi^*) + \dom (1-\alpha_1) \star (\langle \cdot, y_2 \rangle + \phi^*)  \\
&= \alpha_1 \dom \phi^* + (1-\alpha_1) \dom \phi^* = \dom \phi^*.
\end{align*}
Now let $m>2$. We write
$$
\sum_{i=1}^{m} \alpha_i \phi(\cdot - y_i) = (1 - \alpha_m) \sum_{i=1}^{m-1} \tfrac{\alpha_i}{1 - \alpha_m}  \phi(\cdot - y_i) + \alpha_m \phi(\cdot - y_m) = (1-\alpha_m)\psi + \alpha_m\phi(\cdot - y_m),
$$
for $\psi := \sum_{i=1}^{m-1} \beta_i \phi(\cdot - y_i)$ with $\beta_i :=\tfrac{\alpha_i}{1 - \alpha_m}$.
Note that $\alpha_m = 1 - \sum_{i=1}^{m-1} \alpha_i$ and thus $\sum_{i=1}^{m-1} \beta_i = 1$.
By the induction hypothesis we have that $\nabla \psi = \nabla (\sum_{i=1}^{m-1} \beta_i \phi(\cdot - y_i))$ is $\nabla \phi^*$-firmly nonexpansive on $\bR^n$ and $\dom \psi^* = \dom \phi^*$. 
Using the same arguments as above we obtain that $\nabla (\sum_{i=1}^{m} \alpha_i \phi(\cdot - y_i)) = (1 - \alpha_m) \nabla \psi + \alpha_m \nabla \phi(\cdot - y_m)$ is $\nabla \phi^*$-firmly nonexpansive on $\bR^n$ and
$
\dom (\textstyle\sum_{i=1}^{m} \alpha_i \phi(\cdot - y_i))^* 
= \dom \phi^*.
$
Thanks to \cite[Lemma 4.2]{wang2021bregman} and \cref{thm:phi_convex_asmooth} $\sum_{i=1}^{m} \alpha_i \phi(\cdot - y_i)$ is anisotropically smooth relative to $\phi$ as claimed.
\ifx\ifsvjour\true
\qed
\fi
\end{proof}
\begin{proof}[Proof of \cref{thm:closed_addition}]
Let $L:=\max_{1 \leq i \leq m} L_i$. Thanks to \cref{thm:episcaling_asmoothness} for any $1 \leq i \leq m$ the function $f_i$ is anisotropically smooth relative to $\phi$ with constant $L$. Invoking \cref{thm:phi_convex_asmooth} we have $f_i(x) = \inf_{y \in \bR^n} L^{-1} \star \phi(x - y) + \xi_i(y)$ for some $\xi_i :\bR^n \to \exR$.

Without loss of generality assume that $L=1$.
Let $\alpha_i \geq 0$, $\sum_{i=1}^m \alpha_i = 1$ and $x\in \bR^n$. Then we have
\begin{align}
\sum_{i=1}^m \alpha_i f_i(x) &= \sum_{i=1}^m \inf_{y \in \bR^n} \alpha_i \phi(x - y) + \alpha_i \xi_i(y) \notag \\
&=\inf_{y \in (\bR^n)^m} \sum_{i=1}^m \alpha_i \phi(x - y_i) + \sum_{i=1}^m \alpha_i \xi_i(y_i) \label{eq:inf_conv_expr_average_0}.
\end{align}
Let $y \in (\bR^n)^m$. In light of \cref{thm:helper_closed_addition} $\sum_{i=1}^m\alpha_i \phi(\cdot - y_i)$ is anisotropically smooth relative to $\phi$. Hence we have by \cref{thm:phi_convex_asmooth} that
\begin{align}
\sum_{i=1}^m\alpha_i \phi(x - y_i) = \inf_{v \in \bR^n} \phi(x - v) + \rho(v;y), \label{eq:inf_conv_expr_average}
\end{align}
for some $\rho(\cdot; y) :\bR^n \to \exR$.
Substituting \cref{eq:inf_conv_expr_average} in \cref{eq:inf_conv_expr_average_0} an interchange of the order of minimization yields:
\begin{align*}
\sum_{i=1}^m \alpha_i f_i(x) &= \inf_{y \in (\bR^n)^m} \sum_{i=1}^m \alpha_i \phi(x - y_i) + \sum_{i=1}^m \alpha_i \xi_i(y_i) \\
&= \inf_{y \in (\bR^n)^m} \inf_{v \in \bR^n} \phi(x - v) + \rho(v;y) + \sum_{i=1}^m\alpha_i \xi_i(y_i) \\
&=\inf_{v \in \bR^n} \phi(x - v) + \inf_{y \in (\bR^n)^m} \rho(v;y) + \sum_{i=1}^m\alpha_i \xi_i(y_i) \\
&= \inf_{v \in \bR^n} \phi(x - v) + \zeta(v),
\end{align*}
where
$$
\zeta(v) := \inf_{y \in (\bR^n)^m} \rho(v;y) +\sum_{i=1}^m\alpha_i \xi_i(y_i).
$$
By convexity of $\ran \nabla \phi= \intr\dom \phi^*$ we have that 
\[
\ran \nabla (\textstyle\sum_{i=1}^m \alpha_i f_i) \subseteq \alpha_1\ran \nabla f_1 + \cdots + \alpha_m\ran \nabla f_m \subseteq \ran \nabla \phi.
\]
Then, in view of \cref{thm:phi_convex_asmooth}, $\sum_{i=1}^m \alpha_i f_i$ is anisotropically smooth relative to $\phi$.
\ifx\ifsvjour\true
\qed
\fi
\end{proof}

\subsection{Proof of \cref{ex:logistic}} \label{sec:ex_logistic}
\begin{proof}
To show the claimed result we first show that $f_i:=\ell(\langle a_i,\cdot \rangle)$ is anisotropically smooth relative to $\phi$ with constant $L$. The claimed result then follows by verifying joint convexity of $D_{\phi^*}$ and invoking \cref{thm:closed_addition}.
Anisotropic smoothness of $f_i$ is proved via \cref{thm:phi_convex_asmooth} by verifying that $\dom f_i^* \subseteq \dom \phi^*$ such that $\emptyset \neq \intr \dom \phi^* \cap \relint \dom f_i^*$ and $f_i^* \mathbin{\dot{-}} \tfrac{1}{L} \phi^*$ is convex on the relative interior of $\dom f_i^*$.

We have $h^*(t)=(t + 1)\ln((t + 1)/2) + (1-t)\ln((1-t)/2)$ with $\dom h^*=[-1, 1]$ where $h^*(t)=0$ for $t\in\{-1,1\}$.
Owing to \cite[Theorem 11.23(b)]{RoWe98} the conjugate $f_i^*=(\ell \circ a_i^\top)^*$ amounts to the infimal postcomposition $a_i \ell^*$ of $\ell^*(t) = t\ln t + (1-t)\ln(1-t)$ by $a_i$ with $\dom \ell^*=[0,1]$ and $\ell^*(t)=0$ for $t\in\{0,1\}$ and is given as follows:
\begin{align*}
(a_i \ell^*)(x)= \inf \{\ell^*(t) \mid t \in \bR : a_i t = x\} =  \begin{cases} \ell^*(t) & \text{if $a_i t = x$ for some $t\in [0, 1]$,} \\
+ \infty & \text{otherwise}.
\end{cases}
\end{align*}
Then we have that $\dom a_i \ell^* = \{a_i t \mid t \in [0, 1]\} \subseteq \dom \phi^*$ and $\tfrac{1}{2} a_i \in  \intr \dom \phi^* \cap \relint (\dom a_i \ell^*)$.
It remains to verify convexity of $a_i \ell^* \mathbin{\dot{-}} \tfrac{1}{L} \phi^*$ on the relative interior of $\dom a_i \ell^*$:
Equivalently this means that the one-dimensional function $\ell^* \mathbin{\dot{-}} \tfrac{1}{L} \phi^* \circ a_i$ is convex on $(0, 1)$. Since $(h^*)''(t)=2/(1-t^2)$ and $(\ell^*)''(t)=\tfrac{1}{t-t^2}>0$ on $(0,1)$ this is implied by
$$
L\geq \tfrac{a_i^\top \nabla^2 \phi^*(a_i t) a_i}{(\ell^*)''(t)} = a_i^\top \diag (2(t-t^2)/(1-a_i^2 t^2) )a_i =\sum_{j=1}^m \tfrac{2(t-t^2) a_{ij}^2}{1-a_{ij}^2 t^2},
$$
for all $t \in (0,1)$. Clearly, $\tfrac{2(t-t^2) a_{ij}^2}{1-a_{ij}^2 t^2} \leq \tfrac{2(t-t^2) a_{ij}^2}{1-t^2}$ and $\frac{2(t-t^2) a_{ij}^2}{1-t^2}$ is monotonically increasing on $[0,1)$. For $t \to 1$ the expression is indefinite. Thus we apply L'Hospital's rule and obtain:
$$
\lim_{t\to 1} \frac{2(t-t^2) a_{ij}^2}{1-t^2} = \lim_{t\to 1} \frac{2 a_{ij}^2 (1-2t)}{-2t}= \lim_{t\to 1} \frac{(2t-1) a_{ij}^2}{t} = a_{ij}^2,
$$
and thus we have for $t \in (0,1)$
$$
\frac{(\phi^* \circ a_i)''(t) }{(\ell^*)''(t)} =\sum_{j=1}^m \frac{2(t-t^2) a_{ij}^2}{1-a_{ij}^2 t^2} \leq \sum_{j=1}^m \frac{2(t-t^2) a_{ij}^2}{1-t^2} \leq  \sum_{j=1}^m a_{ij}^2= \|a_i\|^2 =:L_i.
$$
In light of \cref{thm:phi_convex_asmooth} $f_i$ is anisotropically smooth relative to $\frac{1}{L_i} \star \phi$. Since $1/(h^*)''(t)=(1-t^2)/2$ is concave by \cite[Theorem 3.3(i)]{bauschke2001joint} $D_{\phi^*}$ is jointly convex. Invoking \cref{thm:closed_addition} we obtain that $f=\frac{1}{m}\sum_{i=1}^m f_i$ is anisotropically smooth relative to $\frac{1}{L} \star \phi$ with $L = \max_{1\leq i \leq m} \|a_i\|^2$.
\ifx\ifsvjour\true
\qed
\fi
\end{proof}

\subsection{Proof of \cref{thm:closed_addition_exp}} \label{sec:thm_closed_addition_exp}
\begin{proof}
First note that $f \equiv 0$ is not exponentially smooth since $\dom f^* =\{0\}\not\subseteq  \intr(\dom \SumExp^*)$.
Let $L:=\max_{1 \leq i \leq m} L_i$. Thanks to \cref{thm:episcaling_asmoothness} $f_i$ is exponentially smooth with constant $L$ for any $1 \leq i \leq m$. 
Thanks to \cref{thm:phi_convex_asmooth} we have $f_i(x) = \inf_{y \in \bR^n} L^{-1}\star \SumExp(x - y) + \xi_i(y)$ for some $\xi_i :\bR^n \to \exR$.
Without loss of generality assume that $L=1$.
Let $x\in \bR^n$. Then we have
\begin{align}
\sum_{i=1}^m \alpha_i f_i(x) &= \sum_{i=1}^m \inf_{y \in \bR^n} \alpha_i \SumExp(x - y) + \alpha_i \xi_i(y) \notag \\
&=\inf_{y \in (\bR^n)^m} \sum_{i=1}^m \alpha_i \SumExp(x - y^i) + \sum_{i=1}^m \alpha_i \xi_i(y^i) \label{eq:exp_smoothness_eq}.
\end{align}
Let $y \in (\bR^n)^m$. It holds that
\begin{align} \label{eq:inf_conv_indicator}
\sum_{i=1}^m \alpha_i \SumExp(x - y^i) &= \sum_{i=1}^m \sum_{j=1}^n \alpha_i \exp(x_j - (y^i)_j) \notag \\
&= \sum_{j=1}^n\exp(x_j) \sum_{i=1}^m \alpha_i \exp(-(y^i)_j) \notag \\
&=\sum_{j=1}^n\exp(x_j + \ln(\textstyle \sum_{i=1}^m \alpha_i \exp(-(y^i)_j))),
\end{align}
where the last equality holds since $\sum_{i=1}^m \alpha_i >0$ and thus $\sum_{i=1}^m \alpha_i \exp(-(y^i)_j) > 0$. We define $s: (\bR^n)^m \to \bR^n$ as $s_j(y):= -\ln(\sum_{i=1}^m \alpha_i \exp(-(y^i)_j))$ for all $1 \leq j \leq n$.
Using \cref{eq:inf_conv_indicator} we can further rewrite $\sum_{i=1}^m \alpha_i \SumExp(x - y^i)$ as
\begin{align*}
\sum_{i=1}^m \alpha_i \SumExp(x - y^i) &=\sum_{j=1}^n\exp(x_j + \ln \textstyle \sum_{i=1}^m \alpha_i \exp(-(y^i)_j)) \\
&= \inf_{v \in \bR^n} \SumExp(x-v) + \delta_{\{0\}}(v-s(y)).
\end{align*}
An interchange of the order of minimization leads to:
\begin{align*}
\sum_{i=1}^m \alpha_i f_i(x) &=\inf_{y \in (\bR^n)^m} \sum_{i=1}^m \alpha_i \SumExp(x - y^i) + \sum_{i=1}^m \alpha_i \xi_i(y^i) \\
&=   \inf_{y \in (\bR^n)^m} \inf_{v \in \bR^n} \SumExp(x-v) + \delta_{\{0\}}(v-s(y)) +\sum_{i=1}^m \alpha_i \xi_i(y^i) \\
&= \inf_{v \in \bR^n}\SumExp(x-v) + \inf_{y \in (\bR^n)^m} \delta_{\{0\}}(v-s(y)) + \sum_{i=1}^m \alpha_i \xi_i(y^i) \\
&= \inf_{v \in \bR^n} \SumExp(x - v) + \zeta(v),
\end{align*}
where $\zeta(v) := \inf_{y \in (\bR^n)^m} \delta_{\{0\}}(v-s(y)) + \sum_{i=1}^m\alpha_i \xi_i(y^i)$. 
By assumption we have that $\ran \nabla f_i \subseteq \bR_{++}^n$and since $\sum_{i=1}^m \alpha_i >0$ we have $\ran (\sum_{i=1}^m \alpha_i \nabla f_i) \subseteq \bR_{++}^n = \ran \Exp$. Again thanks to \cref{thm:phi_convex_asmooth} $\sum_{i=1}^m \alpha_i f_i$ is exponentially smooth as claimed.
\ifx\ifsvjour\true
\qed
\fi
\end{proof}

\subsection{Proof of \cref{thm:exponential_smoothness_penalty}} \label{sec:thm_exponential_smoothness_penalty}
We first require a lemma:
\begin{lemma}
\ifx\ifsvjour\true
 \label[lemma]{thm:exponential_smoothness}
\else
  \label{thm:exponential_smoothness}
\fi
Let $A \in \bR_+^{m \times n}$ such that each column contains at least one nonzero element. Then $f(x):= \SumExp(Ax)$ is exponentially smooth with constant $L = \|A\|_{\infty}$.
\end{lemma}
\begin{proof}
In the following we identify the rows of the matrix $A\in \bR_{+}^{m \times n}$ as vectors $a_i$ and for now assume that $a_i \in \bR_{++}^n$.
Then $f(x)=\SumExp(Ax)=\sum_{i=1}^m f_i(x)$ for $f_i(x)=\exp(\langle a_i, x \rangle)$.
We show that $f_i$ is anisotropically smooth relative to $\phi$ with constant $L$ and then invoke \cref{thm:closed_addition_exp}. 
Anisotropic smoothness of $f_i$ is proved via \cref{thm:phi_convex_asmooth} by verifying that $\dom f_i^* \subseteq \dom \phi^*$ such that $\emptyset \neq \intr \dom \phi^* \cap \relint \dom f_i^*$ and $f_i^* \mathbin{\dot{-}} \tfrac{1}{L} \phi^*$ is convex on the relative interior of $\dom f_i^*$.

Owing to \cite[Theorem 11.23(b)]{RoWe98} the conjugate $(\exp \circ a_i^\top)^*$ amounts to the infimal postcomposition $a_i h$ of the von Neumann entropy $h(t):=t \ln(t) - t$ by $a_i$ and is given as
\begin{align*}
(a_i h)(x)= \inf \{h(t) \mid t \in \bR : a_i t = x\} =  \begin{cases} h(t) & \text{if $a_i t = x$ for some $t\in \bR_{+}$,} \\
+ \infty & \text{otherwise}.
\end{cases}
\end{align*}
Denote by $H(x):= \sum_{j=1}^n h(x_j)$ the von Neumann entropy on $\bR^n$.
We have $\dom a_i h = \{a_i t \mid t \in \bR_{+}\} \subseteq \dom H$ and since $a_i \in \bR_{++}^n$ we have $\relint \dom (a_i h) \cap \intr \dom \phi^* \neq \emptyset$.
It remains to verify convexity of $a_i h \mathbin{\dot{-}} \frac{1}{L} H$ on $\relint \dom (a_i h)$. Equivalently this means that the one-dimensional function $h \mathbin{\dot{-}} \frac{1}{L} H \circ a_i$ is convex on $\bR_{++}$. This is implied by
$$
h''(t) - \tfrac{1}{L} a_i^\top \nabla^2 H(a_i t) a_i= h''(t) - \tfrac{1}{L} a_i^\top \nabla^2 H(a_i t) a_i = \tfrac{1}{t} - \tfrac{1}{L} \|a_i\|_1\tfrac{1}{t} = \tfrac{1}{t}(1 -\|a_i\|_1/L) \geq 0,
$$
for all $t > 0$. This is guaranteed by the choice of $L= \max_{1 \leq i \leq m} \|a_i\|_1$.
In light of \cref{thm:phi_convex_asmooth} $f_i$ is anisotropically smooth relative to $\frac{1}{L} \star \phi$. Applying \cref{thm:closed_addition_exp} we obtain the claimed result.

Next we are going to relax the assumption $A \in \bR_{++}^{m\times n}$ by $A \in \bR_{+}^{m \times n}$ as long as each column contains at least one nonzero element.
For that purpose we add $\varepsilon$ to the entries in $A$. We denote the perturbed matrix as $A_\varepsilon \in \bR_{++}^{m\times n}$ and define $f_\varepsilon(x):=\SumExp(A_\varepsilon x -b)$. Then $\nabla f_\varepsilon(x) = A_\varepsilon^\top \Exp(A_\varepsilon x -b) \in \bR_{++}^n$ and by the structure of $A$, $A^\top \Exp(A x -b) \in \bR_{++}^n$ as well.
Thanks to the previous part of the proof the descent inequality is valid for $f_\varepsilon$ and any $x,\bar x\in \bR^n$ fixed. Passing $\varepsilon \searrow 0$, by continuity, the descent inequality is also valid for $\varepsilon=0$.
\ifx\ifsvjour\true
\qed
\fi
\end{proof}
\begin{proof}[Proof of \cref{thm:exponential_smoothness_penalty}]
In light of \cref{thm:phi_convex_asmooth,thm:translation_invariance}, $h(\cdot - b) = \xi \infconv \frac{1}{L}\star\SumExp$ for some function $\xi$.
We have by definition of the infimal convolution
    \begin{align} \label{eq:infconv_f}
        h(Ax-b) =(\xi \infconv \SumExp)(Ax) = \inf_{y \in \bR^m} \tfrac{1}{L} \star \SumExp(Ax - y) + \xi(y).
    \end{align}
    Let $y \in \bR^m$. In light of \cref{thm:exponential_smoothness} and the invariance under positive scaling \cref{thm:closed_addition_exp}, $x \mapsto \frac{1}{L} \star \SumExp(Ax - y)=\frac{1}{L} \SumExp(LAx - Ly)$ is exponentially smooth with constant $M=L\|A\|_{\infty}$. Again invoking \cref{thm:phi_convex_asmooth} there is a function $\omega(\cdot; y)$ such that
    $$
    \tfrac{1}{L}\star \SumExp(Ax  - y) = \inf_{z \in \bR^m} \tfrac{1}{M}\star\SumExp(x - z) + \omega(z; y).
    $$
    Substituting this identity in \cref{eq:infconv_f} we obtain via an interchange of the order of minimization
        \begin{align*}
        h(Ax-b) &= \inf_{y \in \bR^m} \tfrac{1}{L} \star \SumExp(Ax - y) + \xi(y) \\
        &=\inf_{y,z \in \bR^m} \tfrac{1}{M}\star\SumExp(x - z) + \omega(z; y) + \xi(y) \\
        &= \inf_{z \in \bR^m} \tfrac{1}{M}\star\SumExp(x - z) + \inf_{y \in \bR^m} \omega(z; y) + \xi(y) \\
        &= \inf_{z \in \bR^m} \tfrac{1}{M}\star\SumExp(x - z) + \theta(z),
    \end{align*}
    for $\theta(z):= \inf_{y \in \bR^m} \omega(z; y) + \xi(y)$. 
    Since $h(\cdot-b)$ is exponentially smooth we have that $\ran \nabla h(\cdot-b) \subseteq \bR_{++}^n$.  Since $A\in\bR_+^{m \times n}$ such that each column contains at least one nonzero element, we have that $\nabla (h(\cdot -b) \circ A)= A^\top \nabla h(Ax-b) = A^\top \nabla h(Ax-b) \in \bR^n_{++}$, and hence $f=h(\cdot -b) \circ A$ is exponentially smooth with constant $M=L\|A\|_{\infty}$.
\ifx\ifsvjour\true
\qed
\fi
\end{proof}

\subsection{Proof of \cref{ex:polylog}} \label{sec:ex_polylog}
\begin{proof}
Thanks to the fundamental theorem of calculus the derivative amounts to $\vartheta'(t) = \ln(1+\exp(t))$ and hence $\ran \vartheta' = (0,\infty)$ and $\vartheta' \in \cC^1$ implying that $\vartheta \in \cC^2$. Since $\vartheta'$ is monotonically increasing $\vartheta$ is convex.
We compute $(\vartheta^*)'(t^*)=(\vartheta')^{-1}(t^*)=\ln(\exp(t^*)-1)$ and $(\vartheta^*)''(t^*)=\exp(t^*)/(\exp(t^*)-1)$. Denote by $h(t^*)= t^*\ln(t^*)-t^*$ the von-Neumann entropy.
Since $\exp(t^*) > 1$ for $t^* \in (0,+\infty)$ and $1 + \exp(t^*)(t^*- 1)>0$ for all $t^* \in (0,\infty)$ we have that $((\vartheta^*)''-h'')(t^*)=(1 + \exp(t^*)(t^*- 1)) / ((\exp(t^*) -1) t^*) > 0$ for any $t^* \in (0, \infty)$. Hence $\vartheta^* - h$ is convex on $(0,\infty)$. Invoking \cref{thm:phi_convex_asmooth} we obtain that $\vartheta$ is exponentially smooth with constant $1$. Due to separability, by \cref{thm:separable}, $\Theta$ is exponentially smooth with constant $1$. 
To prove the final claim note that 
$
\ln(1+\exp(\tau)) - \ln(1+\exp(-\tau)) 
=\tau,
$
and hence
$
\vartheta(t)+\vartheta(-t) 
=\tfrac{1}{2}t^2. 
$
\ifx\ifsvjour\true
\qed
\fi
\end{proof}

\subsection{Proof of \cref{thm:strong_implies_pg_dominance}} \label{sec:thm_strong_implies_pg_dominance}
First we shall prove the following monotonicity property of the gap function $\gap{\phi}{F}(\cdot, \lambda)$ under epi-scaling adapting the proof of \cite[Lemma 1]{karimi2016linear} to the anisotropic case:
\begin{lemma}[monotonicity of the gap function under epi-scaling]
\ifx\ifsvjour\true
 \label[lemma]{thm:scaling_property}
\else
 \label{thm:scaling_property}
\fi
Let $g \in \Gamma_0(\bR^n)$. For any $0<\lambda_2 \leq \lambda_1$ we have that
$$
\gap{\phi}{F}(\bar x, \lambda_2) \geq  \gap{\phi}{F}(\bar x, \lambda_1).
$$
\end{lemma}
Now we are ready to prove that anisotropic strong convexity of $f$ implies anisotropic proximal gradient dominance of $F=f+g$:
\begin{proof}
Without loss of generality let $\bar x \in \dom g$. We rewrite 
\begin{align}
\gap{\phi}{F}(\bar x, \lambda) &= \tfrac{1}{\lambda}\big(F(\bar x) - F_{\lambda}(\bar x) \big) \notag \\
&= \phi(\nabla \phi^*(\nabla f(\bar x))) -\min_{x \in \bR^n} \left\{ \phi(\lambda^{-1}(x - \bar x) + \nabla \phi^*(\nabla f(\bar x))) + \lambda^{-1}(g(x) - g(\bar x)) \right\} \notag \\
&= \phi(\nabla \phi^*(\nabla f(\bar x))) -\min_{u \in \bR^n} \left\{ \phi(u + \nabla \phi^*(\nabla f(\bar x)))  + \lambda^{-1}(g(\lambda u +\bar x) - g(\bar x)) \right\} \label{eq:subst_scaling_1}\\
&= \phi(\nabla \phi^*(\nabla f(\bar x)))-\min_{u \in \bR^n} \left\{ \phi(u + \nabla \phi^*(\nabla f(\bar x))) + \lambda^{-1}\star \zeta(u) \right\}\label{eq:subst_scaling_2},
\end{align}
where \cref{eq:subst_scaling_1} follows by the change of variable $u := \lambda^{-1}(x - \bar x)$ and \cref{eq:subst_scaling_2} follows by defining $\zeta:=g(\cdot + \bar x) - g(\bar x)$. Thus $\gap{\phi}{F}(\bar x, \lambda_2) \geq \gap{\phi}{F}(\bar x, \lambda_1)$ is implied by
$\tfrac{1}{\lambda_2}\star \zeta(u) \leq \tfrac{1}{\lambda_1}\star \zeta(u)$ for all $u \in \bR^n$,
which holds by convexity of $\zeta$ and $\zeta(0) = 0$ thanks to \cite[Lemma 4.4]{burke2013epi}; also see the proof of \cite[Lemma 1]{karimi2016linear}.
%
\ifx\ifsvjour\true
\qed
\fi
\end{proof}
\begin{proof}[Proof of \cref{thm:strong_implies_pg_dominance}]
Since $f$ is anisotropically strongly convex we have $\ran \nabla f \supseteq \ran \nabla \phi$. By anisotropic smoothness of $f$ we also have $\ran \nabla f \subseteq \ran \nabla \phi$. Combining the inclusions we have $\ran \nabla f = \ran \nabla \phi=\intr\dom \phi^*$. Let $\bar x \in \bR^n$. 
The anisotropic strong convexity inequality yields
\begin{align*}
f(x) &\geq f(\bar x) + \tfrac1 \mu \star \phi(x-\bar x + \mu^{-1}  \nabla\phi^*(\nabla f(\bar x))) - \tfrac1 \mu \star \phi(\mu^{-1} \nabla\phi^*(\nabla f(\bar x))) \\
&> f(\bar x) + \langle \nabla f(\bar x), x - \bar x\rangle,
\end{align*}
for all $x \in \bR^n$ with $x\neq \bar x$, where the second inequality follows due to strict convexity of $\phi$. Thus $f$ is strictly convex and therefore $F=f+g$ is strictly convex relative to $\dom F = \dom g$.
Thanks to \cref{thm:moreau_decomposition:decomp}
$$
\xi_{\bar x}(x) := f(\bar x) + \tfrac1 \mu \star \phi(x-\bar x + \mu^{-1} \nabla\phi^*(\nabla f(\bar x))) - \tfrac1 \mu \star \phi(\mu^{-1}\nabla\phi^*(\nabla f(\bar x))) + g(x),
$$
has the unique minimizer $\argmin_{x \in \bR^n} \xi_{\bar x}(x) = \aprox[\mu]{\phi_-}{g}(\bar x - \mu^{-1} \nabla\phi^*(\nabla f(\bar x)))$. Invoking \cite[Theorem 11.8(b)]{RoWe98} we know that $0 \in \dom \partial (\xi_{\bar x})^*$. Since $\xi_{\bar x}$ is strictly convex relative to $\dom g$ we have that $(\xi_{\bar x})^*$ is essentially smooth implying via \cite[Theorem 26.1]{Roc70} that $0 \in \dom \partial (\xi_{\bar x})^*=\intr(\dom (\xi_{\bar x})^*)$. In view of \cite[Theorem 11.8(c)]{RoWe98} this means that $\xi_{\bar x}$ is coercive. Adding $g(x)$ to both sides of the anisotropic strong convexity inequality yields $\xi_{\bar x} \leq F$ and hence the cost $F$ is coercive too. This implies via strict convexity of $F$ relative to its domain that $F$ has a unique minimizer $x^\star$. 
By minimizing both sides in $F(x) \geq \xi_{\bar x}(x)$ wrt $x$ we obtain
\begin{align*}
F(x^\star)&\geq \inf_{x \in \bR^n} \tfrac1 \mu \star \phi(x-\bar x + \mu^{-1}  \nabla\phi^*(\nabla f(\bar x))) - \tfrac1 \mu \star \phi(\mu^{-1} \nabla\phi^*(\nabla f(\bar x))) + g(x)\\
& = F_{\mu^{-1}}(\bar x).
\end{align*}
This implies that $\mu^{-1}\gap{\phi}{F}(\bar x, \mu^{-1}) = F(\bar x) - F_{\mu^{-1}}(\bar x) \geq F(\bar x) - F(x^\star)$.
Invoking \cref{thm:scaling_property} we have $\mu(F(\bar x) -F(x^\star)) \leq \gap{\phi}{F}(\bar x, \mu^{-1}) \leq \gap{\phi}{F}(\bar x, \lambda)$,
for any $\lambda \leq \mu^{-1}$.
\ifx\ifsvjour\true
\qed
\fi
\end{proof}

\subsection{Proof of \cref{thm:conjugate_duality_astrong}} \label{sec:thm:conjugate_duality_astrong}
The proof is an adaption of the proof of \cite[Theorem 4.3]{laude2021conjugate} for $\phi$ which are possibly not super-coercive. 
\begin{proof}[Proof of \cref{thm:conjugate_duality_astrong}]
Without loss of generality assume that $\mu = 1$ by replacing $\phi$ with $\mu^{-1} \star \phi$. Choose $\Phi(x,y) = + \mu \star \phi(x-y)$. 

``\labelcref{thm:conjugate_duality_astrong:bsmooth} $\Rightarrow$ \labelcref{thm:conjugate_duality_astrong:astrong}'': $\dom f^* \supseteq \dom \phi^*$ implies that $\intr \dom f^* \supseteq \intr \dom \phi^*$ and thus $\ran \partial f=\dom \partial f^* \supseteq \intr \dom f^* \supseteq \intr \dom \phi^*=\ran \nabla \phi$. In particular $\phi^* + (- f^*)$ is finite-valued and convex on $\intr \dom \phi^*$. This implies that $-f^*$ is finite-valued on $\intr \dom \phi^*$. Since $\phi^*$ is smooth on $\intr \dom \phi^*$ we can invoke \cite[Exercise 8.20(b)]{RoWe98} to show that $-f^*$ is regular on $\intr \dom \phi^*$. Invoking \cite[Exercise 8.8(c)]{RoWe98} we obtain that $\partial(\phi^* +(- f^*)) =\widehat \partial(\phi^* +(- f^*)) = \nabla \phi^* + \widehat \partial (-f^*)$ on $\intr \dom \phi$. Since, in addition, $f^* \in \Gamma_0(\bR^n)$, both, $\pm f^*$ are regular with $\widehat\partial (\pm f^*)$ nonempty on $\intr \dom \phi^*$.
Invoking \cite[Theorem 9.18]{RoWe98} this means that $f^*$ is smooth on $\intr \dom \phi^*$ and as such so is $g:= \phi^* \mathbin{\dot{-}} f^*$.                   
Fix $(\bar x, \bar v) \in \gph \partial f \cap (\bR^n \times \intr \dom \phi^*)$ and let $\bar y:= \bar x - \nabla \phi^*(\bar v)$. Then $\nabla \phi(\bar x - \bar y) = \bar v \in \partial f(\bar x)$ and since $f \in \Gamma_0(\bR^n)$ using convex conjugacy we have $\bar x = \nabla f^*(\nabla \phi(\bar x - \bar y))$. Since $f^*\equiv \phi^* \mathbin{\dot{-}} g$ on $\dom \phi^*$, $\dom \phi^* \supseteq \dom (\phi^* \mathbin{\dot{-}} g^{**})$ and $g^{**} \leq g$
we have
\begin{align}
f(x) &= f^{**}(x) = \sup_{v \in \bR^n} \langle v, x \rangle - f^*(v) \notag \\
&\geq \sup_{v \in \dom \phi^*} \langle v, x \rangle - f^*(v) = \sup_{v \in \dom \phi^*} \langle v, x \rangle+g(v)-\phi^*(v)\geq \sup_{v \in \dom \phi^*} \langle v, x\rangle -\phi^*(v) + g^{**}(v) \notag\\
&\geq \sup_{v \in \dom (\phi^* \mathbin{\dot{-}} g^{**})} \langle v, x\rangle -(\phi^* \mathbin{\dot{-}} g^{**})(v) = \sup_{v \in \bR^n} \langle v, x \rangle - (\phi^* \mathbin{\dot{-}} g^{**})(v) =(\phi^* \mathbin{\dot{-}} g^{**})^*(x) \notag \\
&= \sup_{y \in \bR^n} \phi(x-y) \mathbin{\text{\d{\ensuremath{-}}}} g^*_-(y)\label{eq:hirriart_urruty}  \\
&\geq \phi(x-\bar y) \mathbin{\text{\d{\ensuremath{-}}}} g^*_-(\bar y) \label{eq:inequality_f},
\end{align}
where \cref{eq:hirriart_urruty} follows by \cite[Theorem 7.1]{cabot2017envelopes}.
Smoothness of $f^*$ and $g$ on $\intr \dom \phi^*$ yields:
\begin{align*}
\bar x = \nabla f^*(\nabla \phi(\bar x - \bar y)) &= \nabla (\phi^* - g)(\nabla \phi(\bar x - \bar y)) =\bar x - \bar y - \nabla g(\nabla \phi(\bar x - \bar y)),
\end{align*}
and therefore $\nabla g(\nabla \phi(\bar x - \bar y)) = - \bar y$. Recall that $\nabla \phi(\bar x - \bar y) = \bar v \in \partial f(\bar x)$. Then we have thanks to convex conjugacy
\begin{align*}
f(\bar x) &= \langle \nabla \phi(\bar x - \bar y), \bar x \rangle - f^*(\nabla \phi(\bar x - \bar y)) = \langle \nabla \phi(\bar x - \bar y), \bar x \rangle - \phi^*(\nabla \phi(\bar x - \bar y)) + g(\nabla \phi(\bar x - \bar y)) \\
&= \langle \nabla \phi(\bar x - \bar y), \bar x - \bar y \rangle - \phi^*(\nabla \phi(\bar x - \bar y)) - \big( \langle \nabla \phi(\bar x - \bar y), - \bar y \rangle -g(\nabla \phi(\bar x - \bar y)) \big) \\
&= \phi(\bar x - \bar y) - g^*_-(\bar y).
\end{align*}
By combining this result with \cref{eq:inequality_f} we obtain
\[
f(x) \geq f(\bar x) + \phi(x-\bar y) -\phi(\bar x - \bar y) = f(\bar x) + \phi(x-\bar x + \nabla \phi^*(\bar v)) -\phi(\nabla \phi^*(\bar v)). 
\]

``\labelcref{thm:conjugate_duality_astrong:astrong} $\Rightarrow$ \labelcref{thm:conjugate_duality_astrong:bsmooth}'': Let $\bar v \in \intr \dom \phi^* = \ran \nabla \phi$. 
Thanks to the constraint qualification we have $\bar v \in \ran \partial f$. Thus there exists $\bar x \in \dom \partial f$ such that $\bar v \in \partial f(\bar x)$ and the subgradient inequality \cref{eq:a_strongly_mu1} holds true for $(\bar x, \bar v)$. Then the anisotropic subgradient inequality means by definition that $\bar y := \bar x - \nabla \phi^*(\bar v) \in \partial_\Phi f(\bar x)$. Invoking \cref{thm:phi_subgradients} we have $f(\bar x) + f^\Phi(\bar y) = \phi(\bar x - \bar y)$ and $\bar x \in \partial_\Phi f^\Phi(\bar y)$ where the latter means by definition that $\bar y \in \argmax \{\phi(\bar x - \cdot) - f^\Phi\}$. Combined these yield
\begin{align} \label{eq:supremum_y}
\sup_{y \in \bR^n} \phi(\bar x - y) - f^\Phi(y) = \phi(\bar x - \bar y) - f^\Phi(\bar y) = f(\bar x).
\end{align}
Fenchel duality yields $\phi(\bar x - y) = \phi^{**}(\bar x - y) = \sup_{v \in \bR^n} \langle \bar x - y, v \rangle - \phi^*(v)$. Define $q(y) := \sup_{v \in \bR^n} \xi(y, v)$ for $\xi(y, v) := \langle \bar x - y, v \rangle - \phi^*(v) - f^\Phi(y)$. Then we can rewrite the supremum in \cref{eq:supremum_y} in terms of the joint supremum $f(\bar x) = \sup_{y \in \bR^n} q(y)= q(\bar y)$ and in particular $\bar y \in \argmax \,q$. We have $\bar y = \bar x - \nabla \phi^*(\bar v)$ or equivalently $\bar v = \nabla \phi(\bar x - \bar y)$ which via convex conjugacy implies that $\bar v \in \argmax \{\langle \bar x - \bar y, \cdot \rangle - \phi^*\}$. This shows that $\bar v \in \argmax \xi(\bar y, \cdot)$. Define $p(v) = \sup_{y \in \bR^n} \xi(y, v)$. Then \cite[Proposition 1.35]{RoWe98} yields that $(\bar y, \bar v) \in \argmax \xi$ as well as $\bar v \in \argmax \,p$. Overall this yields
\begin{align*}
f(\bar x) &= \sup_{y \in \bR^n} q(y) = \sup_{v\in \bR^n} p(v) = p(\bar v) \\
&= \langle \bar x, \bar v \rangle - \phi^*(\bar v) + \sup_{y \in \bR^n} \langle y, -\bar v \rangle - f^\Phi(y) = \langle \bar x, \bar v \rangle - \phi^*(\bar v) + (h^\Phi)^*(-\bar v),
\end{align*}
which combined with the fact that $\bar v \in \partial f(\bar x)$ results in
$$
f^*(\bar v) = \langle \bar x, \bar v \rangle - f(\bar x) = \phi^*(\bar v) - (f^\Phi)^*_-(\bar v) \in \bR.
$$
Since $\bar v \in \intr \dom \phi^*$ was arbitrary we have $\phi^* \mathbin{\dot{-}} f^* \equiv (f^\Phi)^*_-$ on $\intr \dom \phi^*$ and thus $\phi^* \mathbin{\dot{-}} f^*$ is convex on $\intr \dom \phi^*$.

Let $x \in \dom \phi^*$ and $x_0 \in \intr \dom \phi^*$. Then we have 
\begin{align*}
f^*(x) &\leq \liminf_{\tau \nearrow 1} f^*((1-\tau) x_0 + \tau x) \\
&\leq \lim_{\tau \nearrow 1} f^*(x_0) + \langle \nabla f^*(x_0), (1-\tau) x_0 + \tau x - x_0 \rangle + D_{\phi^*}((1-\tau) x_0 + \tau x, x_0) \\
&= f^*(x_0) + \langle \nabla f^*(x_0), x - x_0 \rangle + D_{\phi^*}(x, x_0),
\end{align*}
where the first inequality holds since $f^*$ is lsc, the second since in view of the previous part of the proof, $\phi^*-f^*$ is convex and smooth on $\intr \dom \phi^*$, and $(1-\tau) x_0 + \tau x \in \intr \dom \phi^*$ for $\tau \in [0, 1)$ by the line segment principle \cite[Theorem 2.33]{RoWe98} and the last equality holds due to \cite[Theorem 2.35]{RoWe98} and since $\phi^* \in \Gamma_0(\bR^n)$.
Since $x \in \dom \phi^*$ and $x_0 \in \intr \dom \phi^*$ the right-hand side of the inequality is finite and thus $x \in \dom f^*$.
\ifx\ifsvjour\true
\qed
\fi
\end{proof}

%% file: missing_proofs_phi_dca.tex
\subsection{Proof of \cref{thm:phi_subgradient_gradient_correspondence}} \label{sec:thm_phi_subgradient_gradient_correspondence}
\begin{proof}
Without loss of generality assume that $L=1=\lambda$ by redefining $\phi$ as $\lambda \star \phi$ and choose $\Phi(x,y):=-\phi(x-y)$.

``\labelcref{thm:phi_subgradient_gradient_correspondence:asmooth} $\Rightarrow$ \labelcref{thm:phi_subgradient_gradient_correspondence:phi_subgrad}'': Let $\bar x \in \bR^n$. Then we have that for any $x$
$$
-h(x) \geq -\phi(x - \bar x +\nabla \phi^*(\nabla h(\bar x))) + \phi(\nabla \phi^*(\nabla h(\bar x))) -h(\bar x),
$$
which means by definition of the $\Phi$-subdifferential that $\bar x - \nabla \phi^*(\nabla h(\bar x)) \in \partial_\Phi (-h)(\bar x)$.
Now take $\bar y \in \partial_\Phi (-h)(\bar x)$. This means in view of \cref{thm:phi_subgradients} that $\bar x \in \argmin_{x \in \bR^n} \{\phi(x - \bar y) - h(x)\}$ and thus by Fermat's rule
$0 = \nabla \phi(\bar x - \bar y) - \nabla h(\bar x)$. This means that $\bar y = \bar x - \nabla \phi^*(\nabla h(\bar x))$.

``\labelcref{thm:phi_subgradient_gradient_correspondence:phi_subgrad} $\Rightarrow$ \labelcref{thm:phi_subgradient_gradient_correspondence:asmooth}'': Assume that $\partial_\Phi (-h)= \id - \nabla \phi^* \circ \nabla h$. Let $\bar x \in \bR^n$. Then $\bar x - \nabla \phi^*(\nabla h(\bar x)) \in \partial_\Phi (-h)(\bar x)$. By definition of the $\Phi$-subdifferential this means the anisotropic descent inequality holds at $\bar x$.
\end{proof}

\subsection{Proof of \cref{thm:invariance_double_min}} \label{sec:thm_invariance_double_min}
\begin{proof}
In view of \cref{thm:moreau_decomposition:smooth} the following gradient formula for the anisotropic Moreau envelope holds true:
\begin{align}
\nabla (g \infconv \lambda \star \phi_-) = \nabla \phi_- \circ \lambda^{-1}(\id - \aprox[\lambda]{\phi_-}{g}),
\end{align}
and in particular $\ran \nabla(g \infconv \lambda \star \phi_-) = \ran \nabla \phi_- \circ \lambda^{-1}(\id - \aprox[\lambda]{\phi_-}{g}) \subseteq \ran \nabla \phi_-$. Thus we can invoke \cref{thm:phi_convex_asmooth} and obtain that $g \infconv \lambda \star \phi_-$ has the anisotropic smoothness property relative $\phi_-$ with constant $\lambda^{-1}$ and in particular we have for the forward-step applied to $g \infconv \lambda \star \phi_-$ at any $y \in \bR^n$:
\begin{align*}
y - \lambda \nabla \phi_-^*(\nabla (g \infconv \lambda \star \phi_-)(y))&=y - \lambda \nabla \phi_-^*(\nabla \phi_-(\lambda^{-1}(y - \aprox[\lambda]{\phi_-}{g}(y))))\\
&= y - (y - \aprox[\lambda]{\phi_-}{g}(y)) =\aprox{\lambda \star \phi_-}{g}(y).
\end{align*}
Define $\Phi(x,y):= -\lambda \star \phi(x-y)$. Let $y \in \aprox[\lambda]{\phi}{f \infdeconv \lambda \star \phi}(x) =\argmax_{y \in \bR^n} \;\{\Phi(x,y) -(-f)^\Phi(y)\}$ be the backward-step applied to $f \infdeconv \lambda \star \phi=(-f)^\Phi$ at $x$.
By definition of the $\Phi$-subdifferential this is equivalent to $x \in \partial_\Phi (-f)^\Phi(y)$. Since $f$ is anisotropically smooth relative to $\phi$ with constant $L$, by \cref{thm:episcaling_asmoothness}, $f$ is anisotropically smooth relative to $\phi$ with constant $\lambda^{-1}\geq L$ as well. By \cref{thm:phi_convex_asmooth} this means that $-f$ is $\Phi$-convex and thus by \cref{thm:phi_subgradients} $x \in \partial_\Phi (-f)^\Phi(y)$ is equivalent to $y \in \partial_\Phi (-f)(x)$. Invoking \cref{thm:phi_subgradient_gradient_correspondence} this implies that $\aprox[\lambda]{\phi}{f \infdeconv \lambda \star \phi}(x) =\partial_\Phi (-f)(x)= (\id -\lambda \nabla \phi^* \circ \nabla f)(x) = x - \lambda \nabla \phi^*(\nabla f(x))$.

If, in addition, $f$ is convex, thanks to \cite[Theorem 2.2]{hiriart1986general} the deconvolution $f \infdeconv \lambda \star \phi = (f^* \mathbin{\dot{-}} \lambda \phi^*)^* \in \Gamma_0(\bR^n)$. Since $\ran \nabla f \subseteq \ran \nabla \phi$ we have that $\relint \dom f^* \subseteq \dom \partial f^* \subseteq \intr\dom \phi^*$. Since $\dom( f^* \mathbin{\dot{-}} \lambda \phi^*) = \dom f^*$ we have $\dom( f^* \mathbin{\dot{-}} \lambda \phi^*) \cap \intr \dom \phi^* \neq \emptyset$. Since $(f \infdeconv \lambda \star \phi)^* = (f^* \mathbin{\dot{-}} \lambda \phi^*)^{**} \leq f^* \mathbin{\dot{-}}\lambda \phi^*$ we have $\dom (f \infdeconv\lambda \star \phi)^* \supseteq \dom( f^* \mathbin{\dot{-}}\lambda \phi^*)$. This implies that $\dom (f \infdeconv \lambda \star \phi)^*\cap \intr \dom\phi^*\neq \emptyset$.
\end{proof}

\subsection{Proof of \cref{thm:strongly_convex}} \label{sec:thm_strongly_convex}
\begin{proof}
Thanks to \cref{thm:conjugate_duality_astrong} anisotropic strong convexity of $g \infconv L^{-1} \star \phi_-$ with constant $\sigma$ is implied by Bregman smoothness of $(g \infconv L^{-1} \star \phi_-)^* = g^* + L^{-1}\phi_-^*$ relative to $\sigma^{-1} \phi_-^*$, i.e., convexity of
$
\sigma^{-1} \phi_-^* \mathbin{\dot{-}} (g \infconv L^{-1} \star \phi_-)^* = (\sigma^{-1}-L) \phi_-^*  \mathbin{\dot{-}}  g^*,
$
on $\intr \dom \phi_-^*$ and $\dom \partial(g \infconv L^{-1} \star \phi_-)^* \supseteq \intr \dom \phi_-^*$.
Since $g$ is anisotropically strongly convex relative to $\phi_-$ we have $\dom \partial g^* \supseteq \intr \dom\phi_-^*$. In light of \cref{thm:conjugate_duality_astrong}, $\mu^{-1}\phi^*_- \mathbin{\dot{-}} g^*$ is convex on $\intr \dom \phi_-^*$. Thus $(\sigma^{-1}-L^{-1}) \phi_-^* \mathbin{\dot{-}} g^*$ is convex on $\intr \dom \phi_-^*$ if $\sigma^{-1}\geq L^{-1}+\mu^{-1}$, i.e., $\sigma \leq 1/(L^{-1}+\mu^{-1})$.
We also have thanks to the constraint qualification and the sum-rule and the inclusion $\dom \partial g^* \supseteq \intr \dom\phi_-^*$ that $\dom \partial(g \infconv L^{-1} \star \phi_-)^* =\dom (\partial g^* + \nabla \phi_-^*) = \dom \partial g^*\cap \dom \nabla \phi_-^* = \dom \partial g^* \supseteq \intr \dom\phi_-^*$. This proves the claimed result.
\end{proof}

\subsection{Proof of \cref{thm:linear_convergence_dual}} \label{sec:thm_linear_convergence_dual}
\begin{proof}
By \cref{thm:strongly_convex} $g \infconv \lambda \star \phi_-$ is anisotropically smooth with constant $L$ and anisotropically strongly convex with constant $1/(L^{-1}+\mu^{-1})$. By \cref{thm:invariance_double_min} $f \infdeconv \lambda \star \phi \in \Gamma_0(\bR^n)$ with $\dom(f \infdeconv \lambda \star \phi) \cap \intr \dom \phi^* \neq \emptyset$ and therefore we can invoke \cref{thm:strong_implies_pg_dominance} to show that $G$ has a unique minimizer $y^\star$ and $G=(g \infconv \lambda \star \phi_-) + (f \infdeconv \lambda \star \phi)$ satisfies the anisotropic proximal gradient dominance condition relative to $\phi_-$ with constant $1/(L^{-1}+\mu^{-1})$ and parameter $L$. 
Thanks to \cref{thm:invariance_double_min} we have for every $k \in \bN_0$ that $y^{k+1}=x^{k+1} - \lambda \nabla \phi^*(\nabla f(x^{k+1})) = \aprox[\lambda]{\phi}{f \infdeconv \lambda \star \phi}(x^{k+1})$ and $x^{k+1}=\aprox[\lambda]{\phi_-}{g}(y^k) = y^k - \nabla \phi_-^*(\nabla (g \infconv \phi_-))(y^k)$. 
In combination we have $y^{k+1}=\aprox[\lambda]{\phi}{f \infdeconv \lambda \star \phi}(y^k - \nabla \phi_-^*(\nabla (g \infconv \phi_-))(y^k))$. This is identical to \cref{alg:aproxgrad} with step-size $\lambda$, reference function $\phi_-$ and initial iterate $y^0$ applied to $G=(g\infconv \lambda \star \phi_-) + (f \infdeconv \lambda \star \phi)$.
Thus we can apply \cref{thm:linear_convergence} to the sequence $\{y^k\}_{k =0}^\infty$ to show that $\{G(y^k) \}_{k =0}^\infty$ converges $R$-linearly to $G(y^\star)$ with the claimed constant.
\end{proof}

\subsection{Proof of \cref{thm:fbe_dc}} \label{sec:thm_fbe_dc}
\begin{proof}
Let $k\in \bN_0$.
In view of \cref{thm:invariance_double_min} we have $y^k=x^k - \lambda\nabla\phi^*(\nabla f(x^k)) = \argmin_{y \in \bR^n} \{\lambda \star \phi(x^k - y) + (f \infdeconv \lambda \star \phi)(y)\}$. Since $f$ is anisotropically smooth relative to $\phi$ with constant $L$, by \cref{thm:episcaling_asmoothness}, $f$ is anisotropically smooth relative to $\phi$ with constant $\lambda^{-1}\geq L$ as well. 
Owing to \cref{thm:phi_convex_asmooth} $-f$ is $\Phi$-convex relative to $\Phi(x,y)=-\lambda \star \phi(x-y)$ and thus by \cref{thm:phi_envelope_equivalence} $f=-(-f)^{\Phi\Phi}=\inf_{y \in \bR^n} \lambda \star \phi(x^k - y) + (f \infdeconv \lambda \star \phi)(y)$ implying that
\begin{align} \label{eq:fbe_dc_1}
(f \infdeconv \lambda \star \phi)(y^k) + \lambda \star \phi(x^k - y^k)=f(x^k).
\end{align}
By definition of the forward-backward envelope since $x^k - y^k=\lambda \nabla \phi^*(\nabla f(x^k))$ we have that $F_\lambda(x^k)=(g\infconv \lambda \star \phi_-)(y^k) + f(x^k) - \lambda \star \phi(x^k - y^k)$. In combination with \cref{eq:fbe_dc_1} we obtain that
\[
F_\lambda(x^k)= (g\infconv \lambda \star \phi_-)(y^k) + (f \infdeconv \lambda \star \phi)(y^k) = G(y^k).\qedhere
\]
\end{proof}

\subsection{Proof of \cref{thm:linear_convergence_primal}} \label{sec:thm_linear_convergence_primal}
\begin{proof}
Let $k\in \bN_0$.
Since $f$ is anisotropically smooth the anisotropic descent inequality holds true which after adding $g(x^{k+1})$ to both sides of the inequality amounts to:
\begin{align} \label{eq:ineq_fbe_1}
F(x^{k+1}) &\leq f(x^k) + \lambda \star \phi(x^{k+1} - y^k) - \lambda \star \phi(\lambda\nabla \phi^*(\nabla f(x^k))) + g(x^{k+1})\notag \\
&= F_\lambda(x^k),
\end{align}
where the last equality follows by definition of $F_\lambda$ and $x^{k+1}$.
Define $\xi(x, x^k):=f(x^k) + \lambda \star \phi(x - x^k + \lambda \nabla \phi^*(\nabla f(x^k))) - \lambda \star \phi(\lambda\nabla \phi^*(\nabla f(x^k))) + g(x^{k+1})$. Then by definition $F_\lambda(x^k) = \inf_{x \in \bR^n} \xi(x, x^k)$ and $\xi(x^k, x^k) = F(x^k)$. Therefore we also have
\begin{align} \label{eq:ineq_fbe_2}
F_\lambda(x^k) \leq F(x^k).
\end{align}
Invoking \cref{thm:linear_convergence_dual} $G$ has a unique minimizer $y^\star$ and we have after $K$ iterations:
$$
G(y^{k}) - G(y^\star) \leq (1 - \tfrac\mu{\mu + L} )^{k} (G(y^0) - G(y^\star) ).
$$
By double-min duality \cref{eq:double_min} we have $G(y^\star) = \inf F$. Invoking \cref{thm:fbe_dc} we have that $F_\lambda(x^k)=G(y^k)$ for all $k \in \bN_0$.
Thanks to the identities \cref{eq:ineq_fbe_1} and \cref{eq:ineq_fbe_2} we have $F(x^{k+1}) \leq F_\lambda(x^{k})=G(y^{k})$ and $F(x^0) \geq F_\lambda(x^0)=G(y^{0})$ and thus we obtain the claimed result.
\qedhere
\end{proof}